\documentclass[12pt]{amsart}
\usepackage{amsmath, mathabx}
\usepackage{amsxtra}
\usepackage{amscd}
\usepackage{amsthm}
\usepackage{amsfonts}
\usepackage{amssymb}
\usepackage{mathrsfs}
\usepackage{mathbbol}

\usepackage {pstricks}
\usepackage {tikz}
\usepackage{pstricks,pst-node}
\usepackage[all]{xy}
\usepackage{mathbbol}

\usepackage[scr=esstix,cal=boondox]{mathalfa} 

\usepackage{verbatim}
\usepackage{enumitem}
\usepackage{multicol}

\usepackage{dashundergaps}
\usepackage{tikz-cd}

\usepackage{graphicx}

\usepackage{epigraph}
\usepackage{physics}

\usepackage{lmodern}

\usepackage{stmaryrd}

\newtheorem{theorem}{Theorem}[section]

\newtheorem{lemma}[theorem]{Lemma}

\newtheorem{corollary}[theorem]{Corollary}

\theoremstyle{definition}
\newtheorem{definition}[theorem]{Definition}
\newtheorem{example}[theorem]{Example}

\newtheorem{notation}[theorem]{Notation}

\theoremstyle{remark}
\newtheorem{remark}[theorem]{Remark}

\numberwithin{equation}{section}
\numberwithin{figure}{section}

\DeclareRobustCommand{\loongrightarrow}{%
  \DOTSB\relbar\joinrel\relbar\joinrel\relbar\joinrel\rightarrow
}

\newcommand{\mc}{\mathcal}

\newcommand{\be}{\begin{equation}}
\newcommand{\ee}{\end{equation}}

\newcommand{\C}{{\mathbb C}}
\newcommand{\R}{{\mathbb R}}
\newcommand{\Z}{{\mathbb Z}}

\newcommand{\K}{{\mathbb K}}

\newcommand{\BE}{{\mathbb E}}

\newcommand{\CA}{{\mathcal A}}
\newcommand{\CB}{{\mathcal B}}
\newcommand{\CC}{{\mathcal C}}

\newcommand{\CL}{{\mathcal L}}

\newcommand{\CT}{{\mathcal T}}

\newcommand{\CJ}{\mc J}

\newcommand{\SJ}{{\mathscr J}}

\newcommand{\SD}{{\mathscr D}}

\newcommand{\SR}{{\mathscr R}}

\newcommand{\mf}{\mathfrak}

\newcommand{\fg}{{\mf g}}
\newcommand{\fh}{{\mf h}}
\newcommand{\fb}{{\mf b}}
\newcommand{\fn}{{\mf n}}

\newcommand{\fsl}{{\mathfrak {sl}}}
\newcommand{\fsp}{{\mathfrak {sp}}}
\newcommand{\fso}{{\mathfrak {so}}}

\newcommand{\gl}{{\mathfrak {gl}}}
\newcommand{\osp}{{\mathfrak {osp}}}

\newcommand{\id}{{\rm{id}}}

\newcommand{\U}{{\rm{U}}}
\newcommand{\End}{{\rm{End}}}
\newcommand{\Hom}{{\rm{Hom}}}

\newcommand{\GL}{{\rm{GL}}}

\newcommand{\Sym}{{\rm{Sym}}}

\newcommand{\wt}{\widetilde}

\newcommand{\ol}{\overline}

\newcommand{\ot}{\otimes}
\newcommand{\OSp}{{\rm OSp}}

\newcommand{\beq}{\begin{eqnarray}}
\newcommand{\eeq}{\end{eqnarray}}

\newcommand{\baln}{\begin{aligned}}
\newcommand{\ealn}{\end{aligned}}

\newcommand{\lra}{\longrightarrow}

\newcommand{\wh}{\widehat}

\newcommand{\Vect}{{{\mathfrak V}(\Gamma, \omega)}}

\newcommand{\brhd}{\blacktriangleright}
\newcommand{\blhd}{\blacktriangleleft}

\newcommand{\typeA}
{\begin{picture}(50, 20)(0, 0)
\put(5, 10){\circle{10}}
\put(10, 10){\line(1, 0){10}}
\put(20, 9){...}
\put(30, 10){\line(1, 0){10}}
\put(45, 10){\circle{10}}
\end{picture}}

\newcommand{\typeGA}
{\begin{picture}(50, 20)(0, 0)
\put(5, 10){\circle{10}}\put(5, 10){\circle{5}}
\put(10, 10){\line(1, 0){10}}
\put(20, 9){...}
\put(30, 10){\line(1, 0){10}}
\put(45, 10){\circle{10}}\put(45, 10){\circle{5}}
\end{picture}}

\begin{document}

\normalfont

\title[Quantum groups of Lie colour algebras]
{Quantum groups of Lie colour algebras \\ fulfilling the Cartan--Weyl paradigm}

\author{R.B. Zhang}
\address{School of Mathematics and Statistics,
The University of Sydney, Sydney, N.S.W. 2006, Australia}
\email{ruibin.zhang@sydney.edu.au}
\date{\today}

\begin{abstract}
Let $\Gamma$ be an additive abelian group with a commutative factor $\omega$. 
We describe the simple and untwisted affine Lie colour algebras which admit a 
Cartan-Weyl description.
For both classes, we construct the quantised universal enveloping colour algebras,
 and establish their quasi-triangular Hopf colour algebraic structure.  
These ``colour quantum groups'' generalise 
the Drinfeld–Jimbo quantum groups, 
which are recovered when $\Gamma$ is trivial.
They provide an algebraic foundation for investigations in knot theory 
and soluble models in statistical mechanics.
\end{abstract}

\maketitle

\tableofcontents




\section{Introduction}
\noindent{1.1.}
Lie colour algebras and superalgebras, that is, Lie $(\Gamma,\omega)$-algebras for general abelian groups $\Gamma$ 
and commutative factors $\omega$, have attracted considerable interest in recent years, 
in particular because of their applications in mathematical physics \cite{RW, RW-2, Sch79, Sch83, SchZ, Z25}. 
They provide a natural algebraic framework for generalised statistics and supersymmetries 
beyond the conventional $\mathbb Z_2$-graded setting. In particular, Lie colour (super)algebras 
have been used in the study of parastatistics \cite{BFRT, SV18, SV23, T, Tf1, Tf2, YJ} (see also \cite{Z24}), 
including $\mathbb Z_2\times\mathbb Z_2$-graded and more general $\mathbb Z_2^n$-graded 
paraparticles \cite{Tf1, Tf2, Tf3,  BFRT}. Colour supersymmetries have also been developed in 
quantum mechanics \cite{AD, AIKT, AKT1, AKT2, BFRT, BD, Tf4} 
and quantum field theory \cite{JYW, LR, V, YJ, Bruce, DA}. 
Other physical applications include $\mathbb Z_2\times\mathbb Z_2$-graded extensions of the Liouville, 
sine-Gordon and sinh-Gordon theories \cite{AIKT}, KdV and mKdV hierarchies obtained 
by Drinfeld--Sokolov reduction for affine Lie colour superalgebras \cite{AFI}, and $\mathbb Z_2\times\mathbb Z_2$-graded symmetries of the Lévy--Leblond equations \cite{AKTT} (see also \cite{Rm25} and references therein). 
The physical applications of Lie colour (super)algebras continue to develop actively; 
see, for example, \cite{AIT, AIKT, AlI, KT26, SV25a, SV25b}.

There have also been substantial developments in the mathematical theory of Lie colour (super)algebras. 
Various classes of finite-dimensional and affine Lie colour (super)algebras have been studied, 
including examples associated with $\Gamma=\mathbb Z_2\times\mathbb Z_2$ and 
related commutative factors \cite{AI, AA, ISV, KT21, Rm25, SV23, SV25b, AIS, AS}.  
A cohomology theory for Lie colour (super)algebras was introduced in \cite{SchZ}, 
and Casimir operators for Lie colour algebras were constructed in \cite{AFSV}; see also \cite[\S5.4.1]{Z25}. 
Representation-theoretic questions have been investigated for generic Lie colour (super)algebras from 
a ring-theoretic point of view \cite{F}, as well as for the colour Lie algebra $\fso_3$ \cite{CSO}. 
More recently, the representation and invariant theory of the general linear Lie colour superalgebra 
was systematically developed in \cite{Z25}. Lie colour (super)algebras have also played a role 
in the study of graded associative and Hopf algebras and group gradings of Lie algebras \cite{BFM, BFM-2, Mo, BK}.

\medskip
\noindent{1.2.}
The developments described above naturally lead to the question of whether Lie colour algebras 
admit quantum deformations analogous to the Drinfeld--Jimbo quantisations of 
ordinary Lie algebras and Lie superalgebras. 
Our original goal is to construct such quantised universal enveloping colour algebras. 
In the course of pursuing this problem, however, we encountered a structural issue 
that must be resolved before the quantisation problem can be addressed in the usual Drinfeld--Jimbo framework.

The Drinfeld--Jimbo construction \cite{D1, D2, J86, J} relies on the Cartan--Weyl description of Lie algebras, and   
it is the same structure that underlies the construction of quantum supergroups \cite{BGZ, Y91, Y94, ZGB}.  
However, a (suitably generalised) Cartan--Weyl description can not be taken for granted for arbitrary simple Lie colour algebras and superalgebras, 
even for most cases of the orthogonal Lie colour algebras over $\C$. 
It can happen that the grading rules out the existence of a suitable Cartan subalgebra which is homogeneous of degree $0$. In this case, root space decomposition is unavailable, and consequently there is no direct Drinfeld--Jimbo construction based on root systems. 
This leads us to the more fundamental structural problem of determining which Lie colour algebras fulfil the Cartan--Weyl paradigm (see Section \ref{sect:CWP-0} for explanation).

We show that the Cartan--Weyl paradigm holds for important classes of Lie colour algebras and superalgebras, 
including the general and special linear Lie colour superalgebras \cite{Z25},  symplectic Lie colour algebras,  and some special cases of the orthogonal Lie colour algebras and orthosymplectic Lie colour superalgebras. 
In the orthosymplectic case (containing the orthogonal case), we obtain the necessary and sufficient condition for the paradigm to hold.  

With the Cartan-Weyl structure of these important classes of Lie colour (super)algebras understood, 
it becomes clear that the construction pioneered by Kac \cite{Kv} and Moody \cite{Mr} (also see \cite[\S1]{Kv-2} and \cite{Kv-1}) adapts to the present colour setting, enabling us to construct all Lie colour (super)algebras which fulfil the Cartan--Weyl paradigm. Now each such algebra is determined by a Cartan matrix together with a sequence 
of elements of the grading group $\Gamma$. 
Such Lie colour algebras associated with reduced Cartan matrices of  finite and untwisted affine types are the ones
which we quantise in this paper. 

Our earlier work \cite{Z25} provides part of the background for this programme. There, the representation and invariant theory of the general linear Lie colour superalgebra and the associated colour analogues of algebraic groups were systematically developed. That work also led to an interesting form of noncommutative geometry \cite[\S6.3]{Z25}, including a realisation of simple tensor modules in terms of ``holomorphic'' sections of line bundles on noncommutative flag varieties \cite{GZ, Z04}. The present paper continues this programme by establishing the Cartan--Weyl structures for a broader class of Lie colour algebras, and constructing the associated quantum groups.

\medskip
\noindent{1.3.}
We emphasise a point that is important for the physical motivation of colour quantum groups. 
In quantum systems with a Lie or Hopf colour (super)algebra symmetry, 
the grading generally reflects the statistics of the states, such as the usual Bose-Fermi statistics, 
parastatistics \cite{G} and its generalisations. In such situations, the grading carries physical information and 
therefore should not simply be eliminated by a Klein-type transformation, 
such as cocycle twisting \cite{Sch79} (see also \cite{BFM-2, MB}), or bosonisation \cite{Ma}. 
For this reason, we work directly with Hopf $(\Gamma,\omega)$-algebras rather than their bosonisations 
in this paper.

There is also a mathematical reason for retaining the grading. Eliminating the grading can obscure
or even eliminate structural information that distinguishes a colour algebra from its ungraded counterpart. 
For example, passing from the orthosymplectic Lie colour superalgebra $\osp(V;\kappa)$ 
to the corresponding ordinary orthosymplectic Lie superalgebra removes precisely 
the grading information responsible for the possible failure of the Cartan--Weyl paradigm. 
Thus, a bosonised version may not fully capture the structural properties of the original colour algebra.

From the point of view of deformation theory \cite{Gm, GS}, it is therefore natural to 
quantise the universal enveloping colour algebras as Hopf $(\Gamma, \omega)$-algebras. 
Working in this setting preserves both the grading and the commutative factor $\omega$ 
as intrinsic parts of the structure in the quantisation process.

\medskip
\noindent{1.4.}
We now describe the main results. But first, a word on the terminology. 
We make a distinction between Lie colour algebras and Lie colour superalgebras,
depending on whether there exists any non-zero homogeneous component with degree 
$\alpha$ such that $\omega(\alpha, \alpha)=-1$. 
Further explanation can be found at the end of Section \ref{sect:CWP} (see \cite[\S2.3]{Z25} for more details). 

\smallskip

The present paper contains the following main results. 

\smallskip
($i$). {\em Lie and affine colour algebras fulfilling the Cartan--Weyl paradigm}.

We first examine the Cartan--Weyl paradigm in the setting of Lie colour (super)algebras (see Section \ref{sect:CWP} 
for details). We show that the paradigm holds for important families of Lie colour (super)algebras, 
including the general and special linear Lie colour superalgebras \cite{Z25} and the symplectic Lie colour algebras, 
whereas it fails for other families.

Our analysis leads, in particular, to the following criterion. 
The orthosymplectic Lie colour superalgebra $\osp(V;\kappa)$ fulfils the Cartan--Weyl paradigm
 if and only if there exists at most one $2$-torsion element $\alpha\in\Gamma$, that is, $\alpha+\alpha=0$, 
such that the homogeneous component $V_\alpha\subset V$ has odd dimension. Thus in most cases,  $\osp(V;\kappa)$ does not fulfil the Cartan--Weyl paradigm. This result identifies explicitly an obstruction which has no 
direct counterpart in the ordinary Lie algebraic setting.

The analysis of the classical and affine Lie colour superalgebras shows that
the Kac-Moody methodology \cite{Kv, Mr}  \cite[\S1]{Kv-2} (also see \cite{Kv-1} for Lie superalgebras) 
naturally adapts to the colour setting for constructing all Lie colour superalgebras fulfilling the Cartan--Weyl paradigm.
We construct such simple Lie colour algebras and their untwisted affine Lie colour algebras in Section \ref{sect:present}.  
Each algebra in this class is determined by a reduced Cartan matrix of finite or affine type together 
with a sequence of elements of $\Gamma$. 
We give the explicit presentations in terms of generators and defining relations for these algebras.  

Conceptual implications of the failure of the Cartan--Weyl paradigm for certain Lie colour (super)algebras 
are discussed in Remarks \ref{rmk:color-y-n}, \ref{rmk:change}, and \ref{rmk:intricacy}.

\smallskip
($ii$). {\em Quantised universal enveloping colour algebras}.

The Cartan--Weyl structure established in the first part provides the basis for the construction 
of Drinfeld--Jimbo type quantisation. We construct quantised universal enveloping colour algebras 
for the finite-dimensional Lie colour algebras and their affine analogues described in Section \ref{sect:present}. 
These algebras, denoted by $\U_{q,\Xi}(A)$ and referred to as \emph{colour quantum groups}, generalise the usual quantum groups \cite{D2, D3, J} and quantum supergroups \cite{BGZ, Y91, Y94, ZGB, Z98}. 
When the grading group is trivial, they recover the corresponding ordinary quantum groups. 
By construction, the colour quantum groups inherit the Cartan--Weyl structure in 
the same way that ordinary quantum groups and quantum supergroups do.

We further establish the quasi-triangular Hopf colour algebra structure of the colour quantum groups. 
This is obtained, 
in Section \ref{sect:QT}, through the study of quantum doubles of Hopf $(\Gamma,\omega)$-algebras developed 
in Section \ref{sect:double-gen}; see Theorem \ref{thm:QD}. 
This provides the appropriate framework for the construction of universal $R$-matrices and for applications of colour quantum groups to areas in which quasi-triangular structures play a fundamental role.

As we will see in Section \ref{sect:ACQG}, 
our construction of the colour quantum groups $\U_{q,\Xi}(A)$ 
is not by merely imposing relations on generators. 
Instead, we first construct a Hopf colour algebra $\wt\U_{q, \Xi}(A)$ which may be regarded as a universal Hopf colour algebra associated with $A$ and $\Xi$, and then obtain $\U_{q,\Xi}(A)$ as a canonical quotient by a Hopf ideal $\CJ$. Here  $\CJ$ is generated by the left and right radicals of the Hopf pairing between the two quantised colour Borel subalgebras of $\wt\U_{q, \Xi}(A)$, and the colour quantum Serre elements (see Section \ref{sect:U+U-}) are precisely  the  generators of the left and right radicals.  
Therefore, the colour quantum Serre relations are not imposed ad hoc but arise canonically from the Hopf pairing.

\smallskip
($iii$). {\em Quasi-triangular Hopf $(\Gamma,\omega)$-algebras}.

For completeness, we collect in Appendix \ref{sect:double}  basic results 
on quasi-triangular Hopf $(\Gamma,\omega)$-algebras needed in the paper. 
In particular, the quantum double construction for Hopf $(\Gamma,\omega)$-algebras is developed 
in some detail in Section \ref{sect:double-gen}. We include these results to ensure that the construction 
of the colour quantum groups is self-contained and accessible to readers 
whose main area may not be Hopf algebra theory.

\medskip
\noindent{1.5.}
The resulting colour quantum groups provide a natural setting in which the structural theory of Lie colour algebras, 
quantum-group methods, and applications to mathematical physics can be brought together. 
Since they are quasi-triangular Hopf colour algebras, they possess the structures that make 
ordinary quantum groups and quantum supergroups useful in the study of integrable models 
in statistical mechanics \cite{J86, BGZ, ZGB91-b}, and low-dimensional topology \cite{RT, LGZ, ZGB91, Z92-b, Z95}. 
Also see \cite{AK} for Vassiliev invariants arising from Lie colour algebras.

Colour quantum groups are also naturally connected with noncommutative geometry \cite{M, La}. 
Lie colour (super)algebras themselves are closely related to the new geometries discussed 
in \cite[\S6.3]{Z25}, including noncommutative para-manifolds \cite{Z24} for $\Gamma=\mathbb Z^n$ 
and particular commutative factors $\omega$, and $\mathbb Z_2^n$-graded supermanifolds \cite{CGP}. 
The noncommutative geometries associated with colour quantum groups may be viewed as quantum deformations 
of these structures. Developing these connections further is a natural direction for future work.

\medskip
\noindent{1.6.}
We point out a connection with the Nichols algebra approach to quantum groups and 
quantum supergroups \cite{AO17, Hec09, Ros98, Lus10}. In that approach, 
bosonised Nichols algebras \cite{Ma} provide quantum Borel subalgebras, 
while the corresponding full quantum groups can be constructed using quantum doubles \cite{AHS10, HS20, Nic78, Wor89}. The colour quantum Borel algebras constructed in the present paper are naturally related to this framework: 
after bosonisation, one expects their braided parts to be described by Nichols algebras. 
Establishing the precise relationship, however, requires keeping track of the grading and commutative factor 
through the bosonisation process and is more subtle than simply passing to an ungraded algebra \cite{XZ, Z92, Z97}. 
Some indications of this relationship are given in Section \ref{sect:complete}.

After the first version of this paper was posted on the arXiv, we were informed of the work \cite{AAB} 
on pointed Hopf colour algebras. In \S3 of that paper, quantum doubles of certain bosonised Nichols algebras 
are constructed, together with quotient Hopf algebras that appear to be closely related to bosonisations of 
the quantised universal enveloping colour algebras $\U_{q,\Xi}(A)$ considered here. The aims of \cite{AAB}, 
however, are different from those of the present paper. In particular, that work does not introduce 
quantised universal enveloping colour algebras. We discuss the relationship between the two constructions further in Remarks \ref{rmk:AAB-1}. 

The relation to the Nichols algebra approach should not, however, be interpreted as making the colour structure inessential. Bosonisation replaces a colour Hopf algebra by an ordinary Hopf algebra together with additional group-like data, 
whereas in the present paper the $\Gamma$-grading and the commutative factor $\omega$ remain intrinsic 
throughout the construction. This distinction is important both mathematically and for physical applications 
of colour quantum groups as discussed in Section 1.3.

\medskip
\noindent{1.7.}
The constructions developed in this paper therefore provide a framework in which Lie colour algebras 
satisfying the Cartan--Weyl paradigm can be treated systematically, together with their affine extensions and Drinfeld--Jimbo type quantum deformations. 
We hope that this framework will contribute to the further development of Lie colour theory 
and its applications in mathematical physics. 

To make the paper accessible to a broad audience,  we explain notions potentially unfamiliar 
to readers who are not specialised in Lie algebras or Hopf algebras, 
prove results by elementary means whenever possible and present the pertinent details.

\section{Background material}
We present some background material pertinent for the study of classical 
Lie $(\Gamma, \omega)$-algebras. More details can be found in, 
e.g., \cite{Sch79, Z25}.  
\subsection{The category of $\Gamma$-graded vector spaces}
Fix an additive abelian group $\Gamma$, and let $\omega: \Gamma\times \Gamma\lra \C^*$ be 
a commutative factor with the defining properties \cite{B}
\beq
&&\omega(\alpha, \beta) = \omega(\beta, \alpha)^{-1}, \label{eq:cycle-1} \\
&&\omega(\alpha, \beta+\gamma)= \omega(\alpha, \beta) \omega(\alpha, \gamma), \label{eq:cycle-2}\\
&&\omega(\alpha +\beta, \gamma) = \omega(\alpha, \gamma) \omega(\beta, \gamma), \quad \forall
\alpha, \beta, \gamma.  \label{eq:cycle-3}
\eeq
Note in particular that 
$\omega(\alpha, 0)=\omega(0, \alpha)=1$, and $\omega(\alpha, \alpha)=\pm 1$ for all $\alpha$.
We let 
\beq
\Gamma^\pm=\{\alpha\in \Gamma\mid \omega(\alpha, \alpha)=\pm 1\}, \quad 
\Gamma_{\Z_2}= \{\alpha\in \Gamma\mid \alpha+\alpha=0\}. 
\eeq
It is evident that $\Gamma^+$ and $\Gamma_{\Z_2}$ are subgroups. We will call $\Gamma_{\Z_2}$ the $2$-torsion subgroup. Note that if $\alpha\in \Gamma_{\Z_2}$, then  
$\omega(\alpha, \beta) = \omega(\beta, \alpha)=\pm 1$ for all $\beta\in\Gamma$.

A $\Gamma$-graded vector space $V$ is the direct sum $V=\sum_{\alpha\in\Gamma} V_\alpha$ of its homogeneous subspaces $V_\alpha$. Denote $m_\alpha=\dim V_\alpha$. 
For any homogeneous element $v\in V$,  
 we use $d(v)\in \Gamma$ to denote the degree of $v$. 
Clearly $V=V_+\oplus V_-$ with $V_\pm=\sum_{\alpha\in \Gamma^\pm} V_\alpha$.  
In \cite{Z25}, we introduced the set 
\beq\label{eq:support}
\Gamma_R(V)=\{\alpha\in \Gamma\mid V_\alpha\ne 0\}, 
\eeq
which will be referred to as the graded support of $V$. 
If $\dim V<\infty$, the graded support  is necessarily a finite set, and hence so are also $\Gamma_R^\pm(V)=\Gamma_R(V)\cap \Gamma^\pm$. 
In this case, we denote $\aleph(V)=|\Gamma_R(V)|$ and 
$\aleph^\pm(V)=|\Gamma_R^\pm(V)|$. 
We denote by
\beq\label{eq:torsion}
\Gamma_{R, \Z_2}(V)=\Gamma_R(V)\cap\Gamma_{\Z_2},\quad 
\Gamma^\pm_{R, \Z_2}(V)=\Gamma^\pm_R(V)\cap\Gamma_{\Z_2},
\eeq
 the sets of $2$-torsion elements of $\Gamma_R(V)$ and $\Gamma^\pm_R(V)$ respectively.

The space of homomorphisms between any two $\Gamma$-graded vector spaces $U$ and $V$ is naturally $\Gamma$-graded. Note that 
$\Hom_\C(U, V)=\sum_{\alpha, \beta\in\Gamma}\Hom_\C(U_\alpha, V_\beta)$, where $\Hom_\C(U_\alpha, V_\beta)$ is of degree $\beta-\alpha$. Write $\Hom_\C(U, V)_\Upsilon=\sum_{\beta-\alpha=\Upsilon}\Hom_\C(U_\alpha, V_\beta)$. Then 
$\Hom_\C(U, V)=\sum_{\Upsilon} \Hom_\C(U, V)_\Upsilon$. 

Denote by $\Vect$ the category of $\Gamma$-graded vector spaces.
It is a strict braided monoidal category, 
with the monoidal structure given by the tensor product $\ot_\C$, 
and the braiding by a functorial map arising from the commutative factor $\omega$, 
\beq\label{eq:def-tau}
\tau_{V, W}: V\ot_\C W\lra W\ot_\C V,
\eeq
which is defined by 
\beq\label{eq:tau}
\tau_{V, W}(v\ot w)= \omega(d(v), d(w)) w\ot v, 
\eeq
 for any homogeneous $v\in V, w\in W$,  and linearly extended to inhomogeneous elements. 
It is evident that 
\beq\label{eq:invol}
\tau_{W, V} \tau_{V, W}=\id_{V\ot W}.
\eeq
One can also easily verify \cite{Z25} that the functorial map $\tau$ indeed defines 
a braiding for $\Vect$, that is, 
for any  $\Gamma$-graded vector spaces $U, V, W$,
\beq\label{eq:braid}
 &&(\tau_{V,W}\ot\id_U)      \circ    (\id_V\ot \tau_{U, W})\circ (\tau_{U, V}\ot\id_W) \\
 &&= (\id_W\ot \tau_{U, V}) \circ (\tau_{U, W}\ot\id_V)\circ  (\id_U\ot \tau_{V, W}).  
 \nonumber
\eeq

\begin{remark}[{\bf Convention}]\label{rmk:conven}
We shall adopt the convention that any explicit formula involving the commutative factor will be written in a form analogous to \eqref{eq:tau}, but is tacitly understood to be extended linearly for inhomogeneous elements. 
\end{remark}

Observe the following easy facts. 

\smallskip
\noindent{\bf Fact 1}. 
Given any $\xi\in \Gamma$,  there exists  a degree shifting functor $\digamma_\xi$ on $\Vect$ such that
\beq\label{eq:shift}
\digamma_\xi V:=\sum_{\alpha\in \Gamma}(\digamma_\xi V)_\alpha, \quad (\digamma_\xi V)_{\alpha+\xi}=V_\alpha.
\eeq
Note in particular that if $\xi\in\Gamma^-$, then $(\digamma_\xi V)_\pm = V_\mp$. 

\smallskip
\noindent{\bf Fact 2}. 
Given abelian groups $\Gamma_i$ with commutative factors $\omega_i: \Gamma_i\times \Gamma_i\lra \C^*$ 
for $i=1, 2$, we let $\Gamma=\Gamma_1\times\Gamma_2$, and define the map
\beq\label{eq:CD}
\omega: \Gamma\times \Gamma\lra \C^*, \quad 
\omega((\alpha, \mu), (\beta, \nu)) =\omega_1(\alpha, \beta)\omega_2(\mu, \nu), 
\eeq
for all $(\alpha, \mu), (\beta, \nu)\in\Gamma$. It gives rise to 
a commutative factor on $\Gamma$.  There is a functor 
\[
\mathcal{D}: {\mathfrak V}(\Gamma_1, \omega_1)\times {\mathfrak V}(\Gamma_2, \omega_2)\lra {\mathfrak V}(\Gamma, \omega)
\]
such that 
for any object $V=\sum_{\alpha\in\Gamma_1}V_\alpha$ of ${\mathfrak V}(\Gamma_1, \omega_1)$
and $W=\sum_{\beta\in\Gamma_2}W_\beta$ of ${\mathfrak V}(\Gamma_2, \omega_2)$, we have  
 $\mathcal{D}(V\times W)=\sum_{(\alpha, \beta)\in \Gamma} (\mathcal{D}(V\times W))_{(\alpha, \beta)}$ $\in {\mathfrak V}(\Gamma, \omega)$, where $(\mathcal{D}(V\times W))_{(\alpha, \beta)}=V_\alpha\ot W_\beta$.  In particular,  
\[
\mathcal{D}(V\times \C)=V\ot \C\simeq V, \quad \mathcal{D}(\C\times W)=\C\ot W\simeq W,
\]
where we regard $\C$ as being homogeneous of degree $0$.

\begin{example}\label{rmk:parity} 
Note the following special case of Fact 2. 
For $\Gamma_2 =\Z_2=\{0, 1\}$ with $\omega_2=sign:\Z_2\times\Z_2\lra \{1, -1\}$ 
given by $sign(i, j)=(-1)^{i j}$, there is the following commutative factor 
$
\omega((\alpha, i), (\beta, j)) =(-1)^{i j}\omega_1(\alpha, \beta)
$
on $\Gamma_1\times\Z_2$. Hence we have the functor 
\[
\mathcal{D}: {\mathfrak V}(\Gamma_1, \omega_1)\times {\mathfrak V}(\Z_2, sign)\lra {\mathfrak V}(\Gamma_1\times\Z_2, \omega_1\times sign).
\]
\end{example}

This example has important implications on bilinear forms on $\Gamma$-graded vector spaces, see 
Section \ref{sect:deg-shift} below. 

\subsection{Bilinear forms}

\subsubsection{Bilinear forms}\label{sect:forms}

Fix $\gamma\in\Gamma$. A bilinear form $\kappa: V\times V\lra\C$ is said to 
be homogeneous of degree $\gamma$ if 
$\kappa(V_\alpha, V_\beta)=0$ for all $\alpha, \beta$ such that $\alpha+\beta+\gamma\ne 0$.  
Note in particular that 
\begin{itemize}
\item if $\gamma\in \Gamma^-$, then $\kappa(V_+, V_+)=\kappa(V_-, V_-)=\{0\}$;
 \item if $\gamma\in \Gamma^+$, then $\kappa(V_+, V_-)=\kappa(V_-, V_+)=\{0\}$. 
\end{itemize}
If $\kappa$ is non-degenerate, then $\dim V_\alpha = \dim V_{-\alpha-\gamma}$ for all $\alpha$. 

\begin{remark}\label{rmk:pairing}
Assume that $\kappa: V\times V\lra\C$ is a non-degenerate bilinear form, 
which is homogeneous of degree $0$. Then 
$\dim V_\alpha = \dim V_{-\alpha}$ for all $\alpha$. Furthermore, 
\begin{enumerate}[label=(\roman*)]
\item if $\alpha\not\in \Gamma_{R, \Z_2}(V)$, 
the restriction of $\kappa$ to $V_\alpha \oplus V_{-\alpha}$ is non-degenerate, and 
$\kappa(V_\alpha,  V_{\alpha})=\{0\} = \kappa(V_{-\alpha}, V_{-\alpha})$; 
\item if $\alpha\in \Gamma_{R, \Z_2}(V)$, the restriction of $\kappa$ to $V_\alpha$ is 
non-degenerate. 
\end{enumerate}
\end{remark}

A bilinear form $\kappa$  is said to be  
\[
\baln
\bullet\quad &\text{$\omega$-symmetric if }  
\kappa(v_\alpha, v'_\beta) =\omega(\alpha, \beta) \kappa(v'_\beta, v_\alpha), \\
\bullet\quad &\text{$\omega$-skew symmetric if } 
\kappa(v_\alpha, v'_\beta) =-\omega(\alpha, \beta) \kappa(v'_\beta, v_\alpha),
\ealn
\]
for all $v_\alpha\in V_\alpha$ and $v'_\beta\in V_\beta$. 
We shall often regard a bilinear form on $V$ as a linear function from $V\ot V$ to $\C$. 
Thus we can define 
\[
\kappa^{(s)}:=\frac{1}{2}\kappa\circ(\id_V\ot\id_V+\tau_{V, V}), 
\quad \kappa^{(a)}:=\frac{1}{2}\kappa\circ(\id_V\ot\id_V-\tau_{V, V}). 
\]
The maps $\kappa^{(s)}$ and $\kappa^{(a)}$ lead to $\omega$-symmetric and $\omega$-skew symmetric bilinear forms on $V$ respectively.

\subsubsection{Effects of degree shifting}\label{sect:deg-shift}
It is useful to consider the effect of degree shifting on bilinear forms. 
Note the following facts.

A homogeneous bilinear form of degree $\gamma$ on $V$ gives rise to a 
homogeneous bilinear form of degree $\gamma+2\xi$ on $\digamma_\xi V$.

If $\xi\in \Gamma^-\cap\Gamma_{\Z_2}$, a homogeneous $\omega$-symmetric (resp. skew symmetric) bilinear form 
of degree $\gamma$ on $V$ leads to a homogeneous
$\omega$-skew symmetric (resp. symmetric) bilinear form of 
the same degree on $\digamma_\xi V$.

If $\Gamma^-\cap\Gamma_{\Z_2}=\emptyset$,  we may use the functor ${\mathcal D}$ (see \eqref{eq:CD})
in the situation of Remark \ref{rmk:parity} to map $V$ to 
an object ${\mathcal D}(V\times \C)=V$ in the category ${\mathfrak V}(\Gamma\times\Z_2, \omega\times sign)$, 
and then apply degree  shifting to obtain  $\digamma_{(0, 1)} {\mathcal D}(V\times \C)=\digamma_{(0, 1)}V$.  
Now $\omega$-symmetric (resp. skew symmetric) bilinear forms on $V$ 
become $(\Gamma\times\Z_2, \omega\times sign)$-skew symmetric 
(resp. symmetric) bilinear forms on $\digamma_{(0, 1)}V$.

\subsubsection{The $\omega$-trace bilinear form}
Assume that $\dim V<\infty$. 
There is a generalised trace ${\rm tr}_{(\Gamma, \omega)}: \End_\C(V)\lra\C$, referred to as the $\omega$-trace, which is defined as follows  \cite[\S3.1.2]{Z25}. For
any $\varphi =\sum_{\alpha, \beta\in\Gamma}  \varphi(\alpha, \beta)$ 
with $\varphi(\alpha, \beta)\in \Hom_\C(V_\beta, V_\alpha)$,
\[
{\rm tr}_{(\Gamma, \omega)}(\varphi) =\sum_{\alpha\in \Gamma}  \omega(\alpha, \alpha) {\rm tr}(\varphi(\alpha, \alpha)), 
\]
where ${\rm tr}$ denotes the usual trace.  Clearly ${\rm tr}_{(\Gamma, \omega)}$ is  homogeneous of degree $0$. 
Write  $\varphi_0= \sum_{\alpha\in \Gamma}  \varphi(\alpha, \alpha)$, then
$
{\rm tr}_{(\Gamma, \omega)}(\varphi) =  {\rm tr}_{(\Gamma, \omega)}(\varphi_0).
$

\begin{lemma}[\cite{Z25}]\label{lem:om-trace}
The $\omega$-trace ${\rm tr}_{(\Gamma, \omega)}: \End_\C(V)\lra\C$ has the following property. 
For any $\varphi\in\End_\C(V)_\alpha$ and $\psi\in\End_\C(V)_\beta$, where $\alpha, \beta\in\Gamma$, 
\beq\label{eq:om-sym}
{\rm tr}_{(\Gamma, \omega)}(\varphi \psi)&=&\omega(\alpha, \beta) {\rm tr}_{(\Gamma, \omega)}( \psi \varphi).
\eeq
This implies that $ {\rm tr}_{(\Gamma, \omega)}([\eta, \xi])=0$ for all $\eta, \xi\in\End_\C(V)$, 
where $[\ , \ ]$ is the graded commutator defined by \eqref{eq:bracket} below. 
\end{lemma}

The $\omega$-trace leads to the following bilinear form on $\End_\C(V)$.
\beq\label{eq:biform}
&&( \ , \ ): \End_\C(V)\times \End_\C(V)\lra \C, \\
&&( X,  Y )={\rm tr}_{(\Gamma, \omega)}(X Y), \quad \forall  X, Y\in \End_\C(V).\nonumber
\eeq
This is homogeneous of degree $0$, and was shown to be non-degenerate in \cite{Z25}. 
It follows the $\omega$-symmetry of $tr_{(\Gamma, \omega)}$ that the bilinear form 
\eqref{eq:biform} is $\omega$-symmetric.

\medskip
{\bf Hereafter we  consider only homogeneous bilinear forms  of degree $0$.}

\subsection{Classical Lie $(\Gamma, \omega)$-algebras}

Recall that \cite{RW, Sch79} a Lie $(\Gamma, \omega)$-algebra
is a $\Gamma$-graded vector space 
$\fg=\sum_{\alpha\in\Gamma} \fg_\alpha$ endowed with a bilinear map
$
[\ , \ ]: \fg\times \fg \lra \fg, 
$
the Lie $\omega$-bracket, 
which is homogeneous of degree $0$, and satisfies the conditions 
\beq
&&[X, Y] =- \omega(d(X), d(Y))[Y, X], \label{eq:skew-2}\\
&&[X, [Y, Z]] = [ [X, Y], Z] + \omega(d(X), d(Y)) [Y, [X, Z]],  \label{eq:Jocob-2}
\eeq
for any $X, Y, Z\in \fg$  (in the convention of Remark \ref{rmk:conven}). 
We will loosely call it a Lie colour superalgebra
if $\fg_-=\sum_{\alpha\in\Gamma^-}\fg_\alpha\ne 0$, 
and a Lie colour algebra otherwise. 

The adjoint representation  $ad: \fg \lra \End_\C(\fg)$, $X\mapsto ad_X$, is defined in the usual way by $ad_X(Y)=[X, Y]$ for all $Y$.

Let $V$ be a $\Gamma$-graded vector space. 
The general linear 
Lie $(\Gamma, \omega)$-algebra $\gl(V)$ is 
$\End_\C(V)$ with the Lie $\omega$-bracket given by the graded commutator
\beq\label{eq:bracket}
[X, Y]=X Y - \omega(d(X), d(Y)) Y X, \quad  X, Y\in \gl(V).
\eeq

Let $B(V)=(e_1, e_2, \dots, e_D)$ be an ordered homogeneous bases for $V$, 
where $D= \dim V$. Let $\gamma_a:= d(e_a)$ the $\Gamma$-degree of $e_a$ for each $a$.
As in \cite[\S 3.3.1]{Z25}, we let $\BE_{a b}\in \End_\C(V)$ (for $a, b=1, 2, \dots, d$)  be the matrix units relative to the basis $B(V)$. 
They obey the standard relations 
\beq
 \BE_{a b}\cdot e_c = \delta_{b c} e_a,  \quad
 \BE_{a b} \BE_{c d} = \delta_{b c} \BE_{a d}.
\eeq
The degree of $\BE_{a b}$ is given by $\gamma_a-\gamma_b$. Note in particular that $\BE_{a a}$ is of degree $0$ for all $a$. 
The matrix units form a homogeneous basis of $\gl(V)$, which enable us to express 
the defining relations of $\gl(V)$ as
\beq \label{eq:CR}
\phantom{XXXX}
[\BE_{a b},  \BE_{c d} ]= \delta_{b c} \BE_{a d} 
-\omega(\gamma_a - \gamma_b, \gamma_c - \gamma_d) \delta_{a  d} \BE_{c b}. 
\eeq

Fix $a\ne b$, and let $V_{a b}=\C e_a\oplus \C e_b$. 
By \cite[Lemma 3.1]{Z25}, the elements $\BE_{a b}, \BE_{b a},  \BE_{a a}$, $\BE_{b b}$ 
span a Lie $(\Gamma, \omega)$-subalgebra $\gl(V_{a b})$, which is isomorphic to 

$\bullet$
$\gl_2(\C)$ if $\gamma_a - \gamma_b$ belongs to $\Gamma^+$; 

$\bullet$ 
$\gl_{1|1}(\C)$ if $\gamma_a - \gamma_b$ belongs to $\Gamma^-$.

Since $\tr_{(\Gamma, \omega)}([X, Y])=0$ for all $X, Y\in \gl(V)$ by Lemma \ref{lem:om-trace}, the subspace 
\[
\fsl(V)=\{X\in\gl(V)\mid  \tr_{(\Gamma, \omega)}(X)=0\}
\]
forms a Lie $(\Gamma, \omega)$-subalgebra $\fsl(V)$ of $\gl(V)$, 
the special linear Lie $(\Gamma, \omega)$-algebra of $V$. 
Note that if $\dim V_+=\dim V_-$, then $\fsl(V)$ contains the identity element of $\gl(V)$, which spans an abelian ideal. Thus $\fsl(V)$ is not simple in this case. 

Now we consider Lie $(\Gamma, \omega)$-subalgebras of $\gl(V)$ preserving 
non-degenerate bilinear forms.  
By discussions in Section \ref{sect:deg-shift}, we only need to consider $\omega$-symmetric bilinear forms. 

\begin{definition}[\cite{Sch79}]
Assume that the $\Gamma$-graded vector space $V$ is equipped with a non-degenerate $\omega$-symmetric bilinear form $\kappa: V\times V\lra \C$. 
The orthosymplectic Lie $(\Gamma, \omega)$-algebra $\osp(V; \kappa)$
consists of elements $X=\sum_{\alpha\in\Gamma} X_\alpha\in \End_\C(V)$,  which satisfy the following condition.
\beq\label{eq:orth}
\kappa(X\cdot v, v') + \sum_{\alpha} \omega(\alpha, d(v)) \kappa(v, X_\alpha\cdot v')=0, \quad \forall v, v'\in V.
\eeq
\end{definition}
\begin{proof}[Proof that $\osp(V; \kappa)$ is a Lie $(\Gamma, \omega)$-subalgebra of $\fsl(V)$]
{\   }

\smallskip\noindent
{\em (a).  
$\tr_{(\Gamma, \omega)}(X) =0$ for all $X\in \osp(V; \kappa)$, i.e., 
$\osp(V; \kappa)\subset \fsl(V)$.}

Denote $\kappa_{a b}=\kappa(e_a, e_b)$, 
and let $\kappa^{-1}$ be the inverse of the matrix $(\kappa(e_a, e_b))$. 
Then $\ol{e}_a=\sum_c (\kappa^{-1})_{a c} e_c$ satisfies $\kappa(\ol{e}_a, e_b)=\delta_{a b}$, 
and $d(\ol{e}_a)=-\gamma_a$. 
Note that $\gamma_a+\gamma_b=0$ if $\kappa_{a b}\ne 0$ or  
$(\kappa^{-1})_{a b}\ne 0$, and in this case, $\omega(\gamma_a, \gamma_b)=\omega(\gamma_b, \gamma_a)=\omega(\gamma_a, \gamma_a)= \omega(\gamma_b, \gamma_b)=\pm 1$. 
By the $\omega$-symmetry of $\kappa$, 
\beq
(\kappa^{-1})_{b a}=\omega(\gamma_a, \gamma_a)(\kappa^{-1})_{a b}=\omega(\gamma_b, \gamma_b)(\kappa^{-1})_{a b}. \label{eq:kappa-inv}
\eeq

Recall that for any $X=\sum_\alpha X_\alpha$ with $X_\alpha \in \osp(V; \kappa)_\alpha$ for all $\alpha$, 
\[
\baln
\tr_{(\Gamma, \omega)}(X) &=\tr_{(\Gamma, \omega)}(X_0)= \sum_a \omega(\gamma_a, \gamma_a) \kappa(\ol{e}_a, X_0 e_a) \\
&= \sum_{a, c} \omega(\gamma_a, \gamma_a) (\kappa^{-1})_{a c}\kappa( e_c, X_0 e_a)\\
&= -\sum_{a, c} \omega(\gamma_a, \gamma_a) (\kappa^{-1})_{a c}\kappa( X_0 e_c,  e_a) 
		\quad \text{(by definition of $\osp$)}\\
&= - \sum_{a, c} \omega(\gamma_a, \gamma_a) (\kappa^{-1})_{a c} \omega(\gamma_a,   \gamma_c) \kappa(e_a, X_0  e_c)  
		\quad \text{(by $\omega$-symmetry of $\kappa$)}.
\ealn
\]
Using \eqref{eq:kappa-inv}, we obtain 
\[
\baln
\tr_{(\Gamma, \omega)}(X) &=- \sum_{a, c} \omega(\gamma_a,\gamma_a) (\kappa^{-1})_{a c}\kappa( e_c, X_0 e_a) = - \tr_{(\Gamma, \omega)}(X). 
\ealn
\] 
This shows that $\tr_{(\Gamma, \omega)}(X)=0$, 
and hence $X\in \fsl(V)$.

\smallskip\noindent
{\em (b).  
$[X, Y]\in \osp(V; \kappa)$, $\forall X, Y\in \osp(V; \kappa)$, i.e., 
$\osp(V; \kappa)$ is a subalgebra of $\fsl(V)$. }

We may assume that $X, Y\in \osp(V; \kappa)$ are homogeneous. Then for homogeneous  $v, v'\in V$, 
\[
\baln
\kappa([X, Y]\cdot v, v') 
&= \kappa(X \cdot(Y\cdot v), v') - \omega(d(X), d(Y)) \kappa(Y \cdot(X\cdot v), v').
\ealn
\]
We can re-write  the right hand side as follows,
\[
\baln
&\omega(d(Y), d(v))\omega(d(X), d(Y)+d(v)) \kappa(v, Y\cdot  (X \cdot v')) \\
&- \omega(d(X), d(Y)) \omega(d(Y), d(X)+d(v)) \omega(d(X), d(v))  \kappa( v, X\cdot(Y \cdot v')) \\
&= \omega(d(X)+d(Y), d(v))\omega(d(X), d(Y)) \kappa(v, Y\cdot  (X \cdot v')) \\
&- \omega(d(X)+d(Y), d(v)) \kappa( v, X\cdot(Y \cdot v')) \\
&=- \omega(d([X, Y]), d(v)) \kappa( v, [X, Y] \cdot v'). 
\ealn
\]
Thus $\kappa([X, Y]\cdot v, v') +\omega(d([X, Y]), d(v)) \kappa( v, [X, Y] \cdot v')=0$, 
and hence $[X, Y]\in \osp(V; \kappa)$. This shows that $\osp(V; \kappa)$ is a Lie $(\Gamma, \omega)$-subalgebra of $\fsl(V)$. 
\end{proof}

We now give an explicit construction of $\osp(V; \kappa)$.

 Consider the $\Gamma$-graded vector space isomorphism 
$\varpi: V\ot V\lra \End_\C(V)$ defined,  for any $v, v\in V$, by
\[
\varpi(v\ot v')(w)= \kappa(v', w)  v, \quad \forall w\in V. 
\]
Note in particular that
$
\varpi(e_a\ot\ol{e}_b) =\BE_{a b}.   
$
Let $c_0=\sum_{a} e_a\ot \ol{e}_a$, then $\varpi(c_0)=\id_V$.

\begin{theorem}\label{thm:osp-def}
\begin{enumerate}
\item The orthosymplectic Lie colour superalgebra is 
given by 
\beq
\osp(V; \kappa)= \varpi(\id_V\ot\id_V-\tau_{V, V})(V\ot V), 
\eeq 
and hence $\dim\osp(V; \kappa) = \frac{1}{2}M_+(M_+-1) +\frac{1}{2}M_-(M_-+1) + M_+M_-$.

 \item
The following elements, for $a, b=1, 2, \dots, \dim V$,  
\beq
X_{a b} &:=& \varpi(\id_V\ot\id_V- \tau_{V, V})(e_a\ot\ol{e}_b)   \nonumber\\
&=& \BE_{a b} - \omega(\gamma_{b}, \gamma_{a}) \sum_{a', b'} (\kappa^{-1})_{b b'} \kappa_{a  a'} \BE_{b' a'},   \label{eq:Xab}
\eeq
 span $\osp(V; \kappa)$, and satisfy the relations
\beq
{[X_{a b},  X_{c d}]}&=&\delta_{b c} X_{a d}
	-\omega(\gamma_a-\gamma_b, \gamma_c-\gamma_d)\delta_{d a} X_{c b} \label{eq:def-rel}\\
&&- \omega(\gamma_b, \gamma_a)  \kappa_{a c} \sum_{b'}(\kappa^{-1})_{b b'}X_{b' d}\nonumber\\
&&+\omega(\gamma_b, \gamma_a) \omega(\gamma_a-\gamma_b, \gamma_c-\gamma_d)  (\kappa^{-1})_{b d} \sum_{a'}\kappa_{a a'}X_{c a'}. \nonumber
\eeq
\end{enumerate}
\end{theorem}
\begin{proof}
Let $\pi_\pm = (\id_V\ot\id_V\pm \tau_{V, V})$, and denote
 $\wedge^2_\omega V= \pi_-(V\ot V)$ 
and $S^2_\omega V= \pi_+(V\ot V)$. 
Then $V\ot V= \wedge^2_\omega V\oplus S^2_\omega V$ as $\gl(V)$-module. 
Since $e_a\ot\ol{e}_b$,  for all $a$ and $b$, form a basis for $V\ot V$,  the elements $X_{a b}$ span 
$\varpi(\wedge^2_\omega V)$,  and the elements $S_{a b}=\varpi\pi_+(e_a\ot\ol{e}_b)$  span 
$\varpi(S^2_\omega V)$. 

For arbitrary homogeneous elements  $v, v'$ of $V$, we have 
\[
\baln
\kappa(\varpi(e_a\ot\ol{e}_b) v, v') 
&= \omega(\gamma_a-\gamma_b,  d(v))\kappa(v, \varpi\tau_{V, V}(e_a\ot\ol{e}_b)  v') \\
\kappa(\varpi\tau_{V, V}(e_a\ot\ol{e}_b) v, v') 
&=\omega(\gamma_a-\gamma_b, d(v)) \kappa(v, \varpi(e_a\ot\ol{e}_b) v'), 
\ealn
\]
which can be verified by considering $v=e_c$ and $v'=e_d$ for all $c, d$. 
Therefore, 
\[
\baln
\kappa(X_{a b} v, v') &=-\omega(\gamma_a-\gamma_b, d(v)) \kappa(v, X_{a b} v'),\\ 
\kappa(S_{a b} v, v') &=\omega(\gamma_a-\gamma_b, d(v)) \kappa(v, S_{a b} v'), \quad \forall a, b.
\ealn
\]
This leads to
$\varpi(\wedge^2_\omega V)\subset \osp(V; \kappa)$ and 
$\varpi(S^2_\omega V)\cap \osp(V; \kappa)=0$. 
As $\varpi$ is a bijection, $\varpi(\wedge^2_\omega V)= \osp(V; \kappa)$.

Clearly $\pi_-(e_b\ot e_a)= - \omega(\gamma_b, \gamma_a)\pi_-(e_a\ot e_b)$ for all $a, b$. In particular, if $e_a\in V_+$, then $\pi_-(e_a\ot e_a)=0$;  and if $e_a\in V_-$, then $\pi_-(e_a\ot e_a)=2 e_a\ot e_a$.  The dimension formula for $\osp(V; \kappa)$ follows directly from these observations.  

Equation \eqref{eq:Xab} can be verified by the easy computation below. 
\[
\baln
X_{a b} &= \varpi(e_a\ot \ol{e}_b) - \omega(\gamma_b, \gamma_a)\varpi(\ol{e}_b\ot e_a)\\
&= \varpi(e_a\ot \ol{e}_b) - \omega(\gamma_b, \gamma_a)\sum_{b', a'}  \kappa_{a a'} (\kappa^{-1})_{b b'}\varpi(e_{b'}\ot \ol{e}_{a'})\\
&=\BE_{a b} - \omega(\gamma_b, \gamma_a)\sum_{b', a'} (\kappa^{-1})_{b b'} \kappa_{a a'} \BE_{b' a'}.
\ealn
\]
To prove the commutation relation \eqref{eq:def-rel}, note that
\[
\varpi(u\ot u') \varpi(v\ot v')= \kappa(u', v) \varpi(u\ot v'), \quad \forall u\ot u', v\ot v'\in V\ot V.
\]
Using this, we can prove the following relations by straightforward calculations. 
\[
\baln
X_{a b} X_{c d}
&=\delta_{b c} \varpi(e_a\ot\ol{e}_d) - \omega(\gamma_b, \gamma_a)  
		\kappa_{a c} \varpi(\ol{e}_b\ot\ol{e}_d)\\
	& - \omega(\gamma_d, \gamma_c) (\kappa^{-1})_{d b} \varpi(e_a\ot e_c)
	+ \omega(\gamma_a, \gamma_c - \gamma_b-\gamma_a)\delta_{a d}\varpi(\ol{e}_b\ot e_c), \\
X_{c d} X_{a b}
&=\delta_{d a} \varpi(e_c\ot\ol{e}_b) - \omega(\gamma_d, \gamma_c)  
		\kappa_{c a} \varpi(\ol{e}_d\ot\ol{e}_b)\\
	& - \omega(\gamma_b, \gamma_a) (\kappa^{-1})_{b d} \varpi(e_c\ot e_a)
	+ \omega(\gamma_c, \gamma_a - \gamma_d-\gamma_c)\delta_{c b}
	\varpi(\ol{e}_d\ot e_a). 
\ealn
\]
Combining these relations, we obtain
\[
\baln
[X_{a b},  X_{c d}]
&= X_{a b} X_{c d} - \omega(\gamma_a-\gamma_b, \gamma_c-\gamma_d) X_{c d} X_{a b}\\
&=\delta_{b c} \varpi\pi_-(e_a\ot\ol{e}_d) \\
&-\omega(\gamma_a-\gamma_b, \gamma_c-\gamma_d)\delta_{d a} \varpi\pi_-(e_c\ot\ol{e}_b) \\
&- \omega(\gamma_b, \gamma_a)  \kappa_{a c} \varpi\pi_-(\ol{e}_b\ot\ol{e}_d)\\
& + \omega(\gamma_a-\gamma_b, \gamma_c-\gamma_d) \omega(\gamma_b, \gamma_a) (\kappa^{-1})_{b d} \varpi\pi_-(e_c\ot e_a).
\ealn
\]
This leads to equation \eqref{eq:def-rel}. 
\end{proof}

\subsubsection{Remarks}
We give a few brief remarks on the invariant theory of the classical Lie $(\Gamma, \omega)$-algebras.  

The invariant theory of the general linear Lie colour superalgebra $\gl(V)$ was systematically developed in \cite[\S4]{Z25}, where a Howe duality of type $(\gl(V), \gl(V'))$ was established, 
which in particular implies the first and second fundamental theorems of invariant theory in a $(\Gamma, \omega)$-commutative algebra setting, and 
a Schur-Weyl duality between $\gl(V)$ and the symmetric group, with   
the action of $\Sym_r$ on $V^{\ot r}$, for any $r$, constructed from 
the symmetry $P:=\tau_{V, V}$ defined by \eqref{eq:def-tau}. 
One can easily reformulate some of the results in a categorical framework by constructing 
a full tensor functor from the oriented Brauer category (see, e.g., \cite[\S4]{LZ24})  to the full subcategory of $\gl(V)$-modules with objects being repeated tensor products of $V$ and its dual. 

The main ideas of applying the Brauer category \cite{LZ15} to study the invariant theory of the orthosymplectic supergroup \cite{DLZ, LZ17, LZ21} should also work for the orthosymplectic Lie colour superalgebra $\osp(V; \kappa)$.  Now observe that the symmetry $P$ is an $\osp(V; \kappa)$-map,  and the bilinear form $\kappa$ gives rise to  the $\osp(V; \kappa)$-maps  
\beq
&&\wh{C}: V\ot V\lra \C, \quad v\ot v'\mapsto \kappa(v, v'), \quad \forall v, v'\in V, \\
&&\check{C}: \C\lra V\ot V, \quad a\mapsto a c_0, \quad \forall a\in\C,  
\eeq
where $c_0\in V\ot V$ is defined immediately before Theorem \ref{thm:osp-def}, and $\wh{C}$ is nothing but $\kappa$ regarded as a linear map from $V\ot V$ to $\C$.
The  maps $P$, $\wh{C}$ and $\check{C}$ satisfy relations which are formally the same as those in Lemma \cite[Lemma 5.2]{LZ24}. This enables one to describe $\osp(V; \kappa)$-invariants diagrammatically by constructing a tensor functor from the Brauer category to the full 
subcategory $\CT(V)$ of $\osp(V; \kappa)$-modules with objects $V^{\ot r}$ for all $r\in\Z_+$. 

The tensor functor will not be full in general, but it is expected to product all $\OSp(V;\kappa)$-morphisms $V^{\ot r}\lra V^{\ot s}$ for all $r, s$, where $\OSp(V;\kappa)$ is a $(\Gamma, \omega)$-group scheme, which can be constructed in a similar way as for the general linear colour supergroup $\GL(V)$ given in \cite[\S6.4]{Z25}.

\section{Lie and affine colour algebras fulfilling Cartan--Weyl paradigm}\label{sect:CWP-0}
Recall that Cartan–Weyl theory describes the structure of semisimple Lie algebras and their generalisations in terms of root systems, and develops their representation theory through weight modules, using Cartan subalgebras that are semisimple in the adjoint representation. 
For ease of reference, we shall refer to the conceptual framework underlying this theory 
as the Cartan–Weyl paradigm.
Here we examine the $\Gamma$-graded context, and construct the simple Lie colour algebras
and associated affine Lie colour algebras which satisfy the Cartan--Weyl paradigm.

\subsection{Assessing the Cartan--Weyl paradigm in the colour context}
\label{sect:CWP}

The general and special linear Lie colour superalgebras 
were shown  \cite{Z25} to satisfy the Cartan-Weyl paradigm, with  
homogeneous Cartan subalgebras concentrated in degree $0$, 
which have the required semi-simplicity property. 
However, for a general Lie colour (super)algebra, this cannot be assumed a priori. 
The underlying $\Gamma$-grading can structurally obstruct the existence of any 
suitable homogeneous Cartan subalgebra of degree $0$.  

\begin{remark}\label{rem:Cartan-deg0}
One might be tempted to allow for maximal $\omega$-commutative subalgebras which are 
not homogeneous of degree $0$. 
However, such a subalgebra does not admit 
$\Gamma$-graded weight modules since the ground field is concentrated in degree $0$. 
More seriously, since $\omega$-commutativity in the universal enveloping colour algebra 
means non-commutativity in the usual sense if $\omega$ is non-trivial, 
the standard simultaneous diagonalisation arguments no longer apply. 
\end{remark} 



Let us consider the orthosymplectic Lie colour superalgebra $\osp(V; \kappa)$. 
The situation can be summarised by the following theorem.
\begin{theorem}\label{thm:paradigm}
The orthosymplectic Lie colour superalgebra   
$\osp(V; \kappa)$ admits a homogeneous Cartan subalgebra of degree $0$, 
whose adjoint action on the entire Lie colour superalgebra is semi-simple
and hence yields a root space decomposition, 
if and only if there exists at most one $2$-torsion element $\alpha \in \Gamma_{\Z_2}(V)$ such that $V_\alpha$ is odd dimensional.
\end{theorem}
\begin{proof}
Since $\kappa$ is homogeneous of degree $0$, the Lie colour superalgebra $\osp(V; \kappa)$ contains 
the subalgebra $\osp(V; \kappa)_0 = \fso(V_+; \kappa|_{V_+}) \oplus \fsp(V_-; \kappa|_{V_-})$, 
where the orthogonal Lie colour subalgebra $\fso(V_+; \kappa|_{V_+})$ 
and the symplectic Lie colour subalgebra $\fsp(V_-; \kappa|_{V_-})$ graded commute.

As a $\Gamma$-graded $\osp(V; \kappa)_0$-module, the entire algebra decomposes into the 
direct sum of $\osp(V; \kappa)_0$ and a submodule which is isomorphic to $V_+ \ot V_-$ thus has no degree $0$ component. Therefore the entire degree $0$ component of $\osp(V; \kappa)$ is contained in $\fh\subset \osp(V; \kappa)_0$.

If $\osp(V;\kappa)$ admits a homogeneous degree $0$ Cartan subalgebra $\fh$ giving a Cartan--Weyl decomposition, then $\fh\subset \osp(V; \kappa)_0$, and hence each of the direct summands $\fso(V_+;\kappa|_{V_+})$ and $\fsp(V_-;\kappa|_{V_-})$ must also admit a Cartan--Weyl description.
Conversely, if both $\fso(V_+; \kappa|_{V_+})$ and $\fsp(V_-; \kappa|_{V_-})$ 
satisfy the Cartan--Weyl paradigm, then $V_+ \ot V_-$ is a $\Gamma$-graded weight module for $\osp(V; \kappa)_0$, 
and hence the Cartan--Weyl description extends to the entire algebra $\osp(V; \kappa)$. 
Therefore, $\osp(V; \kappa)$ fulfils the Cartan--Weyl paradigm if and only if both 
$\fso(V_+; \kappa|_{V_+})$ and $\fsp(V_-; \kappa|_{V_-})$ do.

We will show in Section \ref{sect:C} that the symplectic Lie colour algebra   
$\fsp(V_-; \kappa|_{V_-})$ always satisfies the Cartan--Weyl paradigm. 
However, this is not always the case for 
the orthogonal Lie colour algebra $\fso(V_+; \kappa|_{V_+})$.
As shown in Sections \ref{sect:B} and \ref{sect:D}, this depends on 
the structure of the subspace of $V$ with graded support 
$\Gamma^+_{R, \Z_2}(V)$. 
Therefore, the proof reduces to analysing $\fso(V_+; \kappa|_{V_+})$.

If $V_\alpha$ with $\alpha \in \Gamma_{\Z_2}(V)$ is odd-dimensional, 
it must be a subspace of $V_+$. Thus, the given condition on $V$ in the theorem 
is equivalent to requiring that there is at most one $\alpha \in \Gamma_{\Z_2}(V)$ 
such that $(V_+)_\alpha$ is odd-dimensional. 
We show in Sections \ref{sect:D} and \ref{sect:B} that this is the necessary 
and sufficient condition for $\fso(V_+; \kappa|_{V_+})$ to fulfil the Cartan--Weyl paradigm.
\end{proof}

\begin{remark}\label{rmk:color-y-n}
The failure of the Cartan--Weyl paradigm will be impossible to detect if 
one works with the cocycle twisted Lie superalgebras \cite{Sch79} 
of Lie colour superalgebras. 
This shows that cocycle twisting/bosonisation 
can obscure structurally significant information about Lie colour superalgebras. 
There is a different kind of example supporting this assertion in \cite[\S6]{Mo}, 
where incompatibility was observed between the gradings of a Weyl colour algebra 
and Lie colour superalgebras realised in it.  
\end{remark}

\begin{remark}
Some of the results in this section and Sections \ref{sect:C} and \ref{sect:D-B} below were also obtained by N. Aizawa and J. Van der Jeugt, in particular, the failure of the Cartan--Weyl paradigm in some cases of the orthogonal Lie colour algebra. We thank them for sharing their insights. 
\end{remark}

%
%
%
%

\medskip 
{\bf Hereafter we will consider only Lie colour algebras $\fg$ and their quantum groups}.
%
The super case will be treated in a sequel of this paper. 

\subsection{General and special linear Lie colour algebras}

Let $V$ be a finite dimensional $\Gamma$-graded vector space such that $\Gamma_R(V)= \Gamma_R^+(V)$, thus $V_-=0$.  Then $\gl(V)$ and $\fsl(V)$ are Lie colour algebras. 

The material  in this section is largely extracted from \cite[\S3.1]{Z25}. 
We present it here to set the stage for 
studying other classical Lie colour algebras.  

\subsubsection{Structure theory}
In the special case $a=b$, equation \eqref{eq:CR} reduces to
\beq \label{eq:Cartan}
[\BE_{a a},  \BE_{c d} ]= (\delta_{a c} - \delta_{a d}) \BE_{c d}. 
\eeq
Let $\fh=\sum_{a=1}^{\dim V}\BE_{a a}$. It is of crucial importance that $\fh$ 
is homogeneous of degree $0$, forming a usual abelian Lie subalgebra of $\gl(V)$. 
Equation \eqref{eq:Cartan} shows that $\fh$ is a Cartan subalgebra of $\gl(V)$, 
which enables one to describe the structure 
of $\gl(V)$ in the standard way via root space decomposition. 

We introduce a basis $\{\varepsilon_a\mid 1\le a \le \dim V\}$ for the dual space $\fh^*$ of $\fh$ such that 
\[
\varepsilon_a(\BE_{b b})=\delta_{a b}, \quad a, b=1, 2, \dots, \dim V. 
\]
Then \eqref{eq:Cartan} can be re-written as 
\beq \label{eq:roots}
[\BE_{a a},  \BE_{c d} ]= (\varepsilon_c -  \varepsilon_d)(\BE_{a a}) \BE_{c d}. 
\eeq
Thus the set of roots of $\gl(V)$ with this choice of the Cartan subalgebra is given by 
\[
\baln
\Phi =\{\varepsilon_a -  \varepsilon_b \mid a\ne b\}. 
\ealn
\]
We take the set $\Phi^+$ of  positive roots to consist of the following elements.
\[
\Phi^+ =\{\varepsilon_a -  \varepsilon_b \mid a< b\}. 
\]
Then $\Phi=\Phi^+\cup(-\Phi^+)$. 

For any root $\Upsilon\in \Phi$, the corresponding root space is 
\[
\fg_\Upsilon =\{X\in \gl(V)\mid [h, X]=\Upsilon(h)X, \ \forall h\in\fh\},
\]
which is $1$-dimensional. Note that $\fg_\Upsilon$ is homogeneous in the $\Gamma$-grading. 
We define the following map for later use. 
\beq\label{eq:xi-map}
\xi: \Phi\lra \Gamma, \quad \xi(\Upsilon)=\text{degree of $\fg_\Upsilon$}. 
\eeq

Let $\fn= \sum_{\Upsilon\in\Phi^+} \fg_\Upsilon$ and $\ol{\fn}=\sum_{\Upsilon\in\Phi^+} \fg_{-\Upsilon}$, which are 
Lie  colour subalgebras of $\gl(V)$. We have the following triangular decomposition.
\beq
\gl(V)= \ol{\fn}+\fh+\fn.
\eeq

By \cite{Z25}, the $\omega$-symmetric bilinear $\omega$-trace form \eqref{eq:biform} on $\End_\C(V)$ satisfies
\beq\label{eq:form-basis}
(\BE_{a a},   \BE_{b b}) = \omega(\gamma_a, \gamma_b) \delta_{a b}, 
\eeq
thus is non-degenerate. It is also $ad$-invariant \cite{Z25}, i.e.,  
\[
( [X,  Y], Z ) = ( X,  [Y, Z] ), \quad \forall X, Y, Z\in \gl(V). 
\]
Its restriction $( \ , \ )|_\fh: \fh\times \fh\lra\C$ to the Cartan subalgebra is
non-degenerate and symmetric, thus induces a non-degenerate symmetric bilinear form 
on $\fh^*$,
\beq\label{eq:inner-prod}
(\ , \ ): \fh^*\times \fh^*\lra \C.
\eeq
It follows \cite[(3.1.12)]{Z25} that
\beq\label{eq:basis-prod}
(\varepsilon_a,  \varepsilon_b)=  \delta_{a b}. 
\eeq

Now $\fsl(V)$ inherits the following triangular decomposition from $\gl(V)$, 
\[
\fsl(V) =\ol{\fn}+ \fh_s+\fn, \text{  with $\fh_s= \fh\cap\fsl(V)$}.
\]
Note that $\fh_s$  is spanned by 
$\BE_{a a}- \BE_{b b}$ for all $a, b$.

We have already seen that each $\Upsilon\in\Phi^+$ 
is associated with a  Lie subalgebra $\gl_2(\C)$. 
This enables us to introduce a Weyl group.

\begin{definition}\label{def:Weyl}
Let $\fg$ be a semi-simple Lie colour algebra with the Borel and Cartan subalgebras 
$\fb=\fn+\fh\supset \fh$, where $\fh$ is homogeneous of degree $0$. 
Denote by $\Phi=\Phi^+\cup(-\Phi^+)$ its root system. 
The Weyl group $W$ of $\fg$ is the subgroup of the general linear group $\gl(E)$ for the vector space $E=\R \Phi^+$, 
which is generated by the reflections $\sigma_\Upsilon: E\lra E$, for all positive roots $\Upsilon\in\Phi^+$,  defined by 
\beq\label{eq:refl}
\sigma_\Upsilon(\mu)= \mu- \frac{2(\mu, \Upsilon)}{(\Upsilon, \Upsilon)} \Upsilon, \quad \forall \mu\in E. 
\eeq
\end{definition}

The  Weyl group  $W$ of $\fsl(V)$ is isomorphic to the symmetric group
$Sym_D$ of degree $D=\dim V$. The set of weights of any weight module \cite{Z25} of $\fsl(V)$ is $W$-stable. 
Definition \ref{def:Weyl} generalises to $\gl(V)$ by replacing $E$ by $\sum_a \R\varepsilon_a$.

\subsubsection{Generators and relations}
We take the following set of positive roots 
$
\Pi=\{\Upsilon_a:= \varepsilon_a - \varepsilon_{a+1}\mid a=1, 2, \dots, \dim V-1\}
$
as the simple roots for $\fsl(V)$. There is another piece of information associated with the root system, namely, the following sequence of elements of $\Gamma$.
\[
\Xi:=(\xi(\Upsilon_1), \xi(\Upsilon_2), \dots, \xi(\Upsilon_{\dim V-1})), \quad  \xi(\Upsilon_a)=\gamma_a-\gamma_{a+1}.
\]

\begin{remark}\label{rmk:Xi}
The sequence $\Xi$ depends on the chosen basis $B(V)$. 
\end{remark}

We define Cartan matrices and Dynkin diagrams as follows.  
\begin{definition}\label{def:Cartan}\label{def:Dynkin}
Let $\Phi$ be a root system with the set 
$\Pi=\{\Upsilon_a\mid a=1, 2, \dots, r\}$ of simple roots, 
and a sequence $\Xi=(\xi_1, \xi_2, \dots, \xi_r)$ of elements in $\Gamma$. 
The Cartan matrix $A=(A_{a b})_{a, b\in[1, r]}$ associated with $(\Phi, \Pi, \Xi)$ is the $r\times r$-matrix  with
$
A_{a b} =\frac{2(\Upsilon_a, \Upsilon_b)}{(\Upsilon_a, \Upsilon_a)}, 
$
and the Dynkin diagram associated with $(\Phi, \Pi, \Xi)$ is defined as follows.  
\begin{itemize}

\item 
Draw a node corresponding to each of the simple roots $\Upsilon_1,  \Upsilon_2, \dots, \Upsilon_r$, and order the nodes from left to right starting from that of $\Upsilon_1$.

\item  The $a$-th node is drawn 
as 
\begin{picture}(10, 10)(0, 0)
\setlength{\unitlength}{0.25mm}
\put(5, 5){\circle{10}}
\end{picture}
if $\xi(\Upsilon_a)=0$, 
and
\begin{picture}(10, 10)(0, 0)
\setlength{\unitlength}{0.25mm}
\put(5, 5){\circle{10}}
\put(5, 5){\circle{5}}
\end{picture}
 if 
$\xi(\Upsilon_a)\ne 0$.

\item  For any $a\ne b$, the corresponding nodes are connected by 
$A_{a b} A_{b a}$ lines. Furthermore, if  $|A_{a b}|<|A_{b a}|$, 
draw an arrow pointing to the $b$-th node.
\end{itemize}
\end{definition}

For $\fsl(V)$ with $(\Phi, \Pi, \Xi)$ given above, the Cartan matrix is
the usual one of type $A_{\dim V-1}$, and the Dynkin diagram is as in Figure \ref{fig:A}.
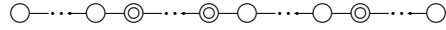
\begin{figure}[h]
\setlength{\unitlength}{0.25mm}
\begin{picture}(300, 20)(30, 0)
\put(50, 0){$\typeA$}
\put(100, 10){\line(1, 0){10}}
\put(110, 0){$\typeGA$}

\put(160, 10){\line(1, 0){10}}
\put(170, 0){$\typeA$}

\put(220, 10){\line(1, 0){10}}




\put(230, 00){$\typeA$}
\put(235, 10){\circle{5}}

\end{picture}
\caption{Dynkin diagram of type $A$}
\label{fig:A}
\end{figure}

The following result is easy to see. 
\begin{theorem}\label{thm:Serre-A}
For any $a=1, 2, \dots, \dim V-1$, let 
\[
X_a= \BE_{a,  a+1}, \quad Y_a= \BE_{a+1,  a}, \quad 
Z_a=\BE_{a  a}  - \omega(\gamma_a-\gamma_{a+1}, \gamma_a-\gamma_{a+1})\BE_{a+1,  a+1}. 
\]
These elements have $\Gamma$-degrees $d(X_a)=\xi(\Upsilon_a)$, $d(Y_a) =-\xi(\Upsilon_a)$, and $d(Z_a)=0$, respectively. They generate $\fsl(V)$ and satisfy the following relations
\beq
&&[Z_a, Z_b]=0, \label{eq:qud-1}\\
&&[Z_a, X_b]= A_{a b} X_b, \quad [Z_a, Y_b]= -A_{a b} Y_b, \label{eq:qud-2} \\
&&[X_a, Y_b]=\delta_{a b} Z_a, \quad \forall a, b, \label{eq:qud-3} \\
&& ad_{X_a}^{1-A_{a b}}(X_b)=0, \quad ad_{Y_a}^{1-A_{a b}}(Y_b)=0, \quad b\ne a. \label{eq:Serre-cl}
\eeq
\end{theorem}

As we have pointed out already, the choice of the ordered basis $B(V)$ for $V$ 
determines the set $\Xi$, and hence also the Dynkin diagram. 
The choices described below lead to the least number of non-zero elements in $\Xi$. 

Fix a total order of $\Gamma_R(V)$, 
and write $\Gamma_R(V)=\{\alpha_1, \alpha_2, \dots, \alpha_{\aleph}\}$ (where $\aleph=|\Gamma_R(V)|$)
in such a way that $\alpha_i < \alpha_j$ for all $i<j$.
Choose an ordered basis 
$
B(\alpha)=(e(\alpha)_1, e(\alpha)_2, \dots, e(\alpha)_{m_\alpha})
$
for each $V_\alpha$ with $m_\alpha=\dim V_{\alpha}$, 
and construct the following basis for $V$.  
\beq\label{eq:BV-s}
B(V)=(B(\alpha_1), B(\alpha_2), \dots, B(\alpha_{\aleph(V)})).
\eeq
Then the Dynkin diagram corresponding to this basis 
has the smallest number, $\aleph-1$, of double circles.  
Let $k_i=m_{\alpha_i}-1$. Then 
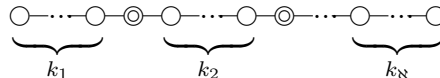
\[
\setlength{\unitlength}{0.25mm}
\begin{picture}(300, 35)(45, -20)
\put(50, 0){$\underbrace{\typeA}_{k_1}$}
\put(100, 10){\line(1, 0){10}}

\put(115, 10){\circle{10}}\put(115, 10){\circle{5}}

\put(120, 10){\line(1, 0){10}}
\put(130, 0){$\underbrace{\typeA}_{k_2}$}
\put(180, 10){\line(1, 0){10}}

\put(195, 10){\circle{10}}\put(195, 10){\circle{5}}

\put(200, 10){\line(1, 0){10}}
\put(210, 9){...}
\put(220, 10){\line(1, 0){10}}
\put(230, 0){$\underbrace{\typeA}_{k_{\aleph}}$}
\end{picture}
\]
However,  even this depends on the chosen order of the elements of $\Gamma_R(V)$. 

\subsection{Symplectic Lie colour algebras}\label{sect:C}

In the case $\Gamma_R(\fg)\subset \Gamma^-$, thus $V=V_-$, 
the orthosymplectic Lie colour superalgebra reduces to the
symplectic Lie colour algebra $\fsp(V; \kappa)$.  It follows Remark \ref{rmk:pairing} that 
if $\alpha \in \Gamma_R(V)$ but $\alpha \not\in \Gamma_R^{\Z_2}(V)$, 
then $V_\alpha$ is paired with $V_{-\alpha}$ under the non-degenerate bilinear form $\kappa$, thus 
$V_{-\alpha}$ is isomorphic to the dual space of $V_\alpha$. If $\alpha\in \Gamma_R^{\Z_2}(V)$, the restriction of $\kappa$ to $V_\alpha$ is non-degenerate. The $\omega$-symmetry of $\kappa$ in effect means skew symmetry on $V_\alpha$ since $V=V_-$.  Thus $\dim V_\alpha$ is even dimensional, and can be decomposed into $V_\alpha = U_\alpha \oplus \ol{U}_\alpha$ such that $\kappa$ pairs $U_\alpha$ with $\ol{U}_\alpha$, and hence $\ol{U}_\alpha$ is isomorphic to the dual space of $U_\alpha$. 
Therefore,  there exists $U=\sum_{\alpha\in \Gamma_R^-(U)} U_\alpha$ 
with  dual space $U^*= \sum_{\alpha\in \Gamma_R^-(U)} (U^*)_{-\alpha}$ such that  
$V= U\oplus U^*$. The bilinear form is now given,  for any homogeneous  $v\in U$ and $\ol{v}'\in U^*$, by
\[
\kappa(\ol{v}', v) = - \kappa(v, \ol{v}') = \ol{v}'(v). 
\]
Note that $\Gamma_R(U^*)=-\Gamma_R(U)\subseteq \Gamma^-$. 

Denote $\ell=\dim U$. 
Choose ordered  homogeneous bases 
\beq\label{eq:BU}
B(U)=(e_1, e_2, \dots, e_\ell) \text{  for $U$}, \quad 
\ol{B}(U^*)=(\ol{e}_1, \ol{e}_2, \dots, \ol{e}_\ell) \text{  for $U^*$},
\eeq 
such that $\ol{e}_i(e_j)=\delta_{i j}$, and  
introduce the ordered homogeneous basis for $V$, 
\beq\label{eq:BV-gen}
B(V)= (B(U), \ol{B}(U^*)). 
\eeq
Write $B(V)=(e_1, e_2, \dots, e_{2\ell})$. Then the elements $e_a$ satisfy
 $\kappa(e_{\ell+i}, e_j)=-\kappa(e_i, e_{\ell+j})$ $= \delta_{i j}$, and hence
 $(\kappa(e_a, e_b)) = \begin{pmatrix}
0 & - I_\ell\\ I_\ell & 0
\end{pmatrix}$.
Denote $\gamma_a=d(e_a)$ for $a=1, 2, \dots, 2\ell$, where $\gamma_i = - \gamma_{\ell+i}$ for all $i\le \ell$. 

The matrix units $\BE_{a b}\in \End_\C(V)$ with respect to $B(V)$ have $\Gamma$-degrees
\[
\baln
&d(\BE_{i j}) = d(\BE_{\ell+j, \ell+i}) = \gamma_i - \gamma_j, \
& d(\BE_{i, \ell+j}) =-d(\BE_{\ell+i, j})= \gamma_i + \gamma_j, \ 
& \forall i, j\le \ell. 
\ealn
\]
Using the explicit description of $\kappa$ given above, and bearing in mind that 
$\omega(\gamma_a, \gamma_a)=-1$ for all $a$, 
we can deduce the following result from Theorem \ref{thm:osp-def}. 
\begin{lemma}
Assume that $V=V_-$ and $\dim V=2\ell$. 
The Lie colour algebra $\fsp(V;\kappa)$ has the following basis.
\[
\baln
X_{i j}&:=\BE_{i j}+ \omega(\gamma_j, \gamma_i) \BE_{\ell+j, \ell+i}, \\
X_{i, \ell+j}&:=\BE_{i, \ell+j} - \omega(\gamma_i, \gamma_j)\BE_{j, \ell+i}, \\
X_{\ell+i, j}&:=\BE_{\ell+i, j} - \omega(\gamma_i, \gamma_j) \BE_{\ell+j, i}, 
\quad \forall i, j\le \ell. 
\ealn
\] 
The basis elements satisfy the following commutation relations.
\[
\baln{}
[X_{i j}, X_{p q}]&= \delta_{j p} X_{i  q}- \omega(\gamma_i-\gamma_j, \gamma_p-\gamma_q)\delta_{i q}X_{p j}, \\
[X_{i j}, X_{p, \ell+ q}]&= \delta_{j p} X_{i, \ell+q} +\omega(\gamma_i-\gamma_j, \gamma_p)
\delta_{j q}   X_{p, \ell+i}, \\
[X_{i j}, X_{\ell+ p, q}]&= \omega(\gamma_j, \gamma_i) (\delta_{i p} X_{\ell+j, q}+\delta_{i q} \omega(\gamma_p, \gamma_i- \gamma_j)X_{\ell+p, j})\\
[X_{i, \ell+j}, X_{\ell+p, q}]
 &= \delta_{j p} X_{i q}- \delta_{i p} \omega(\gamma_i, \gamma_j) X_{j q} \\
 &-  \omega(\gamma_p, \gamma_q)( \delta_{j q}  X_{i p}- \delta_{i q}\omega(\gamma_i, \gamma_j)  X_{j p}),\\
[X_{i, \ell+ j}, X_{p, \ell+ q}]&=0,   \quad [X_{\ell+ i, j}, X_{\ell+ p, q}]=0, \quad \forall i, j, p, q\le\ell.
\ealn
\]
\end{lemma}

Let $\fh=\sum_{i=1}^\ell \C\ X_{i i}$. It follows from the lemma that 
\[
\baln
{}[X_{i i}, X_{p p}]&=0,  \\
{}[X_{i i}, X_{p q}]&= (\delta_{i p} - \delta_{i q})X_{p q}; \quad p\ne q, \\
{}[X_{i i}, X_{p, \ell+ q}]&= (\delta_{i p}  + \delta_{i q})   X_{p, \ell+q}, \\
{}[X_{i i}, X_{\ell+ p, q}]&= - (\delta_{i p} +\delta_{i q}) X_{\ell+p, q}, \quad i, p, q\le \ell. 
\ealn
\]
Thus $\fh$ is a Cartan subalgebra of $\fsp(V;\kappa)$, which is homogeneous of degree $0$. 
The structure of $\fsp(V;\kappa)$ is most conveniently described 
in terms of its root system. 

Let $\delta_i$, for $i=1, 2, \dots, \ell$, be elements of the dual space $\fh^*$
such that $\delta_i(X_{j j})=\delta_{i j}$ for all $i, j=1, 2, \dots, \ell$. 
They form a basis of $\fh^*$.  The last three of the above equations can be re-written as 
\beq
{}[X_{i i}, X_{p q}]&=& (\delta_p - \delta_q)(X_{i i})X_{p  q}, \quad p\ne q, \label{eq:rt-1}\\
{}[X_{i i}, X_{p, \ell+ q}]&=& (\delta_p + \delta_q)(X_{i i})X_{p, \ell+q}, \label{eq:rt-2}\\
{}[X_{i i}, X_{\ell+ p, q}]&=& - (\delta_p + \delta_q)(X_{i i}) X_{\ell+p, q}, \quad i, p, q\le \ell. \label{eq:rt-3}
\eeq

Recall that the $\omega$-trace gives rise to a non-degenerate $\omega$-symmetric bilinear form on $\gl(V)$, whose restriction to $\fsp(V; \kappa)$ is non-degenerate. We define the bilinear form $(\ , \ ): \fsp(V; \kappa)\times \fsp(V; \kappa)\lra \C$ by $(X, Y)=\frac{1}{2}{\rm tr}_{(\Gamma, \omega)}(XY)$ for all $X, Y\in \fsp(V; \kappa)$.  
It is easy to verify that $(X_{i i}, X_{j j})=-\delta_{i j}$, thus the restriction of $(\ , \ )$ to $\fh$ is non-degenerate. It induces a symmetric bilinear form $(\ , \ ): \fh^*\times \fh^*\lra \C$ on $\fh^*$ in the usual way, which satisfies
\[
(\delta_i, \delta_j)=-\delta_{i j}, \quad i, j\in[1, \ell].
\]

It follows equations \eqref{eq:rt-1}- \eqref{eq:rt-3} that the set of roots of $\fsp(V;\kappa)$ relative to the Cartan subalgebra $\fh$ is given by
\[
\Phi=\{\delta_i-\delta_j\mid i, j\in [1,  \ell], i\ne j\}\cup \{\delta_p+\delta_q,  -(\delta_p+\delta_q)\mid  p, q\in [1,  \ell], p\le q\}.
\]

The map $\xi: \Phi \lra \Gamma$ (see \eqref{eq:xi-map}) in this case is defined by 
$\xi(\delta_i\pm \delta_j)= \gamma_i\pm \gamma_j$ for all $i<j$ and 
$\xi(2\delta_i) = 2\gamma_i$. Clearly $\xi(\Upsilon)\in \Gamma^+$ 
for all $\Upsilon\in \Phi$ as 
$\omega(\xi(\Upsilon), \xi(\Upsilon))=1$.

Note that associated with the root $\delta_i-\delta_j$, there is an
$\fsl_2(\C)$ subalgebra in $\fsp(V;\kappa)$ spanned by the elements 
$X_{i j}, X_{j i}, X_{i i}-X_{j j}$.  
Also, corresponding to the root $\delta_i+\delta_j$ for any fixed $i\le j$, we have 
\[
\baln
[X_{i, \ell+j}, X_{\ell+j, i}]&= (1+ \delta_{i j})  (X_{i i} + X_{j j}).
 \ealn
\]
Again the elements $X_{i, \ell+j}, X_{\ell+j, i}, X_{i i} +X_{j j}$ span an $\fsl_2(\C)$. 
Thus for each root $\Upsilon\in \Phi$, there is the reflection $\sigma_\Upsilon$ defined by \eqref{eq:refl}, 
which
generates the Weyl group of the $\fsl_2(\C)$ subalgebra associated with $\Upsilon$. 
The Weyl group $W$ of $\fsp(V; \kappa)$ is the subgroup of $\GL(E)$ 
generated by all $\sigma_\Upsilon$. 

The element $\hat\rho=\sum_{i=1}^\ell (\ell+1-i) X_{i i} \in \fh$ satisfies  
$\Upsilon(\hat\rho)\ne 0$  for all $\Upsilon\in \Phi$.  We use it to define the set 
$\Phi^+:=\{\Upsilon\in \Phi\mid \Upsilon(\hat\rho)>0\}$ of positive roots. We have   
\[
\baln
\Phi^+&= \{\delta_i-\delta_j, \delta_i+\delta_j, 2\delta_k\mid 1\le i<j\le \ell, \ 1\le k\le \ell\}.
\ealn
\] 
Then $\Phi=\Phi^+\cup(-\Phi^+)$. The set of simple roots for this choice of the positive roots is given by 
\[
\Pi=\{\Upsilon_i:=\delta_i-\delta_{i+1}, \Upsilon_\ell:=2\delta_\ell\mid 1\le i<\ell\}. 
\]
Let $\Xi=(\xi(\Upsilon_1), \xi(\Upsilon_2),\dots, \xi(\Upsilon_\ell))$. 
The Cartan matrix of $\fsp(V; \kappa)$ with the root datum $(\Phi, \Pi, \Xi)$ 
is the same as that of type $C_\ell$, and the Dynkin diagram is depicted in Figure \ref{fig:C}, 
where the node corresponds to $\Upsilon_i$ is a double circle if $\xi(\Upsilon_i)\ne 0$, and a circle otherwise. 
\begin{figure}[h]
\setlength{\unitlength}{0.25mm}
\begin{picture}(200, 25)(370, -5)

\put(305, 0){$\typeA$}

\put(355, 10){\line(1, 0){10}}

\put(365, 0){$\typeGA$}

\put(415, 10){\line(1, 0){10}}
\put(425, 0){$\typeA$}
\put(470, 10){\circle{5}}


\put(475, 10){\line(1, 0){10}}

\put(485, 0){$\typeA$}
\put(530, 10){\circle{5}}

\put(534, 12){\line(1, 0){20}}
\put(534, 8){\line(1, 0){20}}
\put(559, 10){\circle{10}}\put(559, 10){\circle{5}}
\put(540, 5){$<$}
\put(580, 5){if $\gamma_\ell\not\in\Gamma_{\Z_2}$,}
\end{picture}

\vspace{3mm}
\begin{picture}(200, 25)(370, -5)

\put(305, 0){$\typeA$}

\put(355, 10){\line(1, 0){10}}

\put(365, 0){$\typeGA$}

\put(415, 10){\line(1, 0){10}}
\put(425, 0){$\typeA$}
\put(470, 10){\circle{5}}


\put(475, 10){\line(1, 0){10}}

\put(485, 0){$\typeA$}
\put(530, 10){\circle{5}}

\put(534, 12){\line(1, 0){20}}
\put(534, 8){\line(1, 0){20}}
\put(559, 10){\circle{10}}
\put(540, 5){$<$}
\put(580, 5){if $\gamma_\ell\in\Gamma_{\Z_2}$.}

\end{picture}
\caption{Dynkin diagram of type $C$}
\label{fig:C}
\end{figure}

We have the following result. 

\begin{theorem}\label{thm:Serre-C}
Consider the following elements of $\fsp(V; \kappa)$, 
\[
\baln
&X_i= X_{i,  i+1}, \quad Y_i= X_{i+1,  i}, \quad 
Z_i=X_{i  i} - X_{i+1,  i+1}, \quad 1\le i <\ell, \\
&X_\ell=\frac{1}{2} X_{\ell,  2\ell}, \quad Y_\ell= \frac{1}{2}X_{2\ell, \ell}, \quad 
Z_\ell=X_{\ell \ell}, 
\ealn
\]
which have $\Gamma$-degrees $d(X_j)=\xi(\Upsilon_j)$, $d(Y_j) =-\xi(\Upsilon_j)$, and $d(Z_j)=0$ respectively. They generate $\fsp(V; \kappa)$, and satisfy relations which are formally the same as \eqref{eq:qud-1}, \eqref{eq:qud-2},
\eqref{eq:qud-3}, and \eqref{eq:Serre-cl}, but with 
$A=(A_{i j})$ being the usual Cartan matrix of type $C_\ell$, and $\Xi$ as described above. 
\end{theorem}

Let us now consider bases for $V$ which will lead to the least number of double circles in the Dynkin diagram.  
Denote  $\aleph=|\Gamma_R(U)|$. 
Fix a total order for $\Gamma_R(U)=\{\alpha_1, \alpha_2, \dots, \alpha_{\aleph(U)}\}$ with $\alpha_i< \alpha_{i+1}$ for all $i$. If $\Gamma_R^{\Z_2}(U)$ is not empty, we further require that  
$\alpha_s<\alpha_t$ for all $\alpha_s\not\in\Gamma_R^{\Z_2}(U)$ and  $\alpha_t\in\Gamma_R^{\Z_2}(U)$. 
Choose an ordered basis $B(\alpha_r)=(e(\alpha_r)_1, e(\alpha_r)_2, \dots e(\alpha_r)_{m_r})$  for each $U_{\alpha_r}$, where $m_r=\dim U_{\alpha_r}$. 
Let $\ol{B}(-\alpha_r)=(\ol{e}(-\alpha_r)_1, \ol{e}(-\alpha_r)_2, \dots, \ol{e}(-\alpha_r)_{m_r})$ be an ordered basis for $(U^*)_{-{\alpha_r}}$ such that $\ol{e}(-\alpha_r)_i(e(\alpha_r))_j=\delta_{i j}$.  
We have the following ordered bases for $U$,  $U^*$ and $V=U\oplus U^*$ respectively.
\beq
&B(U)=(B(\alpha_1), B(\alpha_2),\dots, B(\alpha_\aleph)), \label{eq:BU-s}\\
&\ol{B}(U^*)=(\ol{B}(-\alpha_1),  \ol{B}(-\alpha_2), \dots,  \ol{B}(-\alpha_\aleph)), \label{eq:BU-d} \\
&B(V)= (B(U), \ol{B}(U^*)). \label{eq:BV}
\eeq
Let $\ell_r=\dim U_{\alpha_r}-1$, for $i=1, 2, \dots, \aleph$. 
Then the Dynkin diagram associated with this basis of $V$ is one of the diagrams shown below, 
depending on whether $\alpha_{\aleph(U)}\in \Gamma_R^{\Z_2}(U)$ or not. 
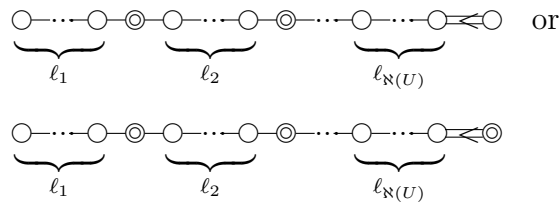
\[
\setlength{\unitlength}{0.25mm}
\begin{picture}(200, 90)(330, -15)

\put(305, 0){$\underbrace{\typeA}_{\ell_1}$}

\put(355, 10){\line(1, 0){10}}

\put(370, 10){\circle{10}}\put(370, 10){\circle{5}}

\put(375, 10){\line(1, 0){10}}
\put(385, 0){$\underbrace{\typeA}_{\ell_2}$}
\put(435, 10){\line(1, 0){10}}
\put(450, 10){\circle{10}}\put(450, 10){\circle{5}}
\put(455, 10){\line(1, 0){10}}
\put(465, 9){...}
\put(475, 10){\line(1, 0){10}}
\put(485, 0){$\underbrace{\typeA}_{\ell_{\aleph(U)}}$}
\put(534, 12){\line(1, 0){20}}
\put(534, 8){\line(1, 0){20}}
\put(559, 10){\circle{10}}\put(559, 10){\circle{5}}
\put(540, 5){$<$}


\put(305, 60){$\underbrace{\typeA}_{\ell_1}$}

\put(355, 70){\line(1, 0){10}}

\put(370, 70){\circle{10}}\put(370, 70){\circle{5}}

\put(375, 70){\line(1, 0){10}}
\put(385, 60){$\underbrace{\typeA}_{\ell_2}$}
\put(435, 70){\line(1, 0){10}}
\put(450, 70){\circle{10}}\put(450, 70){\circle{5}}
\put(455, 70){\line(1, 0){10}}
\put(465, 69){...}
\put(475, 70){\line(1, 0){10}}
\put(485, 60){$\underbrace{\typeA}_{\ell_{\aleph(U)}}$}
\put(534, 72){\line(1, 0){20}}
\put(534, 68){\line(1, 0){20}}
\put(559, 70){\circle{10}}
\put(540, 65){$<$}

\put(580, 65) {or}
\end{picture}
\]

\noindent
Again these are not unique as the basis \eqref{eq:BV} for $V$ is not unique.

\subsection{Orthogonal Lie colour algebras}\label{sect:D-B}
We consider  $\fso(V; \kappa)$, where $V=V_+$.  
As we will see, the structure of $\fso(V; \kappa)$ depends on the subspace $\sum_{\alpha\in \Gamma^+_{R, \Z_2}(V)} V_\alpha$ in a crucial way. 
For any $\alpha\in\Gamma^+_{R, \Z_2}(V)$, the restriction of $\kappa$ to $V_\alpha$ is non-degenerate by Remark \ref{rmk:pairing}.  Since it is symmetric in the present case,
$\dim V_\alpha$ can be even or odd. It is the odd dimensional ones that pose problems.  

\subsubsection{The case with even dimensional $V_\alpha$  
for all $\alpha\in\Gamma^+_{R, \Z_2}(V)$}\label{sect:D}

In this case, we again have 
$V= U\oplus U^*$ with $U=\sum_{\alpha\in \Gamma^+(U)} U_\alpha$, and $U^*= \sum_{\alpha\in \Gamma^+(U)} (U^*)_{-\alpha}$. The bilinear form is given by 
$
\kappa(\ol{v}', v) = \kappa(v, \ol{v}') = \ol{v}'(v)$ for all $v\in U, \ol{v}'\in U^*. 
$

Denote $\ell=\dim U$. 
We construct a homogeneous basis $
B(V)=(e_1, e_2, \dots, e_{2\ell})
$ for $V$
in the same way as described by \eqref{eq:BU},  
\eqref{eq:BV-gen}.   Then
$(\kappa(e_a, e_b)) = \begin{pmatrix}
0 & I_\ell\\ I_\ell & 0
\end{pmatrix}$.
We have $\omega(\gamma_a, \gamma_a)=1$ for all $a=1, 2, \dots, 2\ell$. 

The following result is an easy consequence of Theorem \ref{thm:osp-def}. 
\begin{lemma} 
The Lie colour algebra $\fso(V;\kappa)$ 
has a homogeneous basis consisting of the elements
\[
\baln
X_{i j}&:=\BE_{i j}- \omega(\gamma_j, \gamma_i) \BE_{\ell+j, \ell+i}, \quad \forall i, j\le \ell,  \\
X_{i, \ell+j}&:=\BE_{i, \ell+j}- \omega(\gamma_i, \gamma_j)\BE_{j, \ell+i}, \\
X_{\ell+i, j}&:=\BE_{\ell+i, j}- \omega(\gamma_i, \gamma_j) \BE_{\ell+j, i}, 
\quad \forall i < j\le \ell, 
\ealn
\] 
which satisfy the following commutation relations for all valid $i, j, p, q$.
\[
\baln
{}[X_{i j}, X_{p q}]&= \delta_{j p} X_{i  q}- \omega(\gamma_i-\gamma_j, \gamma_p-\gamma_q)\delta_{i q}X_{p j}, \\
[X_{i j}, X_{p, \ell+ q}]&= \delta_{j p} X_{i, \ell+q} +\omega(\gamma_i-\gamma_j, \gamma_p)
\delta_{j q}   X_{p, \ell+i}, \\
[X_{i j}, X_{\ell+ p, q}]&=- \omega(\gamma_j, \gamma_i) (\delta_{i p} X_{\ell+j, q}+\delta_{i q} \omega(\gamma_p, \gamma_i- \gamma_j)X_{\ell+p, j}), \\
[X_{i, \ell+j}, X_{\ell+p, q}]
&= \delta_{j p} X_{i q}- \delta_{i p} \omega(\gamma_i, \gamma_j) X_{j q} \\
 &-  \omega(\gamma_p, \gamma_q)( \delta_{j q}  X_{i p}- \delta_{i q}\omega(\gamma_i, \gamma_j)  X_{j p}), \\
[X_{i, \ell+ j}, X_{p, \ell+ q}]&=0,  \quad [X_{\ell+ i, j}, X_{\ell+ p, q}]=0.
\ealn
\]
\end{lemma}

Let $\fh=\sum_{i=1}^\ell \C\ X_{i i}$, which is a homogeneous Cartan subalgebra of degree $0$ for $\fso(V;\kappa)$ (thus is an ordinary abelian Lie algebra). It follows from the lemma that 
\[
\baln
{}[X_{i i}, X_{p p}]&=0,  \\
{}[X_{i i}, X_{p q}]&= (\delta_{i p} - \delta_{i q})X_{p q},  \\
{}[X_{i i}, X_{p, \ell+ q}]&= (\delta_{i p}  + \delta_{i q})   X_{p, \ell+q}, \\
{}[X_{i i}, X_{\ell+ p, q}]&= - (\delta_{i p} +\delta_{i q}) X_{\ell+p, q}, \quad i, p, q\le \ell, \ 
p\ne q.
\ealn
\]

Let $\varepsilon_i$, for $i=1, 2, \dots, \ell$, be elements of the dual space $\fh^*$
such that $\varepsilon_i(X_{j j})=\delta_{i j}$ for all $i, j=1, 2, \dots, \ell$. 
They form a basis of $\fh^*$.  The last three of the above equations can be re-written as 
\beq
{}[X_{i i}, X_{p q}]&=& (\varepsilon_p - \varepsilon_q)(X_{i i})X_{p  q},  \label{eq:ort-1}\\
{}[X_{i i}, X_{p, \ell+ q}]&=& (\varepsilon_p + \varepsilon_q)(X_{i i})X_{p, \ell+q}, \label{eq:ort-2}\\
{}[X_{i i}, X_{\ell+ p, q}]&=& - (\varepsilon_p + \varepsilon_q)(X_{i i}) X_{\ell+p, q}, \quad i, p, q\le \ell, \  p\ne q. \label{eq:ort-3}
\eeq

The $\omega$-trace gives rise to a non-degenerate $\omega$-symmetric bilinear form  $(\ , \ ): \fso(V; \kappa)\times \fso(V; \kappa)\lra \C$ on $\fso(V; \kappa)$, which is defined by $(X, Y)=\frac{1}{2}{\rm tr}_{(\Gamma, \omega)}(XY)$ for all $X, Y\in \fsp(V; \kappa)$.  
Its restriction to $\fh$ is non-degenerate, and satisfies $(X_{i i}, X_{j j})=\delta_{i j}$.
It induces a symmetric bilinear form $(\ , \ ): \fh^*\times \fh^*\lra \C$ on $\fh^*$ in the usual way, which satisfies
$(\varepsilon_i, \varepsilon_j)=\delta_{i j}$, for all 
$i, j\in[1, \ell].$

The set of roots of $\fso(V;\kappa)$ relative to the Cartan subalgebra $\fh$ can be read off the equations \eqref{eq:rt-1} - \eqref{eq:rt-2}: 
\[
\Phi=\{\varepsilon_i-\varepsilon_j, \varepsilon_i+\varepsilon_j \mid i, j\in [1,  \ell], i\ne j\}.
\]

Each root $\varepsilon_i-\varepsilon_j$ for $i<j$ is associated with an $\fsl_(\C)$ Lie subalgebra spanned by the elements $X_{i j}$, $X_{j i}$,  and $X_{i  i}-X_{j j}$.
Similarly, each root $\varepsilon_i+\varepsilon_j$ with $i\ne j$ is associated with an $\fsl_(\C)$ Lie subalgebra spanned by the elements $X_{i, \ell+j}, X_{\ell+j, i}$,  and $\wt{H}_{i j}:=X_{i  i}+X_{j j}$, which satisfy the commutation relations. 
\[
[X_{i, \ell+j}, X_{\ell+j, i}]= \wt{H}_{i j}, \quad [\wt{H}_{i j}, X_{i, \ell+j}] = 2 X_{i, \ell+j}, 
\quad [\wt{H}_{i j}, X_{\ell+j, i}]=-2 X_{\ell+j, i}.         
\]
Thus the reflection $\sigma_\Upsilon$ defined by \eqref{eq:refl} for each root $\Upsilon\in \Phi$ 
generates the Weyl group of the corresponding $\fsl_2(\C)$ subalgebra. 
The Weyl group $W$ of $\fso(V; \kappa)$ is the subgroup of $\GL(E)$  
generated by all $\sigma_\Upsilon$, where $E=\R\Phi$.

The element $\hat\rho=\sum_{i=1}^\ell (\ell+1-i) X_{i i} \in \fh$ can again be used to define 
positive roots: $\Phi^+:=\{\Upsilon\in \Phi\mid \Upsilon(\hat\rho)>0\}$.  We have 
\[
\baln
\Phi^+&= \{\varepsilon_i-\varepsilon_j, \varepsilon_i+\varepsilon_j\mid 1\le i<j\le \ell\}.
\ealn
\] 
Then the set of simple roots is given by 
\[
\Pi=\{\Upsilon_i:=\varepsilon_i-\varepsilon_{i+1}, \Upsilon_\ell:=\varepsilon_{\ell-1}+\varepsilon_\ell\mid 1\le i<\ell\}, 
\]
and we let $\Xi=(\xi(\Upsilon_1),  \xi(\Upsilon_2),\dots, \xi(\Upsilon_\ell))$. 
The Cartan matrix is the same as that of type $D_\ell$, and the  
Dynkin diagram is as depicted in Figure \ref{fig:D}, where the $i$-th node is a circle (resp. double circle) if $\xi(\Upsilon_i)=0$ (resp. $\xi(\Upsilon_i)\ne 0$).  
\begin{figure}[h]
\setlength{\unitlength}{0.25mm}
\begin{picture}(200, 40)(330, -10)

\put(305, 0){$\typeA$}

\put(355, 10){\line(1, 0){10}}
\put(365, 0){$\typeGA$}
\put(415, 10){\line(1, 0){10}}
\put(425, 0){$\typeA$}
\put(475, 10){\line(1, 0){10}}
\put(485, 0){$\typeGA$}

\put(534, 12){\line(2, 1){20}}
\put(534, 8){\line(2, -1){20}}
\put(559, 23){\circle{10}}
\put(559, -3){\circle{10}}\put(559, -3){\circle{5}}
\end{picture}
\caption{Dynkin diagram of type $D$}
\label{fig:D}
\end{figure}

\begin{theorem}\label{thm:Serre-D}
Consider the following elements of $\fso(V; \kappa)$, 
\[
\baln
&X_i= X_{i,  i+1}, \quad Y_i= X_{i+1,  i}, \quad 
Z_i=X_{i  i} - X_{i+1,  i+1}, \quad 1\le i <\ell, \\
&X_\ell= X_{\ell-1,  2\ell}, \quad Y_\ell= X_{2\ell, \ell-1}, \quad 
Z_\ell=X_{\ell-1,  \ell-1}+X_{\ell  \ell}, 
\ealn
\]
which have $\Gamma$-degrees $d(X_j)=\xi(\Upsilon_j)$, $d(Y_j) =-\xi(\Upsilon_j)$, and $d(Z_j)=0$ respectively. They generate $\fso(V; \kappa)$, and satisfy relations formally the same as \eqref{eq:qud-1}, \eqref{eq:qud-2},
\eqref{eq:qud-3}, and \eqref{eq:Serre-cl}, but with 
$A=(A_{i j})$ being the usual Cartan matrix of $D_\ell$, and $\Xi$ as described above. 
\end{theorem}

Order the elements of $\Gamma_R(U)$ in the same way  as in Section \ref{sect:C}. 
Let 
\beq\label{eq:k-D}
k_r=\dim U_{\alpha_r}-1, \quad r=1, 2, \dots, \aleph(U). 
\eeq
Then the Dynkin diagram is 
given by one of the two diagrams in 
Figure \ref{fig:D}, depending on whether 
$\alpha_{\aleph(U)}\in\Gamma_{\Z_2}$.
Note that if $k_{\aleph(U)}=0$, the last two nodes in the Dynkin diagram are connected to a double circle.

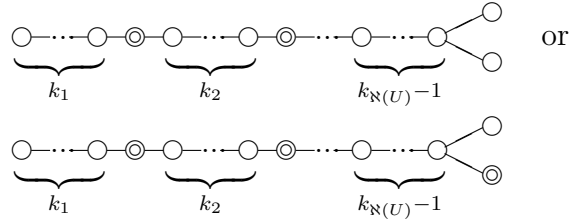
\[
\setlength{\unitlength}{0.25mm}
\begin{picture}(200, 90)(330, -10)

\put(305, 0){$\underbrace{\typeA}_{k_1}$}

\put(355, 10){\line(1, 0){10}}

\put(370, 10){\circle{10}}\put(370, 10){\circle{5}}

\put(375, 10){\line(1, 0){10}}
\put(385, 0){$\underbrace{\typeA}_{k_2}$}
\put(435, 10){\line(1, 0){10}}
\put(450, 10){\circle{10}}\put(450, 10){\circle{5}}
\put(455, 10){\line(1, 0){10}}
\put(465, 9){...}
\put(475, 10){\line(1, 0){10}}
\put(485, 0){$\underbrace{\typeA}_{k_{\aleph(U)}-1}$}
\put(534, 12){\line(2, 1){20}}
\put(534, 8){\line(2, -1){20}}
\put(559, 23){\circle{10}}
\put(559, -3){\circle{10}}\put(559, -3){\circle{5}}


\put(305, 60){$\underbrace{\typeA}_{k_1}$}

\put(355, 70){\line(1, 0){10}}

\put(370, 70){\circle{10}}\put(370, 70){\circle{5}}

\put(375, 70){\line(1, 0){10}}
\put(385, 60){$\underbrace{\typeA}_{k_2}$}
\put(435, 70){\line(1, 0){10}}
\put(450, 70){\circle{10}}\put(450, 70){\circle{5}}
\put(455, 70){\line(1, 0){10}}
\put(465, 69){...}
\put(475, 70){\line(1, 0){10}}
\put(485, 60){$\underbrace{\typeA}_{k_{\aleph(U)}-1}$}
\put(534, 72){\line(2, 1){20}}
\put(534, 68){\line(2, -1){20}}

\put(559, 83){\circle{10}}
\put(559, 57){\circle{10}} 

\put(585, 66){or} 

\end{picture}
\]

\subsubsection{The case with one odd dimensional $V_\eta$ for $\eta\in\Gamma^+_{R, \Z_2}(V)$}
\label{sect:B}
We assume that there is an $\eta\in \Gamma^+_{R, \Z_2}(V)$ 
such that $\dim V_\eta$ is odd, and  
$\dim V_\alpha$ is even if $\eta\ne \alpha\in \Gamma^+_{R, \Z_2}(V)$. 
Then  $V_\eta=U_\eta \oplus 
U_\eta^*\oplus E_\eta$, where $E_\eta =\C e(\eta)$ is a $1$-dimensional homogeneous subspace 
of degree $\eta$ with basis $e(\eta)$, and $U_\eta$ is a homogeneous subspace 
of degree $\eta$ with dual space $U_\eta^*$. 
Now $V= U\oplus U^*\oplus E_\eta$, where $U$ and its dual space $U^*$ are the same as in Section \ref{sect:D}, where $U_\eta\subset U$ and $U_\eta^*\subset U^*$. 
The bilinear form $\kappa$ is given by 
$\kappa(v+\ol{v}'+ a e(\eta), u + \ol{u}'+ b e(\eta)) = \ol{v}'(u)+ \ol{u}'(v) + a b$ 
for all $u, v\in U, \ol{u}', \ol{v}'\in U^*$ and $a, b\in \C.$

Denote $\ell=\dim U$, thus $\dim V=2\dim U +1=2\ell+1$.
We have the ordered bases $B(U)$ and $\ol{B}(U^*)$ for $U$ and $U^*$ respectively, as defined by \eqref{eq:BU}. Now we introduce the following ordered basis for $V$.
\beq\label{eq:B-odd}
B(V)=(e_1, e_2, \dots, e_{2\ell+1}) := (B(U),  e(\eta), \ol{B}(U^*)).
\eeq
Then the bilinear form $\kappa$ is given by
$(\kappa(e_a, e_b)) = \begin{pmatrix}
0 & 0 & I_\ell\\ 0 & 1 & 0 \\ I_\ell & 0 &0
\end{pmatrix}$. 

\begin{lemma}
The Lie colour algebra $\fso(V;\kappa)$ has a basis consisting of the elements
\[
\baln
X_{i j}&:=\BE_{i j}- \omega(\gamma_j, \gamma_i) \BE_{\ell+1+j, \ell+1+i}, \quad \forall i, j=1, 2, \dots,  \ell,  \\
X_{i, \ell+1+j}&:=\BE_{i, \ell+1+j}- \omega(\gamma_i, \gamma_j)\BE_{j, \ell+1+i}, \\
X_{\ell+1+i, j}&:=\BE_{\ell+1+i, j}- \omega(\gamma_i, \gamma_j) \BE_{\ell+1+j, i}, 
\quad \forall i, j=1, 2, \dots,  \ell, \ i < j, \\
X_{i, \ell+1}&:=\BE_{i, \ell+1}- \BE_{\ell+1, \ell+1+i}, \\
X_{\ell+1, i}&:=\BE_{\ell+1, i}- \BE_{\ell+1+i, \ell+1}, 
\quad \forall i=1, 2, \dots,  \ell, 
\ealn
\] 
which satisfy the following relations for all valid $i, j, p, q=1, 2, \dots,  \ell$. 
\[
\baln
{}[X_{i j}, X_{p q}]&= \delta_{j p} X_{i  q}- \omega(\gamma_i-\gamma_j, \gamma_p-\gamma_q)\delta_{i q}X_{p j}, \\
[X_{i j}, X_{p, \ell+1+ q}]&= \delta_{j p} X_{i, \ell+1+q} +\omega(\gamma_i-\gamma_j, \gamma_p)
\delta_{j q}   X_{p, \ell+1+i}, \\
[X_{i j}, X_{\ell+1+ p, q}]&=- \omega(\gamma_j, \gamma_i) (\delta_{i p} X_{\ell+1+j, q}+\delta_{i q} \omega(\gamma_p, \gamma_i- \gamma_j)X_{\ell+1+p, j}), \\
[X_{i, \ell+1+j}, X_{\ell+1+p, q}]
&= \delta_{j p} X_{i q}- \delta_{i p} \omega(\gamma_i, \gamma_j) X_{j q} \\
 &-  \omega(\gamma_p, \gamma_q)( \delta_{j q}  X_{i p}- \delta_{i q}\omega(\gamma_i, \gamma_j)  X_{j p}), \\
[X_{i, \ell+1+ j}, X_{p, \ell+1+ q}]&=0,  \quad [X_{\ell+1+ i, j}, X_{\ell+1+ p, q}]=0, \\
[X_{i j}, X_{p, \ell+1}]&= \delta_{j p} X_{i, \ell+1}, \quad 
[X_{i j}, X_{\ell+1, p}]= -\omega(\gamma_j, \gamma_i) \delta_{i p}X_{\ell+1, j}, \\
[X_{\ell+1+i, j}, X_{p, \ell+1}]&=- \delta_{j p}X_{\ell+1, i} + \omega(\gamma_i, \gamma_j) \delta_{i p} X_{\ell+1, j}, \\
[X_{i, \ell+1+ j} ,   X_{\ell+1, p}]&= -\delta_{j p}X_{i, \ell+1} + \omega(\gamma_i, \gamma_j) \delta_{i p}X_{j, \ell+1}, \\
[X_{i, \ell+1+ j}, X_{p, \ell+1}]&=0, \quad [X_{\ell+1+i, j}, X_{\ell+1, p}]=0, \\
[X_{i, \ell+1}, X_{\ell+1, j}]&=X_{i j}.
\ealn
\]
\end{lemma}
It is easy to extract the following relations from the lemma. 
\[
\baln
{}[X_{i i}, X_{p q}]&= (\delta_{i p} -\delta_{i q})X_{p  q}, \\
[X_{i i}, X_{p, \ell+1+ q}]&= (\delta_{i p} +\delta_{i q})X_{p, \ell+1+q}, \\
[X_{i i}, X_{\ell+1+ p, q}]&=-  (\delta_{i p}+\delta_{i q})X_{\ell+1+p, q}, \\
[X_{i i}, X_{p, \ell+1}]&= \delta_{i p} X_{p, \ell+1}, \\
[X_{i i}, X_{\ell+1, p}]&= -\delta_{i p}X_{\ell+1, p}, \quad \forall i, p, q=1, 2, \dots,  \ell.
\ealn
\]
Thus $\fh=\sum_{i=1}^\ell \C X_{i i}$ forms a Cartan subalgebra, 
which is homogeneous of degree $0$.

Let $\{\varepsilon_i\mid i=1, 2, \dots, \ell\}$ be a basis of the dual space $\fh^*$
such that $\varepsilon_i(X_{j j})=\delta_{i j}$ for all $i, j=1, 2, \dots, \ell$. 
It follows the above equations that the set of roots of $\fso(V; \kappa)$ with respect to the 
Cartan subalgebra $\fh$ is given by
\[
\Phi = \{\pm(\varepsilon_i-\varepsilon_j), \pm (\varepsilon_i+\varepsilon_j)\mid i, j\in [1, \ell], i\le j\}\cup 
\{\pm \varepsilon_i \mid i\in [1, \ell]\}. 
\]

The $\omega$-trace yields a non-degenerate $\omega$-symmetric bilinear form  $(\ , \ ): \fso(V; \kappa)\times \fso(V; \kappa)\lra \C$ on $\fso(V; \kappa)$, whose   
restriction to $\fh$ is non-degenerate and 
 induces a symmetric bilinear form $(\ , \ ): \fh^*\times \fh^*\lra \C$ on $\fh^*$ such that
$
(\varepsilon_i, \varepsilon_j)=\delta_{i j}$.

There is an $\fsl_2(\C)$ subalgebra corresponding to each of the roots $\varepsilon_i\pm \varepsilon_j$ for $i<j$. 
Also, corresponding to each $\varepsilon_i$, the elements $X_{i, \ell+1}$, $X_{\ell+1, i}$ and $X_{i i}$ span an $\fso_3(\C)\simeq \fsl_(\C)$ subalgebra, with the commutation relations 
\[
[X_{i, \ell+1}, X_{\ell+1, i}]=X_{i i}, \quad 
[X_{i i}, X_{i, \ell+1}]=X_{i, \ell+1}, \quad 
[X_{i i}, X_{\ell+1, i}]=-X_{\ell+1, i}.
\]
Thus the reflection $\sigma_\Upsilon$ (cf. \eqref{eq:refl}) for each $\Upsilon\in \Phi^+$ 
generates the Weyl group of the associated $\fsl_2(\C)$ subalgebra. 
The Weyl group $W$ of $\fso(V; \kappa)$ is generated by all $\sigma_\Upsilon$. 

We will take the set of simple roots 
\[
\Pi = \{ \Upsilon_i = \varepsilon_i- \varepsilon_{i+1}, \Upsilon_\ell= \varepsilon_\ell \mid 1\le i<\ell\}, 
\] 
and let $\Xi=(\xi(\Upsilon_1),  \xi(\Upsilon_2),\dots, \xi(\Upsilon_\ell))$. 
Then the Cartan matrix is the standard one for $B_\ell$, and the corresponding Dynkin diagram is as depicted in Figure \ref{fig:B}.

\begin{figure}[h]
\setlength{\unitlength}{0.25mm}
\begin{picture}(300, 20)(340, 0)

\put(305, 0){$\typeA$}

\put(355, 10){\line(1, 0){10}}

\put(370, 10){\circle{10}}\put(370, 10){\circle{5}}

\put(375, 10){\line(1, 0){10}}
\put(385, 0){$\typeA$}
\put(390, 10){\circle{5}}
\put(435, 10){\line(1, 0){10}}
\put(450, 10){\circle{10}}\put(450, 10){\circle{5}}
\put(455, 10){\line(1, 0){10}}
\put(465, 9){...}
\put(475, 10){\line(1, 0){10}}
\put(485, 0){$\typeA$}
\put(490, 10){\circle{5}}

\put(534, 12){\line(1, 0){20}}
\put(534, 8){\line(1, 0){20}}
\put(559, 10){\circle{10}}\put(559, 10){\circle{5}}
\put(540, 5){$>$}
\put(600, 5){if $\xi(\Upsilon_\ell)\ne 0$;}

\end{picture}

\vspace{5mm}

\begin{picture}(300, 20)(340, 0)

\put(305, 0){$\typeA$}

\put(355, 10){\line(1, 0){10}}

\put(370, 10){\circle{10}}\put(370, 10){\circle{5}}

\put(375, 10){\line(1, 0){10}}
\put(385, 0){$\typeA$}
\put(390, 10){\circle{5}}
\put(435, 10){\line(1, 0){10}}
\put(450, 10){\circle{10}}\put(450, 10){\circle{5}}
\put(455, 10){\line(1, 0){10}}
\put(465, 9){...}
\put(475, 10){\line(1, 0){10}}
\put(485, 0){$\typeA$}
\put(490, 10){\circle{5}}

\put(534, 12){\line(1, 0){20}}
\put(534, 8){\line(1, 0){20}}
\put(559, 10){\circle{10}}
\put(540, 5){$>$}
\put(600, 5){if $\xi(\Upsilon_\ell)= 0$.}

\end{picture}

\caption{Dynkin diagrams of type $B$}
\label{fig:B}
\end{figure}
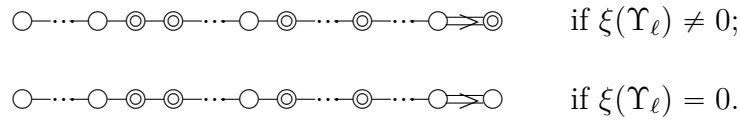

\begin{theorem}\label{thm:Serre-B}
Consider the following elements of $\fso(V; \kappa)$, 
\[
\baln
&X_i= X_{i,  i+1}, \quad Y_i= X_{i+1,  i}, \quad 
Z_i=X_{i  i} - X_{i+1,  i+1}, \quad 1\le i <\ell, \\
&X_\ell= 2X_{\ell,  \ell+1}, \quad Y_\ell= X_{\ell+1, \ell}, \quad 
Z_\ell=X_{\ell \ell}, 
\ealn
\]
which have $\Gamma$-degrees $d(X_j)=\xi(\Upsilon_j)$, $d(Y_j) =-\xi(\Upsilon_j)$, and $d(Z_j)=0$ respectively. They generate $\fso(V; \kappa)$, and satisfy  relations formally the same  as \eqref{eq:qud-1}, \eqref{eq:qud-2},
\eqref{eq:qud-3}, and \eqref{eq:Serre-cl}, but with 
$A=(A_{i j})$ being the usual Cartan matrix of type $B_\ell$, and $\Xi$ as described above.
\end{theorem}

Let us consider bases of $V$ which will lead to the least number of double circles in the Dynkin diagram. 
Order the elements of $\Gamma_R(U)=\{\alpha_1, \alpha_2, \dots, \alpha_{\aleph}\}$ so that 
\begin{enumerate}[label=\roman*)]
\item $\alpha_i<\alpha_j$ if $i<j$;   

\item $\alpha_s< \alpha_t$ if $\alpha_s\not\in\Gamma^+_{R, \Z_2}(U)$ and $\alpha_t\in\Gamma^+_{R, \Z_2}(U)$; and 

\item $\eta$ is minimal. 
\end{enumerate}
Construct bases $B(U)$ for $U$ and $\ol{B}(U^*)$ for $U^*$ as those given by \eqref{eq:BU-s} and \eqref{eq:BU-d}. We again take $B(V)= (B(U), e(\eta), \ol{B}(U^*))$. Then the Dynkin diagram corresponding 
to this basis has $\aleph$ double circles if $\eta\ne 0$, and $\aleph-1$ double circles if $\eta= 0$. Note that in the latter case, the last node in Figure \ref{fig:B} should be a circle.

\subsubsection{Failure of Cartan--Weyl paradigm} 
If there are more than one $\alpha\in\Gamma_{\Z_2}$ such that $V_\alpha$ are odd dimensional, 
 $\fso(V; \kappa)$ fails the Cartan--Weyl paradigm.  

This is very easy to see in the example below.

\begin{example}\label{eg:CW-fail}
Assume that $\dim V=r\ge 2$ and $\Gamma_R(V) =\Gamma^+_{R, \Z_2}(V)$ with $\aleph(V)=r$.  Denote by $\gamma_1, \dots, \gamma_r$ the (distinct) elements of $\Gamma^+_{R, \Z_2}(V)$. Then $V=\sum_{i=1}^r V_{\gamma_i}$ with $\dim V_{\gamma_i}=1$ for all $i$. 
The bilinear form $\kappa$ is diagonal relative to any homogeneous basis of $V$. 
It then follows \eqref{eq:Xab} that
$\fso_r(\gamma_1, \dots, \gamma_r):=\fso(V; \kappa)$ has a homogeneous basis 
$
X_{i j}= \BE_{i j} - \omega(\gamma_{j}, \gamma_{i}) \BE_{j i}
$ for $i<j$. 
It is evident that $\fso_r(\gamma_1, \dots, \gamma_r)$ 
has no non-zero homogeneous element of degree $0$.  
\end{example}

Consider the general case with elements $\gamma_i\in\Gamma^+_{R, \Z_2}(V)$, for $i=1, 2, \dots, r\ge 2$, such that $V_{\gamma_i}$ are odd dimensional. We always have $V=V^{0}\oplus V^{1}$, where $V^{1}$ is the sum of homogeneous $1$-dimensional subspaces of degrees $\gamma_i$ for $i=1, 2, \dots, r$ respectively,  
such that $V^0$ and $V^1$ are orthogonal with respect to $\kappa$.
%
Then $\fso(V; \kappa)$ decomposes into the direct sum of the subalgebra 
$\fso(V^0; \kappa|_{V^0}) \oplus \fso(V^1; \kappa|_{V^1})$ and the module
$V^0\wedge_\omega V^1$ for it. 
We can choose  homogeneous bases for  $V^0$ and $V^1$ such that 
$\kappa=\begin{pmatrix}
0 & I_\ell & 0\\ I_\ell & 0 &0\\
0 & 0 &I_r
\end{pmatrix}$ with $\ell=\frac{\dim V- r}{2}$.
Now $\fso(V^0; \kappa|_{V^0})$ fulfils the Cartan--Weyl paradigm by Section \ref{sect:D},  but 
$\fso(V^1; \kappa|_{V^1})\simeq \fso_r(\gamma_1, \dots, \gamma_r)$ fails as discussed in the example above.
This implies that $\fso(V; \kappa)$ fails the Cartan--Weyl paradigm.  

As further illustration, let us consider the $r=3$ case of Example \ref{eg:CW-fail}. 
\begin{example}\label{eg:Jordan}
For $r=3$, there is a particularly neat case with $\omega(\gamma_i, \gamma_j)=-1$ for all $i\ne j$. 
%
%
%
The universal enveloping colour algebra $\U(\fso_3(\gamma_1, \gamma_2, \gamma_3))$ is generated by $X_{12},  X_{23}, X_{1 3}$ with the following defining relations
\beq
X_{12} X_{23}+  X_{23} X_{12} &=&X_{1 3}, \label{eq:Jordan-1}\\
X_{13}X_{12}+X_{12} X_{13} &=& X_{23},  \label{eq:Jordan-2}\\
X_{13} X_{2 3} + X_{2 3} X_{13} &=&X_{1 2}. \label{eq:Jordan-3}
\eeq
Even allowing inhomogeneous basis, one still can not transform them into a form close to 
Lie (super)algebras such as  $\U(\fso_3(\C))$ or $\U(\osp_{1|2}(\C))$. 
For example, take $X_\pm :=\frac{1}{2}( X_{13} \pm X_{23})$ and $X_0:=X_{12}$,
then the relations become
\[
X_0 X_\pm + X_\pm X_0= \pm X_\pm, \quad
X_+^2 - X_-^2 = 2 X_0, 
\]
where  the last one is very different from relations of the above  Lie (super)algebras. 

In fact, one can see from equations \eqref{eq:Jordan-1}, \eqref{eq:Jordan-2} and  \eqref{eq:Jordan-3} that $\U(\fso_3(\gamma_1, \gamma_2, \gamma_3))$ is 
a special Jordan algebra with the Jordan product $X\circ Y= \frac{1}{2}(X Y + Y X)$. 
\end{example}

\begin{remark}\label{rmk:change}
For the simple Lie colour (super)algebras and their affine analogues which fail the Cartan--Weyl paradigm,  
a new conceptual framework may be required to treat their structure and representation theory. 
Example \ref{eg:Jordan} seems to indicate that Jordan algebra techniques 
could potentially come to play. 
\end {remark}

\begin{remark}\label{rmk:intricacy}
Lie colour superalgebras admit graded Casimir operators 
and central extensions of non-zero degrees as shown in \cite{AFSV}.  
This is further evidence that Lie colour superalgebras are richer 
than their cocycle twisted Lie superalgebras. 
\end {remark}

\subsection{Untwisted affine Lie colour algebras}\label{sect:loop}
We note that examples of affine Lie colour (super)algebras were studied in \cite{AFSV, AIS,  AS}. 

Consider the Laurent polynomial ring $\C[t, t^{-1}]$ in the variable $t$ as a $\Gamma$-graded vector space which is homogeneous of degree $0$. 
Given a finite dimensional Lie colour algebra $\fg$, we let $\CL(\fg)=\fg\ot\C[t, t^{-1}]$ as a $\Gamma$-graded vector space. The loop Lie colour algebra of $\fg$ is $\CL(\fg)$ with the generalised Lie bracket being the bilinear map defined by  
\beq
[X\ot f, Y\ot g]=[X, Y]\ot fg, \quad X, Y\in\fg,\  f, g\in\C[t, t^{-1}]. 
\eeq
The generalised Jacobian identity of $\CL(\fg)$ follows that of $\fg$. 

Assume that there is a $\omega$-symmetric bilinear form $\kappa$ on $\fg$, which is homogeneous of degree $0$, and is ad-invariant.
The untwisted affine Lie colour algebra $\wh\fg = \CL(\fg) +\C c$ of $\fg$ is the central extension of 
the loop algebra $\CL(\fg)$
with a generalised Lie bracket such that
for any $X, Y\in\fg$ and $f, g\in\C[t, t^{-1}]$, 
\beq
&&[X\ot f, Y\ot g]=[X, Y]\ot fg + c \kappa(X, Y) Res_0\left(\frac{dg}{d t} f\right), \\
&&\text{$c$ is central,}
\eeq
where $Res_0(F)$  is the coefficient of $t^{-1}$ in $F\in\C[t, t^{-1}]$.

To verify that this indeed defines a Lie colour (super)algebra,  we drop the tensor product sign from $X\ot f$ etc..  We have 
\[
\baln
[Yg, X f] &= [Y, X] f g + c \kappa(Y, X) Res_0(g\frac{d f}{d t})\\
&= -\omega(d Y, d X) \left([X, Y] f g + c \kappa(X, Y) Res_0(f\frac{d g}{d t})\right)\\
&= -\omega(d Y, d X) [Xf, Yg], 
\ealn
\]
proving the $\omega$-skew symmetry of the generalised Lie bracket. 

The generalised Jacobian identity can be verified in the same way as for usual affine Lie algebras.
Write
\[
\baln
\SJ&=[X   f, [Y   g, Z   h]]-[[X   f, Y   g], Z   h] \\
&-\omega(d(X), d(Y)) [Y   g, [X   f,  Z   h]].
\ealn
\]
By using the generalised Jacobian identity of $\CL(\fg)$, we obtain 
\beq\label{eq:jac-aff}
\SJ&=& c \kappa(X, [Y, Z]) Res_0\left(\frac{dgh}{d t} f\right) 
- c \kappa([X, Y], Z) Res_0\left(\frac{dh}{d t} g f\right)  \\
&& - c\kappa(Y, [X,  Z]) \omega(d(X), d(Y)) Res_0\left(\frac{df h}{d t} g\right).  \nonumber
\eeq
Now $\kappa([X, Y], Z])=\kappa(X, [Y, Z])$ by ad-invariance of $\kappa$, and 
\[
\baln
\phantom{=}&\omega(d(X), d(Y)) \kappa(Y, [X,  Z])\\
&=\omega(d(Y), d(Z)) \kappa([X,  Z], Y) &\quad& \text{by $\omega$-symmetry of $\kappa$}\\
&= \omega(d(Y), d(Z))  \kappa(X, [Z, Y])  &\quad& \text{by ad-invariance of $\kappa$}\\
&=-\kappa(X, [Y, Z]),     &\quad& \text{by $\omega$-skew symmetry of $[\ , \ ]$.}
\ealn
\] 
Hence the right hand side of \eqref{eq:jac-aff} is equal to 
\[
\baln
\phantom{=}& c \kappa(X, [Y, Z]) \left( Res_0\left(\frac{dgh}{d t} f\right) - Res_0\left(\frac{dh}{d t} g f\right) +Res_0\left(\frac{df h}{d t} g\right)\right)\\
&= c \kappa(X, [Y, Z]) Res_0\left(\frac{dfgh}{d t}\right) =0, 
\ealn
\] 
since $\frac{d F}{d t}$ does not contain a $t^{-1}$ term for any $F\in\C[t, t^{-1}]$.  
Hence $\SJ=0$, proving the generalised Jacobian identity for $\wh\fg$.

Assume that the Lie colour algebra  $\fg$ fulfils the Cartan--Weyl paradigm, thus $\fg=\fg(A; \Xi)$ for some reduced Cartan matrix of finite type. 
Denote by $\Upsilon_1, \dots, \Upsilon_\ell$ the simple roots of the Lie colour algebra $\fg(A;\Xi)$. 
There exist $k_i\in\Z_+$ such that $\theta=\sum_{i\in I} k_i\Upsilon_i$ is the highest root of  $\fg(A;\Xi)$. 
Let  
\beq\label{eq:Upsil0}
\Upsilon_0=-\sum_{i\in I} k_i\Upsilon_i, \quad \xi_0=-\sum_{i=1}^\ell k_i \xi_i, 
\eeq
and set $\wh{I}=\{0, 1, 2, \dots, \ell\}$ and $\hat{\Xi}=(\xi_i)_{i\in\wh{I}}$. Denote by $\wh{A}=(\wh{A}_{i j})_{i, j\in \wh{I}}$ the generalised Cartan matrix such that  
$\wh{A}_{0 i}=\frac{2(\Upsilon_0, \Upsilon_i)}{(\Upsilon_0, \Upsilon_0)}$, 
$\wh{A}_{i 0}=\frac{2(\Upsilon_i, \Upsilon_0)}{(\Upsilon_i, \Upsilon_i)}$ for $i\in \wh{I}$, and 
$\wh{A}_{i j} = A_{i j}$ for $i, j\in I$. 
We can choose root vectors $X_\theta$ and 
$Y_\theta$ of $\pm \theta$ respectively such that the Cartan element $h_\theta:=[X_\theta, 
Y_\theta]$ satisfies 
\[
[h_\theta, X_\theta]= 2 X_\theta, \quad 
 [h_\theta, Y_\theta]=-2Y_\theta. 
\]
Then $X_0:= Y_\theta   t$, $Y_0:= X_\theta   t^{-1}$ and $Z_0:=\kappa(Y_\theta, X_\theta)  c - h_\theta$ satisfy the relations
\[
\baln
&[Z_0, X_i]= -\Upsilon_i(h_\theta) X_i, \quad [Z_0, Y_i]= \Upsilon_i(h_\theta) Y_i, \\
&[X_i, Y_0]=0, \quad [Y_i, X_0]=0, \quad\quad \forall i=1, 2, \dots, |I|, \\
&[Z_0, X_0] =2X_0, \quad [Z_0, Y_0] = -2 Y_0, \quad [X_0, Y_0] = Z_0.
\ealn
\]

\subsection{Lie  and affine colour algebras fulfilling the Cartan--Weyl paradigm}\label{sect:present}
Investigations in the last section show that 
the classic construction pioneered by Kac \cite{K} and Moody \cite{Mr} (also see \cite{Kv-1} \cite[\S1]{Kv-2})  naturally adapts  to the colour setting to construct all Lie colour (super)ralgebras which 
fulfil the Cartan--Weyl paradigm. 
The finite dimensional simple Lie colour algebras and their untwisted affine Lie colour algebras are constructed below.

We assume that $\Gamma =\Gamma^+$.

\begin{definition}\label{def:fin}
Given a reduced Cartan matrix $A=(A_{i j})_{i, j\in I}$ of finite type with $I=\{1, 2, \dots, \ell\}$, 
and a fixed sequence $\Xi=(\xi_i)_{i\in I}$ of elements in $\Gamma$, let 
$\fg(A; \Xi)$ be the Lie colour algebra generated by the homogeneous elements 
$
X_i, Y_i, Z_i,$  for $i\in I, 
$
with $\Gamma$-degrees $d(X_i)=- d(Y_i)=\xi_i$ and $d(Z_i)=0$, 
subject to the following relations. 
\beq
&&[Z_i, Z_j]=0, \\
&&[Z_i, X_j]= A_{i j} X_i, \quad [Z_i, Y_j]= -A_{i j} Y_j, \\
&&[X_i, Y_j]=\delta_{i j} Z_i, \quad \forall i, j\le \ell, \\
&& ad_{X_i}^{1-A_{i j}}(X_j)=0, \quad ad_{Y_i}^{1-A_{i j}}(Y_j)=0, \quad i\ne j.
\eeq
\end{definition}

\begin{remark}
We can similarly construct simple Lie colour superalgebras, which fulfil the Cartan--Weyl paradigm, by using the Serre type presentations for simple Lie superalgebras fully established in \cite{Z14}.  
\end{remark}

\begin{definition}\label{def:aff}
Retain notation at the end of Section \ref{sect:loop} and Definition \ref{def:fin}. 
In particular, $\wh{A}$ is the affine Cartan matrix of the reduced finite type Cartan matrix $A$, 
and $\wh\Xi= (\xi_0, \Xi)$ with $\xi_0$ given by \eqref{eq:Upsil0}. 
Let 
$\wh\fg(\wh{A}; \wh\Xi)'$ be the Lie colour algebra generated by the homogeneous elements 
$
X_i, Y_i, Z_i, $  for $i\in \wh{I}, 
$
with $\Gamma$-degrees $d(X_i)=- d(Y_i)=\xi_i$ and $d(Z_i)=0$, 
subject to the following relations. 
\beq
&&[Z_i, Z_j]=0, \\
&&[Z_i, X_j]= \wh{A}_{i j} X_i, \quad [Z_i, Y_j]= -\wh{A}_{i j} Y_j, \\
&&[X_i, Y_j]=\delta_{i j} Z_i, \quad \forall i, j\le \ell, \\
&& ad_{X_i}^{1-\wh{A}_{i j}}(X_j)=0, \quad ad_{Y_i}^{1-\wh{A}_{i j}}(Y_j)=0, \quad i\ne j.
\eeq
Denote by $\wh\fg(\wh{A}; \wh\Xi)$ the Lie colour algebra generated by the subalgebra $\wh\fg(\wh{A}; \wh\Xi)'$ and an element $\partial$, which is homogeneous of degree $0$,  such that 
\beq
{[\partial, Z_i]}=0, \quad 
[\partial, X_i]=\delta_{i 0} X_i, \quad [\partial, Y_i]=-\delta_{i 0} Y_i, \quad \forall i\in\hat{I}. 
\eeq
We call $\wh\fg(\wh{A}; \wh\Xi)$ (and often also $\wh\fg(\wh{A}; \wh\Xi)'$)  the affine Lie colour algebra of $\fg(A; \Xi)$. 
\end{definition}

These colour algebras are the ones which we will quantise.

\section{Quantised universal enveloping colour algebras}

We construct quantised universal enveloping colour algebras of the Lie 
and affine Lie colour algebras defined in Section \ref{sect:present}, 
which fulfil the Cartan--Weyl paradigm.  
The foundation of this section lies in
the theory of Hopf $(\Gamma, \omega)$-algebras,  
which are discussed in considerable detail in Appendix \ref{sect:double}. 

We will work over the field of rational functions $\C(q)$ in the indeterminate $q$. 
Let 
$[n]_q=\frac{q^n - q^{-n}}{q-q^{-1}},$ and 
$ \begin{bmatrix}m\\ r\end{bmatrix}_q= \frac{[m]_q!}{[m-r]_q![r]_q!}.
$\\

{\bf Hereafter we assume that $\Gamma_R(\fg)\subset \Gamma^+$}.

\subsection{Definition of colour quantum groups} \label{sect:def-Uq}

We  first streamline notation for clarity. 
We will use $A=(A_{i j})_{i, j \in I}$ to denote a Cartan matrix either of finite type
or affine type, 
with $I=\{1, 2, \dots, \ell\}$ in the finite case, and $I=\{0, 1, 2, \dots, \ell\}$ in the affine case.  
We also have the sequence $\Xi=(\xi_i)_{i\in I}$ of elements of $\Gamma$ in both cases, where we recall that,  in the affine case, $\xi_0$ is defined by \eqref{eq:Upsil0}. 
We emphasise that here we only consider the case with $\xi_i\in \Gamma^+$ for all $i\in I$. 
Let $d_i$ be the smallest positive even integer such that $B=diag(d_1, \dots, d_\ell) A$ is a symmetric matrix over $\Z$.  Set  $q_i =q^{d_i/2}$.  

\begin{definition}\label{def:main}
Retain notation above. 
Let $\U_{q, \Xi}(A)$ be the unital associative $(\Gamma, \omega)$-algebra 
over $\C(q)$ generated by the  elements 
$
e_i, f_i, k_i^{\pm 1}, 
$
for $i\in I$, 
which are homogeneous in the $\Gamma$-grading with degrees 
$
d(e_i)=\xi_i, d(f_i)=-\xi_i, d(k_i^{\pm 1})=0
$
respectively,
subject to the following relations
\beq
&&k_i  k_j = k_j k_i, \quad k_i k_i^{-1}  =1, \label{eq:quad-1}\\
&&k_i e_j k_i^{-1}= q_i^{A_{i j}} e_j,  \quad k_i f_j k_i^{-1}= q_i^{-A_{i j}} f_j,   \label{eq:quad-2}\\
&& e_i f_j - \omega(\xi_j, \xi_i) f_j e_i= \delta_{i j} \frac{k_i - k_i^{-1}}{q_i-q_i^{-1}}, \quad \forall i, j,  \label{eq:quad-3}\\
&&\sum_{r=0}^{1-A_{i j}} (- \omega(\xi_i, \xi_j))^r\begin{bmatrix}1-A_{i j}\\ r\end{bmatrix}_{q_i}
e_i^{1-A_{i j}-r} e_j e_i^r=0, \quad \forall  i\ne j, \label{eq:S-1} \\
&&\sum_{r=0}^{1-A_{i j}} (- \omega(\xi_i, \xi_j))^r\begin{bmatrix}1-A_{i j}\\ r\end{bmatrix}_{q_i}
f_i^{1-A_{i j}-r} f_j f_i^r=0, \quad \forall  i\ne j. \label{eq:S-2}
\eeq
Call $\U_{q, \Xi}(A)$ the quantised universal enveloping colour algebra of the 
colour Lie algebra $\fg(A; \Xi)$ of Definition \ref{def:fin}, 
or the affine Lie colour algebra $\fg(A; \Xi)$ of Definition \eqref{def:aff}.
We will loosely refer to it as a colour quantum group.
In the case where $A$ is of affine type, we also call $\U_{q, \Xi}(A)$ a quantum affine colour algebra.
\end{definition}

The commutative factor appears explicitly in the defining relations of $\U_{q, \Xi}(A)$. 
Equations \eqref{eq:S-1} and \eqref{eq:S-2} will be called colour quantum Serre relations.  

\begin{remark}
For each fixed $i$, the elements $e_i, f_i, k_i^{\pm 1}$ generate a $\U_q(\fsl_2)$ subalgebra, since for $j=i$,  equation \eqref{eq:quad-3} reduces to the standard $\U_q(\fsl_2)$ relation $e_i f_i -  f_i e_i= \frac{k_i - k_i^{-1}}{q_i-q_i^{-1}}$ if $\omega(\xi_i, \xi_i)=1$.  The case \(\omega(\xi_i,\xi_i)=-1\) belongs to a broader colour superalgebra setting and is not considered here.
\end{remark}

For any homogeneous $X, Y\in \U_{q, \Xi}(A)$, we define 
\beq
& Ad_{e_i}(X)= e_i X -\omega(\xi_i, d(X)) k_i X k_i^{-1} e_i, \label{eq:Ade}\\
& Ad_{f_i}(Y)= f_i Y -\omega(d(Y), \xi_i)  k_i^{-1} Y k_i f_i. \label{eq:Adf}
\eeq

\begin{lemma}\label{lem:S-pm}
Retain notation above. The colour quantum Serre relations \eqref{eq:S-1} and \eqref{eq:S-2} can respectively be expressed as 
\beq
Ad_{e_i}^{1-A_{i j}} (e_j)= 0,  \quad 
Ad_{f_i}^{1-A_{i j}} (f_j)= 0, \quad \forall i\ne j.  \label{eq:S+} \label{eq:S-}
\eeq
\end{lemma}

\begin{proof}
By an easy induction on $N$, one can  prove the following formulae, 
\beq
Ad_{e_i}^N(e_j)= \sum_{r=0}^N \left(- \omega(\xi_i, \xi_j) q_i^{N-1+A_{i j}}\right)^r
\begin{bmatrix}N\\ r\end{bmatrix}_{q_i} e_i^{N-r} e_j e_i^r, \label{eq:Ad-e}\\
Ad_{f_i}^N(f_j)= \sum_{r=0}^N \left(- \omega(\xi_i, \xi_j) q_i^{N-1+A_{i j}}\right)^r
\begin{bmatrix}N\\ r\end{bmatrix}_{q_i} f_i^{N-r} f_j f_i^r,  \label{eq:Ad-f}
\eeq
with the help of the following relations among  $q$-binomial coefficients, 
\beq
\begin{bmatrix}m\\ r\end{bmatrix}_q &=& q^{m-r}\begin{bmatrix}m-1\\ r-1\end{bmatrix}_q + 
q^{-r}\begin{bmatrix}m-1\\ r\end{bmatrix}_q  \label{eq:q-sum}\\
& =& q^{r-m}\begin{bmatrix}m-1\\ r-1\end{bmatrix}_q + 
q^{r}\begin{bmatrix}m-1\\ r\end{bmatrix}_q. \nonumber 
\eeq
The lemma immediately follows by setting $N= 1- A_{i j}$. 
\end{proof}

\begin{example}[Colour quantum Serre relations in type $A$ case]
For type $A$, we have 
$A_{i j}=0$ if $|i-j|>1$, and $A_{i, i\pm 1}=-1$ for all valid $i, i\pm 1$. Thus 
\[
\baln
&e_i e_j -\omega(\xi_i, \xi_j)  e_j e_i =0, \quad 
f_i f_j -\omega(\xi_i, \xi_j)  f_j f_i=0,  \quad \text{if $|i-j|>1$}, \\
&e_i ^2 e_{i\pm 1} - \omega(\xi_i, \xi_{i\pm 1}) (q+q^{-1}) e_i e_{i\pm 1}  e_i 
+ \omega(2\xi_i, \xi_{i\pm 1})  e_{i\pm 1}  e_i ^2=0, \\
&f_i ^2 f_{i\pm 1} - \omega(\xi_i, \xi_{i\pm 1}) (q+q^{-1}) f_i f_{i\pm 1}  f_i 
+ \omega(2\xi_i, \xi_{i\pm 1})  f_{i\pm 1}  f_i ^2=0.
\ealn
\]
\end{example}

Note the following $(\Gamma, \omega)$-subalgebras of
$\U_{q, \Xi}(A)$: 
\beq\label{eq:subgroups}
\baln
\U^0, &\quad \text{generated by $\{ k_i^{\pm 1}\mid i\in I\}$}, \\
\U^+_{q, \Xi}(A), &\quad \text{generated by $\{e_i, k_i^{\pm 1}\mid i\in I \}$},  \\
\U^-_{q, \Xi}(A), &\quad \text{generated by $\{ f_i, k_i^{\pm 1}\mid i\in I\}$}, \\
\U^+_+, &\quad \text{generated by $\{e_i\mid i\in I \}$},  \\
\U^-_-, &\quad \text{generated by $\{ f_i\mid i\in I\}$}. 
\ealn
\eeq
It is a direct consequence of the defining relations \eqref{eq:quad-1}--\eqref{eq:quad-3} that $\U_{q, \Xi}(A)$ has the triangular decomposition 
\[
\U_{q, \Xi}(A)=\U^-_- \U^0 \U^+_+. 
\]

The following important fact will be proven in the next section.
\begin{theorem} \label{thm:Hopf}
The algebra $\U_{q, \Xi}(A)$ is
a Hopf $(\Gamma, \omega)$-algebra, with co-multiplication $\Delta: \U_{q, \Xi}(A)\lra \U_{q, \Xi}(A)\ot \U_{q, \Xi}(A)$, 
co-unit $\varepsilon: \U_{q, \Xi}(A)\lra \C(q)$, and antipode $S: \U_{q, \Xi}(A) \lra \U_{q, \Xi}(A)$, respectively defined by 
\beq
& \Delta(e_i)= e_i\ot k_i + 1\ot e_i,  \ \Delta(f_i)= f_i\ot 1 + k_i^{- 1}\ot f_i,  \label{eq:co-mult}\\
&  \Delta(k_i^{\pm 1})=k_i^{\pm 1}\ot k_i^{\pm 1},  \nonumber\\
& \varepsilon(e_i)=0, \  \varepsilon(f_i)=0, \ \varepsilon(k_i^{\pm 1})=1,  \label{eq:co-unit} \\
& S(e_i)=  - e_i k_i^{-1}, \  S(f_i)=  -  k_i f_i, \ S(k_i^{\pm 1})=k_i^{\mp 1}, \quad \forall i\in I.
\label{eq:antipode}
\eeq
\end{theorem}

It follows from Theorem \ref{thm:Hopf} that $\U^\pm_{q, \Xi}(\fg)$ are 
Hopf $(\Gamma, \omega)$-subalgebras.

\begin{remark}\label{rmk:AAB-1}
A class of Hopf algebras $\U(\mathscr{E})$ was constructed from Nichols algebras in \cite[\S3.2]{AAB}. These algebras are expected to be closely related to the bosonisation of our Hopf colour algebras $\U_{q, \Xi}(A)$ with respect to the subgroups of $\Gamma$ generated by $\Xi$. 
However, to compare $\U(\mathscr{E})$ and the bosonised $\U_{q, \Xi}(A)$ in precise terms, 
one must first determine the explicit quantum Serre relations for $\U(\mathscr{E})$. 
We also note a structural variation between our relation \eqref{eq:quad-3} and \cite[equation (3.10)]{AAB}.
\end{remark}

\subsection{The auxiliary colour quantum  groups and Hopf structure}\label{sect:ACQG}
In this section, we develop a construction for the colour quantum groups as Hopf colour algebras. 
The Hopf structure in Theorem \ref{thm:Hopf} will then follow naturally from this construction.

While our formulation treats the colour quantum group as a unified whole,  
it implicitly encapsulates elements of the Nichols algebra approach \cite{HS20}, 
which will manifest explicitly in Section \ref{sect:QT} when establishing 
the quasi-triangular structure via the quantum double construction.

Let us introduce the following auxiliary Hopf colour algebra associated with $\U_{q, \Xi}(A)$.
\begin{definition}\label{def:U-aux}
Let $\wt\U_{q, \Xi}(A)$ be the unital associative $(\Gamma, \omega)$-algebra
generated by the homogeneous generators 
$
e_i, f_i, k_i^{\pm 1} 
$
for $i\in I$, of degrees $d(e_i)=\xi_i$, $d(f_i)=-\xi_i$ and $d(k_i^{\pm 1})=0$, 
with defining relations \eqref{eq:quad-1}, \eqref{eq:quad-2}  
and \eqref{eq:quad-3} only, i.e., without the colour quantum Serre relations.
\end{definition}
We call $\wt\U_{q, \Xi}(A)$ an auxiliary colour quantum  group.

\begin{remark}
We abused notation by denoting the generators of $\wt\U_{q, \Xi}(A)$ with the same symbols as those for $\U_{q, \Xi}(A)$, but we do not expect this to cause confusion.
\end{remark}

There are the following $(\Gamma, \omega)$-subalgebras of $\wt\U_{q, \Xi}(A)$.
\beq\label{eq:aux-subgp}
\baln
\wt\U^0, &\quad \text{generated by $\{ k_i^{\pm 1}\mid i\in I\}$}, \\
\wt\U^+_{q, \Xi}(A), &\quad \text{generated by $\{e_i, k_i^{\pm 1}\mid i\in I \}$},  \\
\wt\U^-_{q, \Xi}(A), &\quad \text{generated by $\{ f_i, k_i^{\pm 1}\mid i\in I\}$}, \\
\wt\U^+_+, &\quad \text{generated by $\{e_i\mid i\in I \}$},  \\
\wt\U^-_-, &\quad \text{generated by $\{ f_i\mid i\in I\}$}. 
\ealn
\eeq
Clearly $\wt\U^+_{q, \Xi}(A)\supset \wt\U^+_+$ and $\wt\U^-_{q, \Xi}(A)\supset \wt\U^-_-$. 

The following fact is easy to see.
\begin{lemma} 
The algebra $\wt\U_{q, \Xi}(A)$ has the structure of 
a Hopf $(\Gamma, \omega)$-algebra, with co-multiplication 
$\Delta: \wt\U_{q, \Xi}(A)\lra \wt\U_{q, \Xi}(A)\ot \wt\U_{q, \Xi}(A)$, 
co-unit $\varepsilon: \wt\U_{q, \Xi}(A)\lra \C(q)$, and 
antipode $S: \wt\U_{q, \Xi}(A) \lra \wt\U_{q, \Xi}(A)$, 
which are defined by the same formulae \eqref{eq:co-mult},  \eqref{eq:co-unit} and \eqref{eq:antipode} respectively, but interpreted as maps for $\wt\U_{q, \Xi}(A)$. 
\end{lemma}
\begin{proof}
It is evident that $\varepsilon$ is a $(\Gamma, \omega)$-algebra homomorphism. 
To show that $S$ is a $(\Gamma, \omega)$-algebra anti-homomorphism, 
one verifies directly that it preserves the relations \eqref{eq:quad-1} and \eqref{eq:quad-2}. 
Now we  prove that $S$ also preserves \eqref{eq:quad-3}. We have 
\[
\baln
S(e_i f_j - \omega(\xi_j, \xi_i) f_j e_i) 
&= \omega(\xi_j, \xi_i) S(f_j) S(e_i) 
- \omega(\xi_j, \xi_i)  \omega(\xi_i, \xi_j)S(e_i)S(f_j)\\
&= \omega(\xi_j, \xi_i) k_j f_j  e_i k_i^{-1} - e_i k_i^{-1} k_j f_j. 
\ealn
\]
Let us denote the right hand side by $RHS$ and manipulate  it as follows. 
\[
\baln
RHS&=\omega(\xi_j, \xi_i) k_j f_j  e_i k_i^{-1} - q_j^{-A_{j i}} q_i^{A_{i j}}k_j e_i   f_j k_i^{-1}\\
&= - k_j (e_i   f_j  - \omega(\xi_j, \xi_i)  f_j  e_i ) k_i^{-1} \quad \text{(using $q_j^{A_{j i}}=q_i^{A_{i j}}$)}\\
&= -  \delta_{i j} \frac{k_i - k_i^{-1}}{q_i-q_i^{-1}} = \delta_{i j} \frac{S(k_i) - S(k_i^{-1})}{q_i-q_i^{-1}}.  
\ealn
\]
This proves that $S$ preserves \eqref{eq:quad-3}. 

Let us now show that the map $\Delta$ is a $(\Gamma, \omega)$-algebra homomorphism. 
Again it is easy to see that $\Delta$ respects \eqref{eq:quad-1} and \eqref{eq:quad-2}. To show that it also preserves \eqref{eq:quad-3}, we note that 
\[
\baln
\Delta(e_i f_j) 
&= e_i f_j \ot k_i  + k_j^{-1}\ot e_i f_j + \omega(\xi_j, \xi_i) f_j\ot e_i + e_i k_j^{-1} \ot k_i f_j, \\
\Delta( f_j e_i) 
&=f_j e_i\ot k_i + k_j^{- 1}\ot f_j e_i + \omega(\xi_i, \xi_j) e_i k_j^{- 1}\ot  k_i f_j + f_j\ot e_i, 
\ealn
\]
where we have used the relation $\omega(-\alpha, \beta)=\omega(\beta, \alpha)$ for all $\alpha, \beta$, 
and  also the fact that $k_j^{- 1}e_i\ot f_j k_i = e_i k_j^{- 1}\ot  k_i f_j$, which is a consequence of the 
relations in \eqref{eq:quad-2} and the fact that $q_j^{A_{j i}}=q_i^{A_{i j}}$. 
Since $\omega(\xi_i, \xi_j) \omega(\xi_j, \xi_i)=1$, we obtain 
\[
\baln
\Delta(e_i f_j- \omega(\xi_j, \xi_i)  f_j e_i)
&= (e_i f_j- \omega(\xi_j, \xi_i)  f_j e_i)\ot k_i\\
&+ k_j^{-1}\ot (e_i f_j- \omega(\xi_j, \xi_i)  f_j e_i).
\ealn
\]
We can simplify the right hand side by using \eqref{eq:quad-3}, to obtain
\[
\baln
&\delta_{i j} \frac{k_i - k_i^{-1}}{q_i-q_i^{-1}}\ot k_i 
+ \delta_{i j} k_j^{-1}\ot  \frac{k_i - k_i^{-1}}{q_i-q_i^{-1}}
=\delta_{i j}\Delta\left(\frac{k_i - k_i^{-1}}{q_i-q_i^{-1}}\right), 
\ealn
\]
showing  that $\Delta$ preserves \eqref{eq:quad-3}.

Finally, a direct verification on the generators shows that  the maps $\Delta, \varepsilon$ and $S$ of the algebra $\wt\U_{q, \Xi}(A)$ satisfy the defining relations of co-multiplication, co-unit and  antipode, namely, 
$(\varepsilon\ot \id)\Delta=\id$, $(\id\ot \varepsilon)\Delta=\id$, and $\mu(S\ot \id)\Delta= \mu(\id \ot S)\Delta=\varepsilon$, where $\mu$ denotes the multiplication of $\wt\U_{q, \Xi}(A)$.
\end{proof}

Define the following elements of $\wt\U_{q, \Xi}(A)$:
\beq
S^+_{i j} = Ad_{e_i}^{1-A_{i j}} (e_j), \quad S^-_{i j}=Ad_{f_i}^{1-A_{i j}} (f_j), \quad i\ne j,
\eeq 
where $ Ad_{e_i}$ and $ Ad_{f_i}$ are defined similarly as in \eqref{eq:Ade}  and \eqref{eq:Adf}.
We call $S^\pm_{i j}$ the colour quantum Serre elements of $\wt\U_{q, \Xi}(A)$. 
Adapting the proof of Lemma \ref{lem:S-pm}, we obtain
\beq
S_{i j}^+=\sum_{r=0}^{1-A_{i j}} (- \omega(\xi_i, \xi_j))^r\begin{bmatrix}1-A_{i j}\\ r\end{bmatrix}_{q_i}
e_i^{1-A_{i j}-r} e_j e_i^r,  \label{eq:def-S+}\\
 S_{i j}^-=
\sum_{r=0}^{1-A_{i j}} (- \omega(\xi_i, \xi_j))^r\begin{bmatrix}1-A_{i j}\\ r\end{bmatrix}_{q_i}
f_i^{1-A_{i j}-r} f_j f_i^r.  \label{eq:def-S-}
\eeq

The colour quantum Serre elements have the following key properties.
\begin{lemma} \label{lem:ideal}
The elements $S^\pm_{i j}$, for $i\ne j$, satisfy the relations
\beq
&\quad& {[f_p, S^+_{i j}]_\omega}=0, \quad {[e_p, S^-_{i j}]_\omega}=0,\quad \forall p, \label{eq:ideal}
\eeq
where $[\ , \ ]_\omega$ is the $\omega$-commutator defined by  
$[X, Y]_\omega= X Y - \omega(d(X), d(Y)) Y X$. 
\end{lemma}

\begin{lemma}\label{lem:co-ideal}
The elements $S^\pm_{i j}$, for $i\ne j$, are skew primitive in the sense that 
\beq
&&\Delta(S^+_{i j}) = S^+_{i j}\ot k_i^{1-A_{i j}} k_j + 1 \ot S^+_{i j}, \label{eq:co-ideal}\\
&&
 \Delta(S^-_{i j}) = S^-_{i j}\ot 1 + k_i^{-1+A_{i j}} k_j^{-1}  \ot S^-_{i j}, \label{eq:co-ideal-1}
\eeq
and hence 
\beq
S(S^+_{i j}) = - S^+_{i j}\ot k_i^{-1+A_{i j}} k_j^{-1}, \quad 
 S(S^-_{i j}) = -  k_i^{1-A_{i j}} k_j  S^-_{i j}.
\eeq
\end{lemma}

The two lemmas above play key roles in our construction of colour quantum groups. Their proofs will be given in Section \ref{sect:tech-lems}. 

\begin{definition}\label{def:Serre-ideal}
Let $\CJ_+^+$ be the two-sided ideal of $\wt{\U}_+^+$ generated by $S^+_{i j}$ for all $i\ne j$,
and let $\CJ_-^-$ be the two-sided ideal of $\wt{\U}_-^-$ generated by $S^-_{i j}$ for all $i\ne j$. Denote by $\CJ$ the two-sided ideal of $\wt\U_{q, \Xi}(A)$ generated $\CJ_+^+$ and $\CJ_-^-$. 
\end{definition}

We have the following result. 
\begin{corollary}\label{cor:intersect}
The two-sided ideal $\CJ$ has the triangular decomposition 
\[\CJ= \CJ_-^- \wt{\U}^0 \CJ_+^+.\]
Moreover, it intersects the subspace $\wt{Gen}:=\sum_i \wt\U^0 e_i + \wt\U^0 + \sum_i \wt\U^0f_i$ trivially.  
\end{corollary}
\begin{proof}
By Lemma \ref{lem:ideal}, $S^+_{i j}$ (resp. $S^-_{i j}$) $\omega$-commute with $e_k$ (resp. $f_k$) for all $k$. Also all colour quantum Serre elements $q$-commute with $k_k^{\pm 1}$, since they have definite weights. These observations enable us to give $\CJ$ a triangular decomposition. 
As all $S^+_{i j}$ (resp. $S^-_{i j}$) are of degrees greater than $1$ in $e_k$'s (resp. $f_k$'s),  
we have $\CJ\cap \wt{Gen}=0$. 
\end{proof}

We further note that Lemma \ref{lem:co-ideal} and Corollary \ref{cor:intersect} imply the following fact.
\begin{corollary}\label{cor:J}
The two-sided ideal $\CJ$ is a homogeneous Hopf ideal of $\wt\U_{q, \Xi}(A)$.
\end{corollary}
\begin{proof}
Indeed, Corollary \ref{cor:intersect} gives the triangular decomposition
$
\CJ=\CJ_-^-\wt\U^0\CJ_+^+.
$
On the other hand, Lemma \ref{lem:co-ideal} shows that the colour quantum Serre elements are skew primitive and that their antipodes again lie in $\CJ$. It follows that
\[
\Delta(\CJ)\subseteq \CJ\otimes\widetilde U+\widetilde U\otimes\CJ,
\qquad
S(\CJ)\subseteq\CJ,
\]
while $\varepsilon(\CJ)=0$ since the co-unit maps all Serre elements to $0$. Hence $\CJ$ is a homogeneous Hopf ideal.
\end{proof}


We can now prove Theorem \ref{thm:Hopf}. 

\begin{proof}[Proof of Theorem \ref{thm:Hopf}]
Consider the quotient algebra
\[
{\mathbb U}=\wt\U_{q, \Xi}(A)/\CJ.
\]
Since the quantum Serre elements $S^\pm_{i j}$ are homogeneous, $\CJ$ is a homogeneous
 two-sided ideal. Hence  ${\mathbb U}$ is $\Gamma$-graded. 

Since the two-sided ideal $\CJ$ is generated by the colour quantum Serre elements $S^\pm_{i j}$ for all $i\ne j$, 
we have the associative $(\Gamma, \omega)$-algebra isomorphism $\U_{q, \Xi}(A)\simeq {\mathbb U}$ by construction. 
However, it is important to observe that Corollary \ref{cor:intersect} is crucial in guaranteeing that the resulting algebra ${\mathbb U}$ does not collapse to something  nonsensical. In particular, 
the fact $\CJ\cap \wt{Gen}=0$ ensures that taking the quotient does not annihilate 
any of the generators $e_i, f_i$, or impose relations among the generators $k_i^{\pm}$. 

Corollary \ref{cor:J} shows that $\CJ$ is a homogeneous Hopf ideal of $\wt\U_{q, \Xi}(A)$, hence 
${\mathbb U}$ has the structure of a Hopf $(\Gamma, \omega)$-algebra. 
%
%
By inspecting the definition of $\wt\U_{q, \Xi}(A)$, we see that the structure maps of ${\mathbb U}$ agree with those described by Theorem \ref{thm:Hopf}. Hence $\U_{q, \Xi}(A) \simeq {\mathbb U}$ as Hopf $(\Gamma, \omega)$-algebra. 
\end{proof}

\subsection{Key lemmas - their proofs and structural connections}\label{sect:tech-lems}
We prove Lemmas \ref{lem:ideal} and \ref{lem:co-ideal}
and investigate their connections  in this section.  

\subsubsection{Some useful formulae} 
We collect here some formulae to be used later. 
The following relations hold, which are familiar from ordinary quantum groups. 
\beq
[e_i, f_i^r] &=&[r]_{q_i} f_i^{r-1}  \frac{k_i q_i^{-r +1}  -   k_i^{-1}q_i^{ r-1}}{q_i-q_i^{-1}}, \label{eq:e-fn} \\ 
{}[f_i, e_i^r] &=&-[r]_{q_i} e_i^{r-1}  \frac{k_i q_i^{ r-1}- k_i^{-1} q_i^{-r +1} }{q_i-q_i^{-1}};
\label{eq:f-en}\\
\Delta(e_i^n) &=& \sum_{s=0}^n q_i^{s(n-s)}\begin{bmatrix}n \\ s\end{bmatrix}_{q_i} e_i^{n-s} \ot e_i^s k_i^{n-s},\label{eq:Den}\\
\Delta(f_i^n) &=& \sum_{s=0}^n q_i^{s(n-s)}\begin{bmatrix}n \\ s\end{bmatrix}_{q_i}  f_i^s k_i^{-n+s}\ot f_i^{n-s}. \label{eq:Dfn}
\eeq

The first and third formulae can be proven by inductions. 
One can deduce 
the second relation from the first, and the fourth relation from the third, by using 
the following algebra automorphism.  
  
\begin{lemma}\label{lem:auto}
There is a $(\Gamma, \omega)$-algebra automorphism $\theta: \wt\U_{q, \Xi}(A)\lra \wt\U_{q, \Xi}(A)$ 
defined by extending the map    
\[
e_i\mapsto f_i, \quad f_i\mapsto e_i, \quad k_i^{\pm 1} \mapsto k_i^{\mp 1}, \quad \forall i,
\]
which satisfies 
$
 (\theta\ot\theta)\Delta =\Delta'\theta. 
$
Here \(\Delta'\) denotes the braided opposite co-product defined in Appendix \(\ref{sect:Hopf-dual}\).
\end{lemma}

Finally, we have the following formulae. 
\begin{lemma}\label{lem:co-prods}
The co-products of the elements $Ad_{e_i}^n(e_j)$ and $Ad_{f_i}^n(f_j)$ of $\wt\U_{q, \Xi}(A)$, with $i\ne j$, are given by  
\[
\baln
\Delta\left(Ad_{e_i}^n(e_j)\right) &=1\ot Ad_{e_i}^n(e_j)  \\
&+ \sum_{s=0}^{n} \omega(\xi_i, \xi_j)^{n-s} q_i^{s(n-s)}\begin{bmatrix}n \\ s\end{bmatrix}_{q_i}  \left(\prod_{r=s}^{n-1}(1- q_i^{2r} q_i^{2A_{i j}})\right)\\
&\times  Ad_{e_i}^{s}(e_j) \ot e_i^{n-s} k_i^{s}k_j,  \\
\Delta\left(Ad_{f_i}^n(f_j)\right) &=Ad_{f_i}^n(f_j)\ot 1  \\
&+ \sum_{s=0}^{n} q_i^{s(n-s)}\begin{bmatrix}n \\ s\end{bmatrix}_{q_i}  \left(\prod_{r=s}^{n-1}(1- q_i^{2r} q_i^{2A_{i j}})\right) \\
&\times  f_i^{n-s} k_i^{-s}k_j^{-1} \ot  Ad_{f_i}^{s}(f_j),  
\ealn
\]
where for $s=n$, the empty product $\prod_{r=s}^{n-1}(1- q_i^{2r} q_i^{2A_{i j}})$ is taken to be $1$. 
\end{lemma}
\begin{proof}
We first prove the formula for the co-product of $Ad_{e_i}^n(e_j)$ by induction on $n$. This is in principle straightforward, but requires a lot of careful book keeping,  thus we spell out some of the details. 

Let us write $\omega_{i j}=\omega(\xi_i,  \xi_j)$. 
By direct calculations, we can verify that 
\beq
\Delta(Ad_{e_i}(e_j)) &=& 1\ot Ad_{e_i}(e_j) + Ad_{e_i}(e_j)\ot k_i k_j\\
&&+ \omega_{i j}(1- q_i^{2 A_{i j}} ) e_j\ot    e_i k_j,  \nonumber\\
\Delta(Ad_{e_i}^2(e_j))&=& 1\ot Ad_{e_i}^2(e_j)+Ad_{e_i}^2(e_j)\ot k_i^2 k_j \\
&&+ \omega_{i j}q_i (q_i+q_i^{-1})(1- q_i^{2+ 2 A_{i j}} )  Ad_{e_i}(e_j)\ot    e_i k_i k_j \nonumber\\
&&+ \omega_{i j}^2 (1- q_i^{2+ 2 A_{i j}} ) (1- q_i^{2 A_{i j}} )e_j\ot    e_i^2 k_j. \nonumber
\eeq

Now we turn to the proof of the general case.  Note that 
\[
\baln
\Delta(Ad_{e_i}^{n+1}(e_j))&= \Delta(e_i) \Delta(Ad_{e_i}^n(e_j))  
- \omega_{i j} q_i^{2n + A_{i j}}\Delta(Ad_{e_i}^n(e_j)) \Delta(e_i), 
\ealn
\]
where we have used  $\omega(\xi_i,  \xi_i)=1$, which follows from the standing assumption $\Gamma_R(\fg)\subset\Gamma^+$.
We have 
\[
\baln
&\Delta(e_i) (1\ot Ad_{e_i}^n(e_j))   
- \omega_{i j} q_i^{2n + A_{i j}}(1\ot Ad_{e_i}^n(e_j))  \Delta(e_i)
= 1\ot Ad_{e_i}^{n+1}(e_j).
\ealn
\]
By induction hypothesis, 
\[
\baln
&\Delta(Ad_{e_i}^{n+1}(e_j))- 1\ot Ad_{e_i}^{n+1}(e_j)\\
&= \sum_{s=0}^{n} \omega_{i j}^{n-s} q_i^{s(n-s)}\begin{bmatrix}n \\ s\end{bmatrix}_{q_i} \left(\prod_{r=s}^{n-1}(1- q_i^{2r} q_i^{2A_{i j}})\right) \\
&\times \left(\Delta(e_i)  Ad_{e_i}^{s}(e_j) \ot e_i^{n-s} k_i^{s}k_j 
-  \omega_{i j} q_i^{2n + A_{i j}} Ad_{e_i}^{s}(e_j) \ot e_i^{n-s} k_i^{s}k_j  \Delta(e_i)\right).
\ealn
\]
The right hand side can be re-written as 
\[
\baln
&\sum_{s=0}^{n} \omega_{i j}^{n+1-s} q_i^{s(n-s)}(1- q_i^{2n +2s + 2A_{i j}}) \begin{bmatrix}n \\ s\end{bmatrix}_{q_i} \\
&\times \left(\prod_{r=s}^{n-1}(1- q_i^{2r} q_i^{2A_{i j}})\right) Ad_{e_i}^{s}(e_j) \ot e_i^{n+1-s} k_i^{s}k_j \\
&+ \sum_{s=0}^{n} \omega_{i j}^{n-s} q_i^{s(n-s)}\begin{bmatrix}n \\ s\end{bmatrix}_{q_i} \left(\prod_{r=s}^{n-1}(1- q_i^{2r} q_i^{2A_{i j}})\right) \\
&\times \left((e_i\ot k_i)  Ad_{e_i}^{s}(e_j) \ot e_i^{n-s} k_i^{s}k_j 
- \omega_{i j} q_i^{2n + A_{i j}} Ad_{e_i}^{s}(e_j) \ot e_i^{n-s} k_i^{s}k_j  (e_i\ot k_i)\right), 
\ealn
\]
where the second sum can be simplified to 
\[
\baln
\sum_{s=0}^{n} \omega_{i j}^{n-s} q_i^{s(n-s)} q_i^{2n - 2s}
\begin{bmatrix}n \\ s\end{bmatrix}_{q_i} \left(\prod_{r=s}^{n-1}(1- q_i^{2r} q_i^{2A_{i j}})\right) Ad_{e_i}^{s+1}(e_j)  \ot e_i^{n-s} k_i^{s+1}k_j. 
\ealn
\]
Hence 
\[
\baln
&\Delta(Ad_{e_i}^{n+1}(e_j))- 1\ot Ad_{e_i}^{n+1}(e_j)\\
&=  \sum_{s=0}^{n} \omega_{i j}^{n+1-s} q_i^{s(n+1-s)} q_i^{-s}(1- q_i^{2n +2s + 2A_{i j}}) \begin{bmatrix}n \\ s\end{bmatrix}_{q_i} \\
&\times \left(\prod_{r=s}^{n-1}(1- q_i^{2r} q_i^{2A_{i j}})\right) Ad_{e_i}^{s}(e_j) \ot e_i^{n+1-s} k_i^{s}k_j \\
&+  \sum_{s=1}^{n+1} \omega_{i j}^{n+1-s} q_i^{s(n+1-s)} q_i^{n +1 - s}
\begin{bmatrix}n \\ s-1\end{bmatrix}_{q_i} \\
&\times \left(\prod_{r=s-1}^{n-1}(1- q_i^{2r} q_i^{2A_{i j}})\right) Ad_{e_i}^{s}(e_j)  \ot e_i^{n+1-s} k_i^{s}k_j.
\ealn
\]

Let us regroup the terms on the right hand side. 
Denote by  $Q_n$ the  sum of 
the $s=0$ term in the first sum and the $s=n+1$ term in the second sum, and 
the rest of the right hand side by $R_n$. 
Then  
\beq
\Delta(Ad_{e_i}^{n+1}(e_j))- 1\ot Ad_{e_i}^{n+1}(e_j) =Q_n+R_n,  \label{eq:QR}
\eeq
with
\[
Q_n=\omega_{i j}^{n+1} \prod_{r=0}^{n}(1- q_i^{2r} q_i^{2A_{i j}}) e_j \ot e_i^{n+1} k_j
+ Ad_{e_i}^{n+1}(e_j)\ot k_i^{n+1}k_j,  
\] 
\[
\baln
R_n=&\sum_{s=1}^{n} \omega_{i j}^{n+1-s} q_i^{s(n+1-s)} q_i^{-s}(1- q_i^{2n +2s + 2A_{i j}}) \begin{bmatrix}n \\ s\end{bmatrix}_{q_i} \\
&\times \left(\prod_{r=s}^{n-1}(1- q_i^{2r} q_i^{2A_{i j}})\right) Ad_{e_i}^{s}(e_j) \ot e_i^{n+1-s} k_i^{s}k_j \\
&+  \sum_{s=1}^{n} \omega_{i j}^{n+1-s} q_i^{s(n+1-s)} q_i^{n +1 - s}
\begin{bmatrix}n \\ s-1\end{bmatrix}_{q_i} \\
&\times \left(\prod_{r=s-1}^{n-1}(1- q_i^{2r} q_i^{2A_{i j}})\right) Ad_{e_i}^{s}(e_j)  \ot e_i^{n+1-s} k_i^{s}k_j \\
&=
\sum_{s=1}^{n} \omega_{i j}^{n+1-s} C_s q_i^{s(n+1-s)}  \begin{bmatrix}n+1 \\ s\end{bmatrix}_{q_i} \\
&\times \left(\prod_{r=s}^{n}(1- q_i^{2r} q_i^{2A_{i j}})\right) Ad_{e_i}^{s}(e_j) \ot e_i^{n+1-s} k_i^{s}k_j, 
\ealn
\]
where $C_s=\frac{q_i^{-s}}{1- q_i^{2n+2A_{i j}}}  B_s$, and
\[
\baln
B_s:=& (1- q_i^{2n +2s + 2A_{i j}})\frac{[n+1-s]_{q_i}}{[n+1]_{q_i}} +
q_i^{n +1} (1- q_i^{2s -2+2A_{i j}}) \frac{[s]_{q_i}}{[n+1]_{q_i}}. 
\ealn
\]
It is easy to show that 
\[
\baln
(q_i^{n+1} - q_i^{-n-1}) B_s
&=   q_i^{s} (q_i^{n+1} - q_i^{-n-1})(1-q_i^{2n + 2A_{i j}}). 
\ealn
\]
Hence 
\[
\baln
Q_n+R_n
&=\omega_{i j}^{n+1}\prod_{r=0}^{n}(1- q_i^{2r} q_i^{2A_{i j}}) e_j \ot e_i^{n+1} k_j
+ Ad_{e_i}^{n+1}\ot k_i^{n+1}k_j\\
&+\sum_{s=1}^{n} \omega_{i j}^{n+1-s} q_i^{s(n+1-s)}  \begin{bmatrix}n+1 \\ s\end{bmatrix}_{q_i} \left(\prod_{r=s}^{n}(1- q_i^{2r} q_i^{2A_{i j}})\right) Ad_{e_i}^{s}(e_j) \ot e_i^{n+1-s} k_i^{s}k_j \\
&=\sum_{s=0}^{n+1} \omega_{i j}^{n+1-s} q_i^{s(n+1-s)}  \begin{bmatrix}n+1 \\ s\end{bmatrix}_{q_i} \left(\prod_{r=s}^{n}(1- q_i^{2r} q_i^{2A_{i j}})\right) Ad_{e_i}^{s}(e_j) \ot e_i^{n+1-s} k_i^{s}k_j \\
\ealn
\]
Using this in \eqref{eq:QR}, we obtain the claimed formula for $\Delta(Ad_{e_i}^{n+1}(e_j))$. 

The co-product $\Delta(Ad_{f_i}^n(f_j))$ can be obtained from $\Delta(Ad_{e_i}^n(e_j))$ by noting that the automorphism $\theta$ maps  $Ad_{e_i}^n(e_j)$ to $Ad_{f_i}^n(f_j)$.
This completes the proof. 
\end{proof}

\subsubsection{Proof of Lemmas \ref{lem:ideal} and \ref{lem:co-ideal}}

\begin{proof}[Proof of Lemma \ref{lem:ideal}]
Recall from \cite{Z25} that the $\omega$-commutator 
has the property of a $(\Gamma, \omega)$-derivation, i.e., 
$
[X, Y Z]_\omega=[X, Y]_\omega Z + \omega(d(X), d(Y)) Y [X, Z]_\omega.
$
Thus the only cases $p=i, j$ of \eqref{eq:ideal} require proof. 

Consider the second relation. For $p=j$, 
we have 
\[
\baln
[e_j, S_{i j}^-]_\omega &= \sum_{r=0}^{1-A_{i j}} (- \omega(\xi_i, \xi_j))^r  \omega(\xi_j, -\xi_i)^{1-A_{i j} -r}
\begin{bmatrix}1-A_{i j}\\ r\end{bmatrix}_{q_i}
f_i^{1-A_{i j}-r}\frac{ k_j - k_j^{-1}}{q_j-q_j^{-1}} f_i^r\\
&=\omega(\xi_i, \xi_j)^{1-A_{i j}} f_i^{1-A_{i j}}
\sum_{r=0}^{1-A_{i j}} (- 1)^r
\begin{bmatrix}1-A_{i j}\\ r\end{bmatrix}_{q_i}
\frac{ k_j q_i^{-r A_{ij}} - k_j^{-1} q_i^{r A_{ij}} }{q_j-q_j^{-1}}.
\ealn
\]
It follows the following relations among  $q$-binomial coefficients  
\beq
\sum_{i=0}^m (-1)^i q^{\pm i(m-1)} \begin{bmatrix}m\\ i\end{bmatrix}_q =0, 
 \label{eq:0-q-sum}
\eeq
that the right hand side vanishes identically. 
Hence  $[e_j, S_{i j}^-]=0$.

Now consider $[e_i, S_{i j}^-]$. By  using \eqref{eq:e-fn}, we obtain
\[
\baln
[e_i, S_{i j}^-]_\omega &= \sum_{r=0}^{1-A_{i j}} (- \omega(\xi_i, \xi_j))^r\begin{bmatrix}1-A_{i j}\\ r\end{bmatrix}_{q_i}[1-A_{i j}-r]_{q_i}\\
&\times 
f_i^{-A_{i j}-r}  \frac{k_i q_i^{r+A_{i j}}  -   k_i^{-1}q_i^{ -A_{i j}-r}}{q_i-q_i^{-1}} f_j f_i^r\\
&+ \omega(\xi_i, -\xi_j) \sum_{r=0}^{1-A_{i j}} (- \omega(\xi_i, \xi_j))^r\begin{bmatrix}1-A_{i j}\\ r\end{bmatrix}_{q_i} [r]_{q_i}\\
&\times 
f_i^{1-A_{i j}-r} f_j f_i^{r-1} \frac{k_i q_i^{-r +1}  -   k_i^{-1}q_i^{ r-1}}{q_i-q_i^{-1}}.
\ealn
\]
We can re-write the right hand side as
\[
\baln
&\sum_{r=0}^{1-A_{i j}} (- \omega(\xi_i, \xi_j))^r\begin{bmatrix}1-A_{i j}\\ r\end{bmatrix}_{q_i}[1-A_{i j}-r]_{q_i}\\
&\times 
f_i^{-A_{i j}-r}  f_j f_i^r \frac{k_i q_i^{-r }  -   k_i^{-1}q_i^r}{q_i-q_i^{-1}}\\
&+ \omega(\xi_i, -\xi_j) \sum_{r=0}^{1-A_{i j}} (- \omega(\xi_i, \xi_j))^r\begin{bmatrix}1-A_{i j}\\ r\end{bmatrix}_{q_i} [r]_{q_i}\\
&\times 
f_i^{1-A_{i j}-r} f_j f_i^{r-1} \frac{k_i q_i^{-r +1}  -   k_i^{-1}q_i^{ r-1}}{q_i-q_i^{-1}}.
\ealn
\]
This immediately leads to 
\[
\baln
[e_i, S_{i j}^-]_\omega 
&= \sum_{r=0}^{-A_{i j}} (- \omega(\xi_i, \xi_j))^r f_i^{-A_{i j}-r}  f_j f_i^r \frac{k_i q_i^{-r}  -   k_i^{-1}q_i^r}{q_i-q_i^{-1}}\\
&\times \left( \begin{bmatrix}1-A_{i j}\\ r\end{bmatrix}_{q_i}[1-A_{i j}-r]_{q_i}- \begin{bmatrix}1-A_{i j}\\ r+1\end{bmatrix}_{q_i} [r+1]_{q_i}\right).
\ealn
\]
The expression in the second line vanishes for each $r$. Thus 
$[e_i, S_{i j}^-]_\omega =0$,  
proving the second relation of \eqref{eq:ideal}. 

The first relation of \eqref{eq:ideal} can be obtained from the second by applying the automorphism in Lemma \ref{lem:auto}. 
This completes the proof of Lemma \ref{lem:ideal}.
\end{proof}

\begin{proof}[Proof of Lemma \ref{lem:co-ideal}]
The $n= 1- A_{i j}$ case of Lemma \ref{lem:co-prods}
yields 
\[
\baln
\Delta\left(Ad_{e_i}^{1- A_{i j}}(e_j)\right) &=1\ot Ad_{e_i}^{1- A_{i j}}(e_j)  
+ Ad_{e_i}^{1- A_{i j}}(e_j) \ot  k_i^{1- A_{i j}}k_j,  \\
\Delta\left(Ad_{f_i}^{1- A_{i j}}(f_j)\right) &=Ad_{f_i}^{1- A_{i j}}(f_j) \ot 1 +   k_i^{-1+ A_{i j}}k_j^{-1} \ot  Ad_{f_i}^{1- A_{i j}}(f_j).   
\ealn
\]
These are precisely the formulae \eqref{eq:co-ideal} and \eqref{eq:co-ideal-1} for the 
co-products of $S^+_{i j}$ and $S^-_{i j}$. 
The rest of the lemma easily follows from these formulae. 
\end{proof}

\subsubsection{Structural connections between key lemmas}\label{sect:primit}
We investigate the connection between Lemmas \ref{lem:ideal} and \ref{lem:co-ideal}.

Let us introduce some notation. 
If $\nu=\sum_i \ell_i \Upsilon_i$ in $\Z_+\Pi$, we let $|\nu|=\sum_i \ell_i$, and call it the height of $\nu$. 
For any $\mu\in \Z\Pi$, we let $k_\mu$ be the product of $k_i^{\pm 1}$'s such that 
$k_\nu e_i k_\nu^{-1}=q^{(\Upsilon_i, \nu)} e_i$ for all $i$.

Now the $(\Gamma, \omega)$-subalgebra $\wt\U_+^+$ is also $\Z_+\Pi$-graded, with
$\wt\U_+^+=\sum_{\mu\in \Z_+\Pi}(\wt\U_+^+)_\mu$, 
where $(\wt\U_+^+)_\mu$ will be called the weight $\mu$-subspace. 
Let $E({\bf j})= e_{j_1}e_{j_2}\dots e_{j_n}$,  where ${\bf j}=(j_1, \dots, j_n)$ $\in$ $ I^n$ and $n\in\Z_+$. 
[If $n=0$, we set $E({\bf j})=1$.]  
The weight of the element is $\mu({\bf j}) =\sum_{t=1}^n \Upsilon_{j_t}$. We call 
$n$ the length of $E({\bf j})$, which is equal to the height of $\mu({\bf j})$.  
The elements $E({\bf j})$ form a basis of $\wt\U_+^+$.  
Note that any $E_\mu\in (\wt\U_+^+)_\mu$ with $\mu=\sum_{i\in I} \mu_i\Upsilon_i$ is spanned by $E({\bf j})$'s of length $|\mu|=\sum_i \mu_i$ and weight
$\mu=\mu({\bf j})$. 
For any $E({\bf j})$, we define 
\[
\baln
\hat{E}({\bf j)}_{; i}^{R, \pm}&= \sum_{t=1}^n \delta_{i, j_t}\omega\big(\xi_i, \sum_{s>t} \xi_{j_s}\big) q^{\pm(\Upsilon_i, \sum_{s<t} \Upsilon_{j_s})}  e_{j_1} e_{j_2}\dots  e_{j_{t-1}}e_{j_{t+1}} \dots  e_{j_n}, \\
\hat{E}({\bf j)}_{; i}^{L, \pm}&= \sum_{t=1}^n \delta_{i, j_t} \omega\big(\sum_{s<t} \xi_{j_s}, \xi_i\big) q^{\pm (\Upsilon_i, \sum_{s<t} \Upsilon_{j_s})}  e_{j_1} e_{j_2}\dots   e_{j_{t-1}} e_{j_{t+1}} \dots  e_{j_n }\\
\ealn
\]
and generalise this linearly to   
any $E= \sum_{\bf j} c_{\bf j}E({\bf j})\in (\wt\U_+^+)_\mu$ for $c_{\bf j}\in\K$, with
\beq\label{eq:deriv-i}
\hat{E}_{; i}^{R, \pm}  = \sum c_{\bf j} \hat{E}({\bf j})_{;i}^{R, \pm}, \quad 
\hat{E}_{; i}^{L, \pm}  = \sum c_{\bf j} \hat{E}({\bf j})_{;i}^{L, \pm}.
\eeq
Note that 
\beq\label{eq:lr-der}
\hat{E}_{; i}^{R, \pm}= \omega\big(\xi_i, d E -\xi_i\big) 
\hat{E}_{; i}^{L, \pm}.
\eeq

We have the following easy facts. 
\begin{lemma}\label{lem:primit}
Retain notation above. For any $E\in (\wt\U_+^+)_\mu$ with $\mu\ge 2$, 
\[
{}
[f_i, E]_\omega= - \omega(\xi_i, \xi_i) \frac{k_i \hat{E}_{;i}^{L, -} - k_i^{-1} \hat{E}_{;i}^{L, +}}{q_i-q_i^{-1}}, \quad \forall i, 
\]
\[
\baln
\Delta(E)&=E\ot k_\mu +\sum _i\hat{E}_{;i}^{R, +} \ot e_i  k_{\mu-\Upsilon_i}+\dots +  \sum _i e_i\ot k_i \hat{E}_{;i}^{L, -}  + 1\ot E,
\ealn
\]
where  the dots represent terms of the form $X_\nu \ot \wt{X}_{\mu-\nu}\in \wt\U_0(\wt\U_+^+)_\nu\ot  \wt\U_0(\wt\U_+^+)_{\mu-\nu}$ with both $|\nu|$, $|\mu-\nu|$  $\ge 2$. 
\end{lemma} 
\begin{proof}
The formulae immediately follows from the definitions of $\hat{E}_{; i}^{R, \pm}, 
\hat{E}_{; i}^{L, \pm}$. To prove them, we may take $E$ to be $E({\bf j})= e_{j_1}e_{j_2}\dots e_{j_n}$. Then 
\begin{align*}
[f_i,  E]_\omega 
&=\frac{-\omega(\xi_i, \xi_i)}{q_i-q_i^{-1}}  \sum_{r=1}^n \delta_{i j_r} \omega\big(\xi_i, \sum_{s<r} \xi_{j_s}\big)   e_{j_1}  e_{j_{r-1}} (k_i - k_i^{-1}) e_{j_{r+1}} \dots e_{j_n}\\
&=\frac{-\omega(\xi_i, \xi_i)}{q_i-q_i^{-1}} k_i \sum_{r=1}^n \delta_{i j_r} \omega\big(\xi_i, \sum_{s<r} \xi_{j_s}\big)   q^{-(\Upsilon_i, \sum_{s<r} \Upsilon_{j_s})}e_{j_1}  e_{j_{r-1}} e_{j_{r+1}} \dots e_{j_n} \\
&- \frac{-\omega(\xi_i, \xi_i)}{q_i-q_i^{-1}}  k_i^{-1} \sum_{r=1}^n \delta_{i j_r} \omega\big(\xi_i, \sum_{s<r} \xi_{j_s}\big)   q^{(\Upsilon_i, \sum_{s<r} \Upsilon_{j_s})} e_{j_1}  e_{j_{r-1}} e_{j_{r+1}} \dots e_{j_n}. 
\end{align*}
Using the definitions of $\hat{E}_{;i}^{L, \pm}$, we obtain
\begin{align*}
[f_i,  E]_\omega 
&= -\omega(\xi_i, \xi_i) \frac{ k_i \hat{E}_{;i}^{L, -} -  k_i^{-1} \hat{E}_{;i}^{L, +} }{q_i-q_i^{-1}}.
\end{align*}
Note that this is equivalent to 
\beq
[E_\mu, f_i]_\omega= \frac{k_i \hat{E}_{;i}^{R, -} - k_i^{-1}  \hat{E}_{;i}^{R, +}}  
{q_i-q_i^{-1}}.
\eeq

To prove the second formula, note that the co-product of $E$ can always be written as $\Delta(E=\sum X_\nu \ot \wt{X}_{\mu-\nu}$ where $\mu=\mu({\bf j})$. Clearly the term with $|\mu-\nu|=0$ is $E\ot k_\mu$, and that with $|\nu|=0$ is $1\ot E$. The sum of the terms with $|{\mu-\nu}|=1$ is given by 
\begin{align*}
&\sum_{r=1}^n \omega\big(\xi_{j_r}, \sum_{s>r} \xi_{j_s}\big)   e_{j_1}\dots e_{j_{r-1}} e_{j_{r+1}} \dots e_{j_r}\ot k_{j_1}\dots k_{j_{r-1}} e_{j_r} k_{j_{r+1}} \dots k_{j_r}\\
&=  \sum_{r=1}^n \omega\big(\xi_{j_r}, \sum_{s>r} \xi_{j_s}\big)   q^{(\Upsilon_{j_r}, \sum_{s<r} \Upsilon_{j_s})}e_{j_1}\dots e_{j_{r-1}} e_{j_{r+1}} \dots e_{j_n}\ot e_{j_r} k_{\mu-\Upsilon_{j_r}} \\
&= \sum_i \sum_{r=1}^n \delta_{i j_r}\omega\big(\xi_i, \sum_{s>r} \xi_{j_s}\big)   q^{(\Upsilon_i, \sum_{s<r} \Upsilon_{j_s})}e_{j_1}\dots e_{j_{r-1}} e_{j_{r+1}} \dots e_{j_n}\ot e_i k_{\mu-\Upsilon_i} \\
&=\sum_i \hat{E}_{;i}^{R, +}\ot e_i k_{\mu-\Upsilon_i}. 
\end{align*}
Similar calculations show that the sum of the terms with $|\nu|=1$ is that given in the claimed formula.
\end{proof}

\begin{remark}\label{rmk:wtU-}
We can  analyse  $\wt\U_-^-$ similarly to obtain an analogue of Lemma \ref{lem:primit} for 
any $F\in (\wt\U_-^-)_{-\mu}$ with $\mu\ge 2$. 
\end{remark}

We call homogeneous elements $E\in (\wt\U_+^+)_\mu$ and $F\in (\wt\U_-^-)_{-\mu}$ skew primitive if 
\[
\Delta(E)=E\ot k_\mu + 1\ot E, \quad \Delta(F)=F\ot 1+  k_\mu^{-1}\ot F. 
\]
We will say that homogeneous elements $E\in (\wt\U_+^+)_\mu$ and $F\in (\wt\U_-^-)_{-\mu}$  are singular if 
\beq\label{eq:sing}
[f_i, E]_\omega=0, \quad [e_i, F]_\omega=0, \quad \forall i.
\eeq

\begin{lemma}\label{lem:primit-singular}
If $E\in (\wt\U_+^+)_\mu$ and $F\in (\wt\U_-^-)_{-\mu}$ with $|\mu|\ge 2$ are skew primitive, then they are singular. 
\end{lemma}
\begin{proof}
It follows from the second relation in Lemma \ref{lem:primit} that $\hat{E}_{;i}^{R, +}=0$ and $\hat{E}_{;i}^{L, -}=0$ for all $i$. By \eqref{eq:lr-der}, we have $\hat{E}_{;i}^{R, \pm}=\hat{E}_{;i}^{L, \pm}=0$ for all $i$. Then the first formula of Lemma \ref{lem:primit} implies that $[f_i, E]=0$ for all $i$. Hence $E$ is singular. 

To prove the singularity of $F$, we need to appeal to Remark \ref{rmk:wtU-}. We omit the details.  
\end{proof}

Now we apply the lemma to the quantum Serre elements $S^\pm_{i j}$ for $i\ne j$. 
Lemma \ref{lem:ideal} states that $S^\pm_{i j}$'s are singular, and Lemma \ref{lem:co-ideal} shows that they are skew primitive. Hence it immediately follows from Lemma \ref{lem:primit-singular} that
\begin{corollary}
Lemma \ref{lem:co-ideal} implies Lemma \ref{lem:ideal}. 
\end{corollary}
This clarifies the relationship between the  two lemmas. 

It is easy to see that a singular element in $(\wt\U_+^+)_\mu$ is not necessarily skew primitive.  
For example, the square of any homogeneous skew primitive element 
$E$ in $\wt\U_+^+$ is not skew primitive,  
but it is singular as $[f_i, E^2]_\omega = [f_i, E]_\omega E + \omega(d f_i, d E) E [f_i, E]_\omega$ vanishes for all $i$,  since $[f_i, E]_\omega=0$ by Lemma \ref{lem:primit-singular}. 

However, we have the following comments. [{\bf No later result depends on these comments}.]
 Call a homogeneous singular element of $(\wt\U_+^+)_\mu$ elementary if it can not be expressed as a polynomial of homogeneous singular elements in $(\wt\U_+^+)_\nu$, for  all $\nu$  with heights $|\nu|<|\mu|$. We expect that 
elementary homogeneous singular elements in $(\wt\U_+^+)_\mu$ with
$|\mu|\ge 2$ are skew primitive. This is true if the height of $\mu$ is  $2$.

\subsection{Comments on the construction of colour quantum groups}\label{sect:construc}

The construction of the colour quantum group $\U_{q, \Xi}(A)$ 
as a Hopf quotient of the auxiliary colour quantum group
$\wt\U_{q, \Xi}(A)$ (see Definition \ref{def:U-aux})
has several structural consequences.

\smallskip
\noindent{1).}
Let $z^+_r \in (\wt\U^+_+)_\mu$ and $z^-_r \in (\wt\U^-_-)_{-\mu}$ with $|\mu|\ge 2$, for $r=1, 2, \dots, N$, 
be homogeneous elements in both the $\Gamma$-grading and $\Z\Pi$-grading. 
Assume that the elements are skew primitive. It follows from Lemma \ref{lem:primit-singular} that 
they are singular, that is, 
\beq\label{eq:no-intersect}
[f_i, z^+_r]_\omega=[e_i, z^-_r]_\omega=0, \quad \forall i. 
\eeq
The two-sided ideal $\wt\CJ$ generated by these elements 
is a homogeneous Hopf ideal, and consequently 
\[
\U^{intmd}_{q, \Xi}(A) =\wt\U_{q, \Xi}(A)/\wt\CJ
\]
is a Hopf $(\Gamma, \omega)$-algebra. 
Note that $\wt\CJ$ admits a triangular decomposition and intersects $\wt{Gen}=\sum_i \wt\U^0 e_i + \wt\U^0 + \sum_i \wt\U^0f_i$ trivially.    
The quotient retains the generators of $\wt\U_{q, \Xi}(A)$.

\smallskip
\noindent{2).}
For every simple highest or lowest weight $\wt\U_{q, \Xi}(A)$-module $M$ over $\C(q)$, the submodule $\wt\CJ M$ vanishes. 
To explain this, we assume that $M$ is a simple  highest  weight module generated by the highest weight vector $v_\lambda$. When \(\wt\CJ\) acts on \(v_\lambda\), the factor \(\wt\CJ^+ = \wt\CJ\cap\wt\U^+_+\) annihilates \(v_\lambda\), while \(\wt\U^0\) acts by scalars. Hence \(\wt\CJ v_\lambda\) is contained in \(\wt\CJ^-v_\lambda\), 
where $\wt\CJ^- = \wt\CJ\cap\wt\U^-_-$. Since \(\wt\CJ^-\) is generated by the elements \(z_i^-\), all of which have strictly negative weight, every homogeneous element of \(\wt\CJ^-\) has strictly negative weight. Consequently,
\[
\wt\CJ v_\lambda\subseteq\bigoplus_{\mu>0}M_{\lambda-\mu}, 
\]
and in particular \(v_\lambda\notin\wt\CJ M\). Since \(\wt\CJ M\) is a submodule of the irreducible module \(M\), it follows that \(\wt\CJ M=0\). The proof for a simple lowest weight module is similar. 
It follows that 
\begin{lemma}
Any irreducible highest or lowest weight representation of $\wt\U_{q, \Xi}(A)$ factors through $\U^{intmd}_{q, \Xi}(A)$. 
\end{lemma}Thus, the passage from $\wt\U_{q, \Xi}(A)$ to $\U^{intmd}_{q, \Xi}(A)$ leaves the category of irreducible highest and lowest weight representations unchanged, even though $\wt\U_{q, \Xi}(A)$ is strictly larger as an algebra.

\smallskip
\noindent{3).}
It is worth emphasising that in the construction of $\U_{q, \Xi}(A)$, the quantum Serre relations are not imposed \textit{a priori}. Instead, they are determined intrinsically by the input data $A$ and $\Xi$.
As we shall see in Section \ref{sect:U+}, 
skew primitive elements of heights greater than $1$  induce degeneracy of the Hopf pairing between the two auxiliary quantised Borel subalgebras $\wt\U^\pm_{q, \Xi}(A)$ of $\wt\U_{q, \Xi}(A)$.
As we will prove in Lemma \ref{lem:Nichols} and the analogous result for \(\wt\U^-_{q,\Xi}(\fg)\), the homogeneous Hopf ideals generated by the quantum Serre elements \(S^+_{ij}\) and \(S^-_{ij}\) are precisely the right and left radicals of the Hopf pairing. 
Consequently, the homogeneous skew primitive elements \(z_r^\pm\) belong to $\CJ$, 
and hence the homogeneous Hopf ideal $\wt\CJ$ generated by them is contained in $\CJ$. 
Therefore, we have the following fact.
\begin{lemma}
The canonical Hopf $(\Gamma, \omega)$-algebra surjection 
$\wt\U_{q, \Xi}(A)\lra \U_{q, \Xi}(A)$ factors through every intermediate colour quantum group 
$\wt\U^{intmd}_{q, \Xi}(A)$. 
\end{lemma}
In this sense, the Serre quotient  $\U_{q, \Xi}(A)$ is canonical. 

\smallskip
\noindent{4).}
The intermediate colour quantum group $\wt\U^{intmd}_{q, \Xi}(A)$ (or its specialisation to $\C((\hbar))$) does not admit a universal $R$-matrix in the strict sense. 
Nevertheless, in the category ${\wt\U^{intmd}_{q, \Xi}(A)}$-${\rm\bf Mod}_{\U_{q, \Xi}(A)}$
of $\wt\U^{intmd}_{q, \Xi}(A)$-modules which are pullbacks of
$\U_{q, \Xi}(A)$-modules, 
the universal \(R\)-matrix of \(\U_{q,\Xi}(\fg)\) obtained in Theorem \ref{thm:QD}  give rise to an  \(R\)-matrix, which is functorial.  Therefore we have the following result, granting Theorem \ref{thm:QD}. 
\begin{lemma}
There exists a  functorial \(R\)-matrix on the $\wt\U^{intmd}_{q, \Xi}(A)$-module category ${\wt\U^{intmd}_{q, \Xi}(A)}$-${\rm\bf Mod}_{\U_{q, \Xi}(A)}$.
\end{lemma}

We note in particular that if one is interested in irreducible representations of colour quantum groups and related $R$-matrices, 
it is legitimate to work with $\wt\U_{q, \Xi}(A)$ or any intermediate colour quantum group $\U^{intmd}_{q, \Xi}(A)$.

\begin{remark} Our construction of $\U_{q, \Xi}(A)$ is a direct generalisation to the present setting of that in \cite[\S1.2]{Kv-2} for Lie algebras admitting a Cartan-Weyl description.  
The same construction of $\U_q(\gl_{m|n})$ as the Hopf quotient of a corresponding auxiliary quantum supergroup was developed long ago in \cite{Z93}. 
\end{remark}

\section{Quasi-triangularity of colour quantum groups}\label{sect:QT}
The theory of quasi-triangular Hopf algebras \cite{D2, D3} and superalgebras \cite{GZB93}, 
including the quantum double construction, 
naturally generalises to Hopf $(\Gamma, \omega)$-algebras (see Appendix \ref{sect:double}). 
We will prove the quasi-triangular Hopf colour algebraic structure of 
the quantised universal enveloping colour algebra $\U_{q, \Xi}(A)$. 

In this section, we shall work over the formal Laurent series ring $\C((\hbar))$ in
a category where the topological datum is controlled by an underlying $\C[[\hbar]]$-lattice. 
The general setup of topological Hopf $(\Gamma, \omega)$-algebras suitable for our purpose 
is discussed in Section \ref{sect:topo}. 
We will freely use material in Appendix \ref{sect:double}. 

The necessity of working over $\C((\hbar))$ instead of working over $\C[[\hbar]]$ then extnding scalars
is explained in Remark \ref{rmk:necessity}.

Denote $R=\C[[\hbar]]$ and $\K=\C((\hbar))$ throughout this section.  For convenience, we write $q=\exp(\hbar)$ $:=\sum_{i\ge 0} \frac{\hbar^i}{i!}$ 
and $q_i= \exp(\hbar d_i/2)$, for all $i\in I$. 

\subsection{Colour quantum groups over Laurent series ring}

We adopt the convention and notation for finite and affine root systems in Section \ref{sect:def-Uq}.
In particular, we use $A=(A_{i j})_{i, j \in I}$ to denote either a finite type Cartan matrix
with $I=\{1, 2, \dots, \ell\}$
or affine type Cartan matrix with $I=\{0, 1, 2, \dots, \ell\}$.  
We also have the sequence $\Xi=(\xi_i)_{i\in I}$ in each case, where we recall that $\xi_0$ is defined by \eqref{eq:Upsil0} in the affine case. Let $(\Phi, \Pi, \Xi)$ be the root datum associated with $A$, 
where $\Phi$ is the set of roots, and $\Pi=\{\Upsilon_i\mid i\in I\}$ is the set of simple roots. 

\begin{definition}
Let  $\U_\hbar(A; \Xi)_R$ be the topological Hopf $(\Gamma, \omega)$-algebra over $R$ generated by $e_i, f_i, h_i$, for $i\in I$, subject to the defining relations which are formally the same as  
those given in Definition \ref{def:main} but with 
\begin{align}\label{eq:ks}
&k_i^{\pm 1}= \exp(\pm \frac{\hbar  d_i}{2} h_i)
	:=\sum_{n=0}^\infty \frac{\hbar^n}{ n!} \left(\frac{\pm h_i  d_i}{2}\right)^n,   \\
&\frac{k_i-k_i^{-1}}{q_i-q_i^{-1}}=h_i\frac{\sum_{n=0}^\infty \frac{\hbar^{2n}}{ (2n+1)!} \left(\frac{h_i  d_i}{2}\right)^{2n}}{\sum_{n=0}^\infty \frac{\hbar^{2n}}{ (2n+1)!} \left(\frac{d_i}{2}\right)^{2n}}, 
\end{align} 
and the Hopf algebra structure is as described in Theorem \ref{thm:Hopf}.    
\end{definition}

It then follows from \eqref{eq:ks} that the elements $h_i$ commute among themselves and satisfy the following relations with the elements $e_i, f_i$:
\beq
h_i e_j -  e_j h_i= A_{i j} e_j,  , \quad h_i f_j -  f_j h_i=- A_{i j} f_j.  
\eeq
Their co-products are given by 
\beq
\Delta(h_i)=h_i\ot 1+ 1\ot h_i.
\eeq

\begin{definition}\label{def:special}
Retain notation above and set
\[
\U_\hbar(A;\Xi)
:=
\U_\hbar(A;\Xi)_R\ot_R\K.
\]
We regard $\U_\hbar(A;\Xi)$ as a lattice equipped topological Hopf
$(\Gamma,\omega)$-algebra using the distinguished lattice
$\U_\hbar(A;\Xi)_R$ and the lattice topology of Definition \ref{def:latt}.
\end{definition}

We will show that $\U_\hbar(A; \Xi)$ has the structure of a quasi-triangular topological Hopf $(\Gamma, \omega)$-algebra, by reconstructing it from the quantum double (see Definition \ref{def:D(H)}) of a quotient of 
the topological Hopf $(\Gamma, \omega)$-algebra $\wt\U^+_\hbar(A;\Xi)$ 
to be defined presently.

To understand the relationship of these Hopf $(\Gamma, \omega)$-algebras with the colour quantum group of Definition \ref{def:main}, we fix the ring homomorphism 
\beq
\psi: \C(q)\lra \K, \quad q\mapsto \exp(\hbar):=\sum_{i\ge 0} \frac{\hbar^i}{i!},  
\eeq
and consider the specialisation 
$
\U_{q, \Xi}(A)\ot_\psi \K
$
of the Hopf $(\Gamma, \omega)$-algebra $\U_{q, \Xi}(A)$ with respect to $\psi$.  
We introduce extra generators $h_i$ (for $i\in I$), and impose the relations
\begin{align*}
k_i^{\pm 1}\ot_\psi 1= \sum_{n=0}^\infty \frac{\hbar^n}{ n!} \left(\frac{\pm h_i  d_i}{2}\right)^n, 
\quad \forall i.
\end{align*}
Then the topological Hopf $(\Gamma, \omega)$-algebra generated by the specialisation 
$
\U_{q, \Xi}(A)\ot_\psi \K
$ 
and the elements $h_i$ is isomorphic to $\U_\hbar(A; \Xi)$.

\subsection{Auxiliary quantised colour Borel subalgebras}\label{sect:U+}

Let us introduce the following topological $(\Gamma, \omega)$-algebra.
\begin{definition} \label{def:wtU+}
Let  $\wt\U^+_\hbar(A;\Xi)_R$ be the topological Hopf $(\Gamma, \omega)$-algebra over $R$, 
which is generated 
by homogeneous elements $e_i, h_i$, for $i\in I$, 
with $d(e_i)=\xi_i$ and $d(h_i)=0$, subject to the relations 
\beq
&&h_i h_j -  h_j h_i =0, \quad h_i e_j -  e_j h_i= A_{i j} e_j,  \quad \forall i\in I, 
\eeq
and has co-multiplication $\Delta$,  co-unit $\varepsilon$ and antipode $S$,  given by 
\beq
&\Delta(h_i)=h_i\ot 1+ 1\ot h_i, \quad \Delta(e_i)= e_i\ot k_i + 1\ot e_i; \\
& \varepsilon(h_i)=0, \quad \varepsilon(e_i)=0; \\
&S(h_i)=-h_i, \quad S(e_i)=  - e_i k_i^{-1}, \quad \forall i\in I, 
\eeq
where $k_i^{\pm 1}$ are formally defined by the same equation as \eqref{eq:ks} 
but interpreted in the present context. Define the following Hopf $(\Gamma, \omega)$-algebra over $\K$.
\beq
\wt\U^+_\hbar(A;\Xi):= \wt\U^+_\hbar(A;\Xi)_R\ot_R\K 
\eeq
equipped with the lattice topology of Definition \ref{def:latt}.
\end{definition}

Consider the subalgebras $\wt\U_0=\langle h_i\mid i\in I\rangle$ 
and $\wt\U_+=\langle e_i\mid i\in I\rangle$ of $\wt\U^+_\hbar(A;\Xi)$, 
both of which  inherit the lattice topology. 
For any ${\bf m}=(m_i)_{i\in I}\in\Z_+^{|I|}$, let 
$h^{\bf m}=\prod_{i\in I} h_i^{m_i}$. Then $\wt\U_0$
has the monomial  basis $\{h^{\bf m}\mid {\bf m}\in\Z_+^{|I|}\}$.
Here the displayed monomials form a \(\K\)-basis of the underlying algebra; the \(\hbar\)-adic completion is understood coefficient-wise.

We also adapt some notion introduced in Section \ref{sect:primit} to the subalgebra $\wt\U_+$, 
in particular, 
its $\Z_+\Pi$-gradation $\wt\U_+=\sum_{\mu\in \Z_+\Pi}(\wt\U_+)_\mu$.  
For any ${\bf j}=(j_1, \dots, j_n)$ $\in$ $ I^n$ and $n\in\Z_+$, we again denote 
$E({\bf j})= e_{j_1}e_{j_2}\dots e_{j_n}$. 
Furthermore, we have the following basis for $\wt\U^+_\hbar(A;\Xi)$.
\beq
h^{\bf m} E({\bf j}), \quad \text{for   ${\bf m}\in\Z_+^{|I|}$,  ${\bf j}\in I^n$, $n\in\Z_+$}.
\eeq

Note the following formula for the co-product of $h^{\bf m}$.
\[
\Delta(h^{\bf m}) =\Delta'(h^{\bf m})=\sum_{\bf r} 
\begin{pmatrix}
{\bf m}\\ {\bf r}
\end{pmatrix}
h^{{\bf m}-{\bf r}}\ot h^{\bf r}, 
\]
where 
$
\begin{pmatrix}
{\bf m}\\ {\bf r}
\end{pmatrix}=\prod_{i\in I} \begin{pmatrix}
m_i\\  r_i
\end{pmatrix}.
$
For later use, we also set ${\bf m}!= \prod_{i\in I}m_i!$.

The  lattice continuous dual  space $\wt\U^+_\hbar(A;\Xi)^\vee$ of $\U^+_\hbar(A;\Xi)$ (i.e., the space of lattice continuous $\K$-linear functions, see \eqref{eq:V-vee})  is 
a $(\Gamma, \omega)$-algebra with the multiplication 
$
\Delta^*: \wt\U^+_\hbar(A;\Xi)^\vee\wh\ot \wt\U^+_\hbar(A;\Xi)^\vee\lra \wt\U^+_\hbar(A;\Xi)^\vee.
$

Denote by $\wt\U^+_\hbar(A;\Xi)'$ the topological $(\Gamma, \omega)$-subalgebra of $\wt\U^+_\hbar(A;\Xi)^\vee$ 
generated by the homogeneous elements $f_i, \wt{H}_i$ ($ i\in I$) of $\wt\U^+_\hbar(A;\Xi)^\vee$, 
with $\Gamma$-degrees 
$d(f_i)=-\xi_i$ and $d(\wt{H}_i)=0$ respectively,  such that
\[
\varphi_i =(q_i-q_i^{-1}) f_i, \quad \rho_i=\hbar \frac{ d_i}{2} \wt{H}_i, \quad i\in I, 
\]
satisfy the following relations for any $b_\mu\in (\wt\U_+)_\mu$:
\beq\label{eq:pair-def}
\phantom{XXX}
\left\langle \rho_i,  h^{\bf n} b_\mu\right\rangle
=\delta_{\mu 0} \delta_{{\bf n} {\bf 1}_i}, 
\quad \langle \varphi_i,  h^{\bf n} e_j\rangle=\delta_{i j} \delta_{\bf n 0},  
\quad
\langle \varphi_i,  h^{\bf n} b_\mu\rangle=0  \text{ if }  \mu\not\in\Pi, 
\eeq
where $\delta_{\bf n m}=\prod_{r\in I} \delta_{n_r, m_r}$, and 
${\bf 1}_i=(0, \dots, 0, \underbrace{1}_i, 0, \dots, 0)\in \Z_+^{|I|}$.

\begin{lemma}\label{lem:dU-dU}
The following relations hold in $\wt\U^+_\hbar(A;\Xi)'$. 
\beq
&&	\rho_i \rho_j -  \rho_j \rho_i=0, \label{eq:rho-rho}\\
&& 	\rho_i \varphi_j -  \varphi_j \rho_i = - \delta_{i j} \frac{\hbar d_i}{2} \varphi_j,  
		\quad \forall i, j\in I. \label{eq:rho-phi}
\eeq
\end{lemma}
\begin{proof}

Consider \eqref{eq:rho-rho} first.  We have
\beq\label{eq:wt-match}
\langle \rho_i \rho_j  -  \rho_j \rho_i, h^{\bf m} b_\mu\rangle = \langle \rho_i\ot \rho_j  -  \rho_j\ot \rho_i, \Delta(h^{\bf m} b_\mu)\rangle.   
\eeq
The co-product $\Delta(h^{\bf m} b_\mu)$ is the sum of terms of the form $h^{\bf k} X_\nu \ot h^{\bf n} X'_{\nu'}$ with $X_\nu\in (\wt\U_+)_\nu$ and $X'_{\nu'}\in (\wt\U_+)_{\nu'}$ such that $\nu+\nu'=\mu$.  If $\mu\ne 0$, at least one of $\nu, \nu'$ is non-zero. Hence the right hand side of 
\eqref{eq:wt-match} must vanish by the first relation of \eqref{eq:pair-def}. This shows that 
$\langle \rho_i \rho_j  -  \rho_j \rho_i, h^{\bf m} b_\mu\rangle=0$ if $\mu\ne 0$. 
Now
\[
\langle \rho_i \rho_j -  \rho_j \rho_i, h^{\bf m}\rangle
= \langle \rho_i \ot \rho_j -  \rho_j \ot \rho_i, \Delta(h^{\bf m}) \rangle=- \langle \rho_i \ot \rho_j -  \rho_j \ot \rho_i, \Delta'(h^{\bf m}) \rangle. 
\]
Since $\Delta(h^{\bf m}) = \Delta'(h^{\bf m})$, we  obtain $\langle \rho_i \rho_j -  \rho_j \rho_i, h^{\bf m}\rangle=0$, and hence $\rho_i \rho_j -  \rho_j \rho_i=0$.

To prove \eqref{eq:rho-phi}, we first note that 
$\langle \rho_i \varphi_j -  \varphi_j \rho_i,   h^{\bf m} b_\mu\rangle=0$ if $\mu\not\in \Pi$. 
Indeed, 
\[
\langle \rho_i \varphi_j -  \varphi_j \rho_i,   h^{\bf m} b_\mu\rangle
=\langle \rho_i \ot \varphi_j -  \varphi_j \ot \rho_i,   \Delta(h^{\bf m} b_\mu)\rangle.
\]
Express $\Delta(h^{\bf m} b_\mu)$ as a sum of terms of the form $h^{\bf k} X_\nu \ot h^{\bf n} X'_{\nu'}$ with $X_\nu\in (\wt\U_+)_\nu$ and $X'_{\nu'}\in (\wt\U_+)_{\nu'}$ such that $\nu+\nu'=\mu$, and observe that

$\bullet$
$
\langle \rho_i \ot \varphi_j,  h^{\bf k} X_\nu \ot h^{\bf n} X'_{\nu'}\rangle 
$
 is a scalar multiple of  $\delta_{\bf 0 \nu}\langle \varphi_j,  h^{\bf n} X'_{\mu}\rangle$, and  

$\bullet$
$
\langle \varphi_j \ot \rho_i,   h^{\bf k} X_\nu \ot h^{\bf n} X'_{\nu'})\rangle
$
is a scalar multiple of  $\delta_{\bf 0 \nu'}\langle \varphi_j,  h^{\bf k} X_{\mu}\rangle$. \\
If $\mu\not\in \Pi$,  both $\langle \varphi_j,  h^{\bf n} X'_{\mu}\rangle$ and $\langle \varphi_j,  h^{\bf k} X_{\mu}\rangle$ vanish by the third relation of \eqref{eq:pair-def}.  Hence the right hand side of the equation above vanishes. 

Now
\[
\baln
\langle \rho_i \varphi_j -  \varphi_j \rho_i,  h^{\bf m} e_s\rangle
&=\langle \rho_i \ot \varphi_j -  \varphi_j \ot \rho_i, 
\Delta(h^{\bf m}) (e_s\ot k_s+ 1 \ot e_s)  \rangle\\
&=\delta_{j s} \langle \rho_i \ot \varphi_j,  
 \Delta(h^{\bf m}) (1 \ot e_j)\rangle \\
&- \delta_{j s} \langle \varphi_j \ot \rho_i, 
 \Delta(h^{\bf m})(e_j\ot k_j)\rangle.
\ealn
\] 
By using the following relations
\[
\baln
&\langle \rho_i \ot \varphi_j,    \Delta(h^{\bf m}) (1 \ot e_j)\rangle = m_i  \langle \varphi_j,    h^{{\bf m}-{\bf 1}_i} e_j \rangle,  \\
&\langle \varphi_j \ot \rho_i,  \Delta(h^{\bf m}) (e_j\ot k_j)\rangle=
 m_i  \langle \varphi_j,   h^{{\bf m}-{\bf 1}_i} e_j \rangle + \delta_{i j} \hbar \frac{d_i}{2} \langle \varphi_j,   h^{\bf m} e_j\rangle,
\ealn
\] 
we obtain 
\[
\baln
\langle \rho_i \varphi_j -  \varphi_j \rho_i, h^{\bf m} e_s \rangle
&=- \delta_{i j} \hbar \frac{d_i}{2} \langle \varphi_j,  h^{\bf m} e_s  \rangle. 
\ealn
\] 
Hence $\rho_i \varphi_j -  \varphi_j \rho_i = - \delta_{i j} \frac{\hbar d_j}{2} \varphi_j$, proving \eqref{eq:rho-phi}. 
\end{proof}

Let $\wt\U'_0$ be the topological subalgebra of $\wt\U^+_\hbar(A;\Xi)'$ 
generated by the $\rho_i$'s, 
and $\wt\U'_-$ that generated by the $\varphi_i$'s. 
A basis for $\wt\U_0'$ is given by
$\{\rho^{\bf m}=\prod_{i\in I}\rho_i^{m_i} \mid {\bf m}=(m_i)_{i\in I}\in\Z_+^{|I|}\}$. 
Now $\wt\U'_+$ is $-\Z_+\Pi$-graded.  
For any ${\bf j}=(j_1, \dots, j_n)$ $\in$ $ I^n$, where $n\in\Z_+$, let  
$\phi({\bf j})= \varphi_{j_1}\varphi_{j_2}\dots \varphi_{j_n}$. 
Call 
$\mu({\bf j}) =-\sum_{t=1}^n \Upsilon_{j_t}$ the weight, and $n$ the length, of $\phi({\bf j})$.
The elements $\phi({\bf j})$ form a basis of $\wt\U'_+$, 
and we have the following basis for $\wt\U^+_\hbar(A;\Xi)'$.
\beq
\rho^{\bf m} \phi({\bf j}), \quad \text{for   ${\bf m}\in\Z_+^{|I|}$,  ${\bf j}\in I^n$, $n\in\Z_+$}.
\eeq

\begin{lemma}\label{lem:beta-b} 
The following relations hold for all $b_\nu\in (\wt\U_+)_\nu$ and $\beta_\mu\in (\wt\U'_+)_{\mu}$ with $-\mu, \nu \in \Z_+\Phi^+$, and 
${\bf m}, {\bf n} \in\Z_+^{|I|}$. 
\beq
\langle \rho^{\bf m},  h^{\bf n} b_\nu\rangle&=&\delta_{\nu 0} \delta_{\bf m n}\prod_{r\in I} m_r!, \label{eq:beta-b-0}\\
\langle \rho^{\bf m}\varphi_i,  h^{\bf n} e_j\rangle&=&\delta_{i j} \delta_{\bf m n}\prod_{r\in I} m_r!,  \label{eq:beta-b-1}\\
 \langle \rho^{\bf m}\varphi_i,  h^{\bf n} b_\nu\rangle&=&0, \quad \text{if $\nu\ne \Upsilon_i$},   
\label{eq:beta-b-1-1}\\
\left\langle \rho^{\bf m}\beta_{\mu},  h^{\bf n} b_\nu\right\rangle&=&0
\quad \text{if $\mu+\nu\ne 0$},  \label{eq:beta-b-2}
\eeq

\end{lemma}

\begin{proof}
The $|{\bf m}|=1$ case of \eqref{eq:beta-b-0} is the first relation of \eqref{eq:pair-def}. 
Using it in
\[
\baln
\langle \rho^{\bf m},  h^{\bf n} b_\nu\rangle&= \langle \rho_i\ot \rho^{{\bf m}-{\bf 1}_i},  \Delta(h_i^{n_i})\Delta(h^{{\bf n}- n_i{\bf 1}_i } b_\nu)\rangle,
\ealn
\]
we can reduce the right hand side to
$
n_i \langle \rho_i, h_i\rangle\langle \rho^{{\bf m}-{\bf 1}_i},  h^{{\bf n}- {\bf 1}_i } b_\nu\rangle =n_i \langle \rho^{{\bf m}-{\bf 1}_i},  h^{{\bf n}- {\bf 1}_i } b_\nu\rangle
$
by using the explicit formulae for $\Delta(h_j)$ and $\Delta(e_j)$. Thus 
\[
\baln
\langle \rho^{\bf m},  h^{\bf n} b_\nu\rangle&=n_i \langle \rho^{{\bf m}-{\bf 1}_i},  h^{{\bf n}- {\bf 1}_i } b_\nu\rangle.
\ealn
\]
An easy induction on $|{\bf m}|$ proves \eqref{eq:beta-b-0}.

Now $\langle \rho^{\bf m}\beta_\mu,  h^{\bf n} b_\nu\rangle
= \langle \rho^{\bf m}\ot \beta_\mu,  \Delta(h^{\bf n} b_\nu)\rangle$. 
Consider the right hand side using \eqref{eq:beta-b-0} 
and the explicit formula for $\Delta(e_j)$ for all $j$, we obtain 
\[
\baln
\sum_{\bf r}\begin{pmatrix}
{\bf n}\\ {\bf r}
\end{pmatrix} \langle \rho^{\bf m}, h^{\bf r} \rangle \langle \beta_\mu,  h^{{\bf n}-{\bf r}}b_\nu\rangle
&= \sum_{\bf r}\begin{pmatrix}
{\bf n}\\ {\bf r}
\end{pmatrix} \delta_{\bf m r} \langle \beta_\mu,  h^{{\bf n}-{\bf r}}b_\nu\rangle.
\ealn
\]
Hence 
\beq\label{eq:pair-general}
\langle \rho^{\bf m}\beta_\mu,  h^{\bf n} b_\nu\rangle
= 
\left\{
\begin{split}
&\langle \beta_\mu,  h^{{\bf n}-{\bf m}}b_\nu\rangle, \quad \text{if $m_i\le n_i$ for all $i$},\\
&0,  \quad \text{otherwise}.
\end{split}
\right.
\eeq
The $\beta_\mu=\varphi_i$ case of this yields \eqref{eq:beta-b-1} and \eqref{eq:beta-b-1-1} by using the second and third relations of \eqref{eq:pair-def}.  

Equation \eqref{eq:beta-b-2} follows from \eqref{eq:pair-general} 
if we can show that $\langle \beta_\mu,  h^{\bf r}b_\nu\rangle=0$ for $\mu+\nu\ne 0$.  
This is true for $|\mu|\le 1$ by the second relation of \eqref{eq:pair-def}.  
For other $\mu\in -\Z_+\Phi^+$, 
we can always express $\beta_{\mu}$ as a linear combination of $\beta_{\mu +\Upsilon_i} \varphi_i$, 
where $\beta_{\mu +\Upsilon_i} \in (\U_+')_{\mu+\Upsilon_i}$ and $i\in I$. Now
$
\langle \beta_{\mu +\Upsilon_i} \varphi_i,  h^{\bf r} b_\nu\rangle
= \langle \beta_{\mu +\Upsilon_i}\ot \varphi_i,  \Delta(h^{\bf r})\Delta(b_\nu)\rangle.
$
Non-zero contributions can only arise from terms 
$\Delta(h^{\bf r})(b_{\nu; i}^{R, +} \ot e_i k_{\nu -\Upsilon_i} )$ in $\Delta(h^{\bf r}b_\nu)$, 
where $b_{\nu; i}^{R, +}$ is as defined by \eqref{eq:deriv-i} and 
\[
k_{\nu -\Upsilon_i} = \prod_{j} k_j^{\nu_j-\delta_{i j}}, \quad \nu=\sum_{i\in I} \nu_i \Upsilon_i.
\]  
Thus  $\langle \beta_{\mu +\Upsilon_i} e_i,  h^{\bf r} b_\nu\rangle$ is equal to  $q^{-(\Upsilon_i, \nu-\Upsilon_i)}\langle \beta_{\mu +\Upsilon_i}, h^{\bf r} b_{\nu, i}^+\rangle$. Now an induction on $|\mu|$ proves \eqref{eq:beta-b-2}. 
\end{proof}

\begin{lemma}\label{lem:co-dU}
The subalgebra $\wt\U^+_\hbar(A;\Xi)'$ has Hopf $(\Gamma, \omega)$-algebra structure, with 
\[
\baln
\text{co-multiplication} &  & & \Delta^0: \wt\U^+_\hbar(A;\Xi)'\lra \wt\U^+_\hbar(A;\Xi)'\wh\ot \wt\U^+_\hbar(A;\Xi)', \\
& & & \Delta^0(\rho_i)=\rho_i\ot 1 + 1 \ot \rho_i, \quad \Delta^0(\varphi_i)=\varphi_i \ot \wt\nu_i+ 1\ot \varphi_i, \\
\text{co-unit} & & & \epsilon^0: \wt\U^+_\hbar(A;\Xi)'\lra\C[[\hbar]], \\
& & & \epsilon^0(\rho_i)=0,  \quad \epsilon^0(\varphi_i)=0, \\
\text{antipode} & & & S^0: \wt\U^+_\hbar(A;\Xi)'\lra \wt\U^+_\hbar(A;\Xi)', \\
& & & S^0(\rho_i)=-\rho_i,  \quad 
S^0(\varphi_i)=- \varphi_i \wt\nu_i^{-1},   
\ealn
\]
where 
$
\wt{\nu}_i= \exp\left(-\frac{\hbar d_i}{2}\wt{h}_i\right) 
$
with $\wt{h}_i=\sum_j A_{i j}\wt{H}_j$.
\end{lemma}
\begin{proof}
The map $\Delta^0$ is the restriction of $\mu^*: \U^+_\hbar(A;\Xi)^\vee\lra (\U^+_\hbar(A;\Xi)\wh\ot \U^+_\hbar(A;\Xi))^\vee$ to $\U^+_\hbar(A;\Xi)'$.
Since $h^{\bf m} b_\mu h^{\bf n} b_\nu\in \U_0(\U_+)_{\mu+\nu}$, we conclude that 
\[
\baln
\langle \Delta^0(\varphi_i),  h^{\bf m} b_\mu \ot h^{\bf n}b_\nu \rangle =0, \quad \text{if $\mu+\nu\not\in\Pi$}, \\
\langle \Delta^0(\rho_i),  h^{\bf m} b_\mu \ot h^{\bf n}b_\nu \rangle =0, \quad \text{if $\mu+\nu\ne 0$}.
\ealn
\]

If $\mu+\nu\in \Pi$, then $h^{\bf m} b_\mu \ot h^{\bf n}b_\nu$ is either 
$h^{\bf m} e_s \ot h^{\bf n}$ or $h^{\bf m}\ot  h^{\bf n} e_s$ for some $s$ (up to scalar multiples). 
Now 
\[
\baln
\langle \Delta^0(\varphi_i),  h^{\bf m} e_s \ot h^{\bf n} \rangle 
&= \delta_{i s} \delta_{{\bf m}{\bf 0}} \prod_j(-A_{j i})^{n_j},\\
\langle \Delta^0(\varphi_i),  h^{\bf m}\ot  h^{\bf n} e_s \rangle 
&=\delta_{i s} \delta_{{\bf  m}+{\bf n}, {\bf 0}},   
\ealn
\]
which lead to
\[
\baln
\Delta^0(\varphi_i) &= \varphi_i\ot \sum_{{\bf n}}\prod_j\frac{(-A_{j i}\rho_j)^{n_j}}{n_j!} + 1\ot \varphi_i.
\ealn
\]
Recall that $\rho_i= \hbar \frac{d_i}{2}\wt{H}_i$. Thus $\sum_{{\bf n}}\prod_j\frac{(-A_{j i}\rho_j)^{n_j}}{n_j!}$ is equal to
\[
\baln
\sum_{{\bf n}}\prod_j\frac{(-\hbar(\Upsilon_i, \Upsilon_j)\wt{H}_j)^{n_j}}{n_j!}
&= \exp^{-\hbar\sum_j (\Upsilon_i, \Upsilon_j)\wt{H}_j}
= \exp^{-\frac{\hbar d_i}{2}\sum_j A_{i j}\wt{H}_j}=\wt\nu_i, 
\ealn
\]
where we have used $\frac{d_i}{2}A_{ij}=(\Upsilon_i,\Upsilon_j)$.
Hence
$
\Delta^0(\varphi_i) = \varphi_i\ot  \wt{\nu}_i+ 1\ot \varphi_i.
$

If $\mu+\nu=0$, then $h^{\bf m} b_\mu \ot h^{\bf n}b_\nu=h^{\bf m} \ot h^{\bf n}$.
We have 
\[
\langle \Delta^0(\rho_i),  h^{\bf m} \ot h^{\bf n} \rangle =\delta_{{\bf m}+{\bf n}, {\bf 1}_i},
\]
which leads to $\Delta^0(\rho_i)=\rho_i\ot 1 + 1\ot \rho_i$. 

It is easy to verify that $\epsilon^0$ and $S^0$ have the claimed properties, 
completing the proof of the lemma.
\end{proof}

Note that $\wt{h}_i=\sum_j A_{i j}\wt{H}_j$ satisfies
\[
\wt{h}_i \varphi_j -  \varphi_j \wt{h}_i = - A_{i j} \varphi_j, \quad \forall j.
\]

\begin{lemma}
Retain notation of Lemma \ref{lem:beta-b}.  
The following relation holds.
\beq\label{eq:pairing}
\left\langle \rho^{\bf m}\beta_{\mu},  h^{\bf n} b_\nu\right\rangle=\delta_{\bf m n} \left\langle\beta_{\mu},  b_\nu\right\rangle{\bf m}!.    
\eeq
\end{lemma}
\begin{proof}
 We have $\left\langle \rho^{\bf m}\beta_{\mu},  h^{\bf n} b_\nu\right\rangle
= \left\langle \Delta^0(\rho^{\bf m}\beta_{\mu}),  h^{\bf n}\ot  b_\nu\right\rangle$.
Using formulae in Lemma \ref{lem:beta-b} and the explicit form of $\Delta^0$, 
we can express the right hand side as 
\[
\baln
\sum_{\bf r} 
\begin{pmatrix}
{\bf m}\\ {\bf r}
\end{pmatrix} \left\langle 
\rho^{{\bf m}-{\bf r}}\ot \rho^{\bf r}\beta_{\mu},  h^{\bf n}\ot  b_\nu\right\rangle
&=\sum_{\bf r} \delta_{{\bf m}, {\bf n}+{\bf r}} {\bf n}!   \begin{pmatrix}
{\bf m}\\ {\bf r}
\end{pmatrix}   
\left\langle 
\rho^{\bf r}\beta_{\mu},  b_\nu\right\rangle.
\ealn
\]
It follows from \eqref{eq:pair-general} that $\left\langle 
\rho^{\bf r}\beta_{\mu},  b_\nu\right\rangle=0$ unless ${\bf r}=0$.  
This proves \eqref{eq:pairing}. 
\end{proof}

\begin{remark}\label{rmk:necessity}
The co-multiplication for 
$\wt\U^+_\hbar(A;\Xi)'$ involves  exponentials $\wt{\nu}_i$ of elements of $\sum_i\C\rho_i$, 
which are well defined only because $\rho_i=\hbar \frac{d_i}{2} \wt{H}_i$.  
Under the Hopf pairing,   $\big\langle \wt{H}^{\bf m},  h^{\bf n} \big\rangle\propto  \delta_{{\bf m} {\bf n}}\hbar^{-|{\bf m}|}$ as follows from $\left\langle \rho^{\bf m},  h^{\bf n} \right\rangle=\delta_{\bf m n}{\bf m}!$.
To conciliate both, we need to work over $\C((\hbar))$ from the start,  instead of working with $\C[[\hbar]]$ and then extensding scalars to $\C((\hbar))$. 
\end{remark}
\subsection{Quantised colour Borel subalgebras}
\subsubsection{Colour quantum Serre relations}\label{sect:U+U-}

Let $\CJ^+=\{x\in \wt\U^+_\hbar(A;\Xi)\mid \langle \wt\U^+_\hbar(A;\Xi)', x\rangle =\{0\} \}$.  
Clearly $\CJ^+$ is a two-sided ideal of $\wt\U^+_\hbar(A;\Xi)$, 
which is homogeneous since dual space pairing preserves the grading.  If $x\in \CJ^+$, we have 
$\langle \beta\ot \beta', \Delta(x)\rangle = \langle \beta\beta', x\rangle = 0$ for all 
$\beta, \beta'\in \wt\U^+_\hbar(A;\Xi)'$, thus $\Delta(x)\in \CJ^+\ot\wt\U^+_\hbar(A;\Xi)   + \wt\U^+_\hbar(A;\Xi)\ot \CJ^+$. This shows that $\CJ^+$ is a 
homogeneous Hopf ideal of $\wt\U^+_\hbar(A;\Xi)$. 

Now we have the Hopf $(\Gamma, \omega)$-algebra $\U^+_\hbar(A;\Xi):=\wt\U^+_\hbar(A;\Xi)/\CJ^+$. We will still denote the images of $e_i$ and $k_i$ in the quotient by the same symbols. 

An element $E_\mu\in (\wt\U_+)_\mu$, where $\mu>0$,  is called skew primitive  if its co-product is of the form $\Delta(E_\mu)=E_\mu\ot k_\mu + 1 \ot E_\mu$, where  $k_\mu$ is a product of $k_i$'s such that $k_\mu e_i k_\mu^{-1}= q^{(\Upsilon_i, \mu)} e_i$ for all $i\in I$. The $e_i$'s are skew primitive. 
Let $S^+_{i j}$, for $i\ne j$, be the elements in $\wt\U_+$
defined by the formula \eqref{eq:def-S+} but interpreted in the present context.
The proof of Lemma \ref{lem:co-prods} is still valid in the current context, showing that
$S^+_{i j}$ are skew primitive.

The following fact is easy to prove.
\begin{lemma} \label{lem:skew}
Skew primitive elements of $(\wt\U_+)_\mu$ belong to $\CJ^+$ if 
$0<\mu\not\in \Pi$. In particular, the elements $S^+_{i j}$,  for all $i\ne j$, 
belong to $\CJ^+$. 
\end{lemma}
\begin{proof}
Since $0<\mu\not\in \Pi$, we have $\left\langle 1,  E_\mu\right\rangle=0$, and
$\left\langle \varphi_i,  E_\mu\right\rangle=0$ for all $i$ by 
\eqref{eq:beta-b-2}.  Now for any ${\bf j}=(j_1, j_2, \dots, j_n)\in I^n$, we write ${\bf j}'=(j_1, j_2, \dots, j_{n-1})$. Then $\phi({\bf j}) =\phi({\bf j}')\varphi_{j_n}$.    Hence 
\[
\baln
\left\langle \rho^{\bf m}\phi({\bf j}),  E_\mu\right\rangle&=\left\langle \rho^{\bf m}\ot \phi(\bf j),  E_\mu\ot k_\mu + 1 \ot E_\mu\right\rangle= \delta_{0\bf m}\left\langle \phi({\bf j}),  E_\mu\right\rangle, \\
\left\langle \phi({\bf j}),  E_\mu\right\rangle
&=\left\langle \phi({\bf j}')\ot \varphi_{j_n},  E_\mu\ot k_\mu + 1 \ot E_\mu\right\rangle
= 0.
\ealn
\]
This proves the first statement of the lemma.

Since the elements $S^+_{i j}$ ($i\ne j$) are skew primitive with weights not in $\Pi$, 
they must belong to $\CJ^+$. 
\end{proof}

The following result is an immediate consequence of the lemma.
\begin{lemma}\label{lem:Serre-U}
The Hopf $(\Gamma, \omega)$-algebra $\U^+_\hbar(A;\Xi)$ satisfies the  relations 
\[ 
S^+_{i j}+\CJ^+ =0, \quad i\ne j.
\]
Furthermore, $e_i+\CJ^+$ ($i\in I$) are the only skew primitive elements up to scalar multiples.
\end{lemma}

Define $\CJ^-=\{f\in \wt\U^+_\hbar(A;\Xi)'\mid \langle f, \wt\U^+_\hbar(A;\Xi)\rangle =\{0\} \}$.
We can similarly show, as in the case of $\CJ^+$, that $\CJ^-$  is a homogeneous Hopf ideal of $\wt\U^+_\hbar(A;\Xi)'$. 
Let $S^-_{i j}$, for $i\ne j$, be the elements in $\wt\U^+_\hbar(A;\Xi)'$
defined by the formula \eqref{eq:def-S-} with $f_\ell$'s replaced by $\varphi_\ell$. 
Then $S^-_{i j}$ are skew primitive.

We have  the following result, whose proof is similar to that of Lemma \ref{lem:Serre-U}. 
\begin{lemma}\label{lem:Serre-dU}
The quotient $\U^+_\hbar(A;\Xi)':= \wt\U^+_\hbar(A;\Xi)'/\CJ^-$ is a Hopf $(\Gamma, \omega)$-algebra, which satisfies the  relations 
\[ 
S^-_{i j}+\CJ^- =0, \quad i\ne j.
\]
Furthermore, $f_i+\CJ^-$ ($i\in I$) are the only skew primitive elements up to scalar multiples.
\end{lemma}

The following important fact follows from the definitions of $\U^+_\hbar(A;\Xi)$ and $\U^+_\hbar(A;\Xi)'$.
\begin{theorem}\label{thm:non-degenerate}
The dual space pairing $\wt\U^+_\hbar(A;\Xi)^\vee\ot \wt\U^+_\hbar(A;\Xi)\lra \K$  
descends to a non-degenerate pairing $\langle \  , \  \rangle: \U^+_\hbar(A;\Xi)'\otimes \U^+_\hbar(A;\Xi)\lra \K $ of Hopf $(\Gamma, \omega)$-algebras.
\end{theorem}
\begin{proof} Assume that for an element $a+\CJ^-\in \U^-_\hbar(A;\Xi)$, 
\[
\langle a+\CJ^-, x+ \CJ^+ \rangle = 0, \quad \forall x+\CJ^+\in \U^+_\hbar(A;\Xi).
\]
The left hand side is equal to $\langle a, x\rangle$, which is the original Hopf pairing between $\wt\U^+_\hbar(A;\Xi)^*$ and $ \wt\U^+_\hbar(A;\Xi)$. This implies $a\in \CJ^-$, 
and hence $a+\CJ^-=0$. The argument goes through verbatim for the $\U^+_\hbar(A;\Xi)$ side, 
proving the theorem. 
\end{proof}

\subsubsection{Relationship to Nichols algebras}\label{sect:complete}

While our primary development focuses on the explicit presentation of the colour quantum group as a whole in terms of generators and relations, the independent Borel subalgebras isolated in this section can be directly contextualised within Nichols algebra approach \cite{AAB, AO17, Hec09, Ros98, Lus10} (see \cite{HS20} for a review) through a bosonisation process. 

Let $\U_{q; \Xi}^{++}$ be the subalgebra of $\U^+_\hbar(A;\Xi)$ generated by the elements $e_i+\CJ^+$ for all $i\in I$. Also denote by $\U_{q; \Xi}^+(A)$ the subalgebra  of $\U^+_\hbar(A;\Xi)$ generated by $e_i+\CJ^+$ and $k_i^{\pm 1}+\CJ^+$ ($i\in I$), which contains $\U_{q; \Xi}^{++}$. 
Now $\U_{q; \Xi}^{++}$ is directly related to the Nichols algebra $\mathfrak{B}_Q(I)$ over a $|I|$-dimensional diagonally braided vector space with braiding matrix $Q=(q_{i j}) = (\omega(\xi_i, \xi_j) q^{A_{i j}})$. 
Denote by $\CB\mathfrak{B}_Q(I)$ the bosonisation of $\mathfrak{B}_Q(I)$.
By inspecting the definitions of $\mathfrak{B}_Q(I)$  and $\U^+_\hbar(A;\Xi)$, 
one can see that  $\CB\mathfrak{B}_Q(I)$ is isomorphic to 
$\U_{q; \Xi}^+(A)$.  
The following fact is extracted from the Nichols algebra reformulation
(see e.g., \cite{AAB, AO17, Hec09, HS20}) of results in \cite{Lus10, Ros98}.

\begin{lemma}\label{lem:Nichols}
For $A$ being a Cartan matrix of finite or affine type,
$\U_{q; \Xi}^{++}\simeq \mathfrak{B}_Q(I)$ is defined by the colour quantum Serre relations 
$S^+_{i j}+\CJ^+=0$, for all $i\ne j$. 
\end{lemma}

We point out that this is derived from  
the corresponding result in the usual non-graded case, which was first proved by Ross \cite{Ros98} 
for finite type Cartan matrices, and by Lusztig \cite{Lus10} for affine type
Cartan matrices. Work of Andruskiewitsch and co-workers 
(see e.g., \cite{AO17, Hec09} and \cite{AAB} in particular) 
enables one to treat the grading, thus to deduce the lemma from \cite{Lus10, Ros98}. 
[The usual proof for Serre presentations of finite dimensional simple Lie algebras 
(see, e.g., \cite[\S18]{H}) was generalised to the Nichols algebra setting using the notion 
of Weyl groupoid (see e.g., \cite{AO17, Hec09}).]

An analogous statement is true for the dual Hopf algebra $\U_{q; \Xi}^+(A)'$.

\begin{notation}\label{notation}
For simplicity, we denote the elements  $e_i+\CJ^+$ and $h_i+\CJ^+$ of $\U^+_\hbar(A;\Xi)$, and the elements $f_i+\CJ^-$ and $\wt{h_i}+\CJ^-$ of $\U^+_\hbar(A;\Xi)'$, 
by $e_i, h_i , f_i, \wt{h}_i$ respectively. 
\end{notation}


\subsection{Colour quantum groups as quantum doubles}
In this section, the Hopf $(\Gamma, \omega)$-algebra
$\U^+_\hbar(A;\Xi)'$ is taken to be
\beq\label{eq:opp-d}
(\U^+_\hbar(A;\Xi)', \Delta_0=(\Delta^0)', \varepsilon_0=\varepsilon^0, S_0=(S^0)^{-1}). 
\eeq
We can read off the structure maps from Lemma \ref{lem:co-dU}. In particular,   
\beq
\Delta_0(\wt{h}_i) = \wt{h}_i\ot 1 + 1\ot \wt{h}_i, \quad 
\Delta_0(f_i) =f_i\ot 1 +  \wt{\nu}_i\ot f_i.
\eeq

\begin{notation}
Hereafter we denote  the Hopf \((\Gamma,\omega)\)-algebra 
$\U^+_\hbar(A;\Xi)'$ with the opposite Hopf structure \eqref{eq:opp-d} 
by \(\U^-_\hbar(A;\Xi)\).
\end{notation}

\subsubsection{Actions of dual Hopf algebras}
Let us make some preparations for studying the quantum double of $\U^+_\hbar(A;\Xi)$. 
It follows from Section \ref{sect:technical} that there are the following left and right  $\U^+_\hbar(A;\Xi)$-actions on $\U^-_\hbar(A;\Xi)$, 
\[
\baln
\rhd: &\  \U^+_\hbar(A;\Xi)\ot \U^-_\hbar(A;\Xi)\lra \U^-_\hbar(A;\Xi), \\
\lhd: &\  \U^-_\hbar(A;\Xi)\ot \U^+_\hbar(A;\Xi) \lra \U^-_\hbar(A;\Xi), \\
\ealn
\]
and left and right  $\U^-_\hbar(A;\Xi)$-actions on $\U^+_\hbar(A;\Xi)$, 
\[
\baln
\brhd: &\  \U^-_\hbar(A;\Xi)\ot \U^+_\hbar(A;\Xi)\lra \U^+_\hbar(A;\Xi), \\
\blhd: &\  \U^+_\hbar(A;\Xi)\ot \U^-_\hbar(A;\Xi) \lra \U^+_\hbar(A;\Xi), \\
\ealn
\]
which are respectively defined, 
 for any $x\in \U^+_\hbar(A;\Xi)$ and $f\in \U^-_\hbar(A;\Xi)$, by
\[
\baln
&x\rhd f=\sum_{(f)} \omega(d x, d f_{(1)}) \langle f_{(1)}, x\rangle f_{(2)}, \\ 
&f\lhd x=\sum_{(f)} f_{(1)} \langle f_{(2)}, x\rangle,\\
&f\brhd x=\sum_{(x)} \langle S_0(f), x_{(1)}\rangle x_{(2)}, \\
& x\blhd f=\sum_{(x)} \omega(d x_{(2)}, d f) x_{(1)}\langle  S_0(f),  x_{(2)}\rangle. 
\ealn
\]
Clearly 
\beq \label{eq:U-act-dU-0}
1\brhd x=x,  \quad  1\lhd x=\epsilon(x), \quad 
f\brhd 1=\epsilon_0(f), \quad  f\lhd 1=f. 
\eeq

Recall that for any polynomial $E$ of the elements $e_i$ which is of definite weight and is homogeneous in the $\Gamma$-grading, we defined  the elements 
$\hat{E}_{; i}^{R, \pm}$ and $\hat{E}_{; i}^{L, \pm}$ over $\C(q)$ by \eqref{eq:deriv-i}. 
The definition can be adaptedto the present setting without change.  
We can show that for any $E\in (\U_+)_{\mu}$ with $\mu>0$,  
\[
\baln
\varphi_i\brhd E = -k_i \hat{E}_{; i}^{L, -}, \quad
E\blhd \varphi_i
=- \omega(\xi_i, \xi_i) \hat{E}_{; i}^{R, +},
\ealn
\]
and hence we obtain the relation 
\beq\label{eq:Delta(E)}
\quad \Delta(E)= E\ot k_\mu  +  \sum _i \hat{E}_{;i}^{R, +} \ot e_i  k_{\mu-\Upsilon_i}
+\dots  +  \sum _i e_i\ot k_i \hat{E}_{;i}^{L, -} + 1\ot E,
\eeq
which was already proved in Lemma \ref{lem:primit}. 


The lemma below can be verified by direct calculations. 
\begin{lemma} \label{lem:U-act-dU}
The following relations hold.
\[
\baln
&{\rho_i}\brhd {h_j} = -1, \quad  \rho_i \lhd {h_j} =1, \\
& {\wt\nu_i}\brhd {h_j} = h_j + A_{j i}, 
\quad  \rho_i \lhd k_j=\rho_i+ \delta_{i j}\frac{d_i}{2} \hbar, \\
&{\varphi_i}\brhd {h_j}=0, \quad  {\varphi_i}\lhd {h_j} =0, \quad  {\varphi_i}\lhd k_j=\varphi_i, \\
&{\rho_i}\brhd {e_j}=0, \quad  \rho_i \lhd {e_j} = 0,  \quad  \wt{\nu_i}\brhd {e_j} =e_j, \\
&{\varphi_i}\brhd {e_j}= - \delta_{i j} k_j, 
\quad {\varphi_i}\lhd {e_j}=\delta_{i j} \wt\nu_i.
\ealn
\]
\end{lemma}

There is also a left action $R: \U^-_\hbar(A;\Xi)\ot \U^+_\hbar(A;\Xi)\lra \U^+_\hbar(A;\Xi)$ defined by 
\beq
f\ot x\mapsto R_f(x):=\sum_{(x)} \omega(d f, d x_{(1)}) x_{(1)} \langle f, x_{(2)}\rangle,
\eeq
which is the analogue of right translations in the context of Lie groups. 

\begin{lemma} 
The operators $R_{\varphi_j}$ are $\Gamma$-graded skew derivations on $\U^+_\hbar(A;\Xi)$, i.e.,  
for all $x,  y\in\U^+_\hbar(A;\Xi)$, 
\[
R_{\varphi_j}(x y)= \omega(d x, \xi_j) x R_{\varphi_j}(y) 
+   R_{\varphi_j}(x) Ad_{k_j^{-1}}(y).
\]
\end{lemma}
\begin{proof} The formula can be verified by direct calculations. We may assume that 
$y\in U_0(\U_+)_\mu$. Then  
\[
\baln
R_{\varphi_j}(x y)&= \sum_{(x), (y)} \omega(d x_{(1)} + d y_{(1)}, \xi_j)  \omega(d x_{(2)}, d y_{(1)})x_{(1)} y_{(1)} \langle \varphi_j, x_{(2)} y_{(2)}\rangle\\
&= \sum_{(x), (y)} \omega(d x_{(1)} + d y_{(1)}, \xi_j)  \omega(d x_{(2)}, d y_{(1)})\\
&\times x_{(1)} y_{(1)} \langle \varphi_j\ot \wt\nu_j + 1\ot \varphi_j, x_{(2)}\ot  y_{(2)}\rangle\\
&= \sum_{(y)} \omega(d x + d y_{(1)}, \xi_j) x y_{(1)} \langle \varphi_j,  y_{(2)}\rangle\\
&+ \sum_{(x)} \omega(d x_{(1)}, \xi_j)  x_{(1)} y 
\langle \varphi_j, x_{(2)}\rangle \langle \wt\nu_j, k_\mu\rangle\\
&= \omega(d x, \xi_j) x R_{\varphi_j}(y) 
+   R_{\varphi_j}(x) y  \langle \wt\nu_j, k_\mu\rangle\\
&= \omega(d x, \xi_j) x R_{\varphi_j}(y) 
+   R_{\varphi_j}(x) Ad_{k_j^{-1}}(y).
\ealn
\]
This proves the lemma. 
\end{proof}

One can show that for any $E\in (\U_+)_\mu$ with $\mu>0$, 
\[
\baln
R_{\varphi_j} (E)&=\sum_{(E_\mu)} \omega( d E_{(1)}, \xi_j) E_{(1)} \langle \varphi_j, E_{(2)}\rangle\\
&= \omega( d {E}-\xi_j, \xi_j) \wh{E}_{; j}^{R, +} q^{-(\Upsilon_j, \mu-\Upsilon_j)}\\
&= \omega( d E-\xi_j, \xi_j) Ad_{k_j^{-1}}(\wh{E}_{; j}^{R, +}), 
\ealn
\]
that is, 
\beq
\wh{E}_{; j}^{R, +}
=\omega(\xi_j, d E) Ad_{k_j}( R_{\varphi_j} (E)).
\eeq

\subsubsection{Quantum double construction}

Theorem \ref{thm:non-degenerate} gives a non-degenerate Hopf pairing between  
$\U^+_\hbar(A;\Xi)$ and $\U^-_\hbar(A;\Xi)$. It enables us to construct 
the quantum double (see Defintion \ref{def:D(H)}) of $\U^+_\hbar(A;\Xi)$,
\beq
\SD_\hbar(A; \Xi)=\U^+_\hbar(A;\Xi)\ot_{\K } \U^-_\hbar(A;\Xi),
\eeq 
 where $ \U^-_\hbar(A;\Xi)$ has the opposite Hopf algebra structure as described 
 by \eqref{eq:opp-d}. 
We denote by $\SR$ the universal $R$-matrix of the topological quasi-triangular Hopf $(\Gamma, \omega)$-algebra $\SD_\hbar(A; \Xi)$. 

The Hopf algebras $\U^+_\hbar(A;\Xi)$ 
and $\U^-_\hbar(A;\Xi)$
are naturally embedded in  $\SD_\hbar(A; \Xi)$ 
with images $\U^+_\hbar(A;\Xi)\ot 1$ and 
and $1\ot \U^-_\hbar(A;\Xi)$ respectively. 
For any $y\in  \U^+_\hbar(A;\Xi)$ and $f\in \U^-_\hbar(A;\Xi)$, 
we shall simply write their images $y\ot 1$ and $1\ot f$ as $y$ and $f$ respectively. 
Then $y f \in \SD_\hbar(A; \Xi)$ is understood as $y\ot f$, and 
$f y\in \SD_\hbar(A; \Xi)$ as $(1\ot f)(y \ot 1)$. 
Recall from Section \ref{sect:double-gen} that in $\SD_\hbar(A; \Xi)$, 
\[
f y
=\sum_{(f), (y)} \omega(d f_{(2)}, d y_{(1)})  (f_{(1)} \brhd y_{(1)})   (f_{(2)} \lhd y_{(2)}).
\]

We now work out the relations between the generators of $\U^+_\hbar(A;\Xi)$ 
and $\U^-_\hbar(A;\Xi)$ in $\SD_\hbar(A; \Xi)$. 
By using equation  \ref{eq:U-act-dU-0},  one can show that
\[
\baln
\rho_i h_j &= \sum \left( ({\rho_i}\brhd {h_j}_{(1)}) (1\lhd {h_j}_{(2)})
+ (1\brhd {h_j}_{(1)}) (\rho_i \lhd {h_j}_{(2)})\right)\\
&= {\rho_i}\brhd {h_j} + \sum {h_j}_{(1)} (\rho_i \lhd {h_j}_{(2)})\\
&= {\rho_i}\brhd {h_j} +  {h_j} \rho_i + \rho_i \lhd {h_j};
\ealn
\]
\[
\baln
\rho_i e_j &= \sum \left( ({\rho_i}\brhd {e_j}_{(1)}) (1\lhd {e_j}_{(2)})
+ (1\brhd {e_j}_{(1)}) (\rho_i \lhd {e_j}_{(2)})\right)\\
&= {\rho_i}\brhd {e_j}+ \sum{e_j}_{(1)}(\rho_i \lhd {e_j}_{(2)})\\
&= {\rho_i}\brhd {e_j}+ \rho_i \lhd {e_j} + {e_j} (\rho_i \lhd k_j);
\ealn
\]
\[
\baln
\varphi_i h_j &= \sum \left( ({\varphi_i}_{(1)}\brhd {h_j}) ({\varphi_i}_{(2)}\lhd 1)
+({\varphi_i}_{(1)}\brhd 1) ({\varphi_i}_{(2)}\lhd {h_j})  \right)\\
&= \sum ({\varphi_i}_{(1)}\brhd {h_j}) {\varphi_i}_{(2)}
+{\varphi_i}\lhd {h_j}  \\
&= {\varphi_i}\brhd {h_j} + ({\wt\nu_i}\brhd {h_j}) \varphi_i
+{\varphi_i}\lhd {h_j}  \\
&=   ({\wt\nu_i}\brhd {h_j}) \varphi_i, 
\ealn
\]
\[
\baln
\varphi_i e_j &=\sum \left( ({\varphi_i}_{(1)}\brhd {e_j}) ({\varphi_i}_{(2)}\lhd k_i)
+\omega(\xi_j, \xi_i) ({\varphi_i}_{(1)}\brhd 1) ({\varphi_i}_{(2)}\lhd {e_j})  \right)\\
&= ({\varphi_i}\brhd {e_j}) (1\lhd k_i) + \omega(\xi_j, \xi_i) (\wt{\nu_i}\brhd {e_j}) ({\varphi_i}\lhd k_i)
+\omega(\xi_j, \xi_i) {\varphi_i}\lhd {e_j}  \\
&= {\varphi_i}\brhd {e_j} 
+ \omega(\xi_j, \xi_i) (\wt{\nu_i}\brhd {e_j}) ({\varphi_i}\lhd k_i)
+\omega(\xi_j, \xi_i) {\varphi_i}\lhd {e_j},  
\ealn
\]
that is, 
\begin{align}
\rho_i h_j &= {\rho_i}\brhd {h_j} +  {h_j} \rho_i + \rho_i \lhd {h_j};\\
\rho_i e_j &= {\rho_i}\brhd {e_j}+ \rho_i \lhd {e_j} + {e_j} (\rho_i \lhd k_j);\\
\varphi_i h_j &=   ({\wt\nu_i}\brhd {h_j}) \varphi_i,  \label{eq:phi-h}\\
\varphi_i e_j &= {\varphi_i}\brhd {e_j} 
+ \omega(\xi_j, \xi_i) (\wt{\nu_i}\brhd {e_j}) ({\varphi_i}\lhd k_i)
+\omega(\xi_j, \xi_i) {\varphi_i}\lhd {e_j}.  
\end{align}
Applying Lemma \ref{lem:U-act-dU}, we obtain 
\begin{align}
&\rho_i h_j = h_j \rho_i; \label{eq:h-h} \\
&\rho_i e_j = e_j\left(\rho_i + \delta_{i j}\frac{d_i}{2} \hbar \right), \quad
\varphi_i h_j = h_j \varphi_i   + A_{j i}\varphi_i,  \label{eq:h-e-h-f} \\
&\varphi_i e_j =\omega(\xi_j, \xi_i)  e_j \varphi_i +\delta_{i j} (\wt{\nu_i} - k_i ).  \label{eq:e-f}
\end{align}
Consider, for example, the second relation of \eqref{eq:h-e-h-f}. 
We have ${\wt\nu_i}\brhd {h_j} = h_j + A_{j i}$ from Lemma \ref{lem:U-act-dU}. 
Using this in \eqref{eq:phi-h}, we obtain the desired relation:
\[
\varphi_i h_j =  ({\wt\nu_i}\brhd {h_j}) \varphi_i =   h_j \varphi_i + A_{j i}\varphi_i, 
\] 
It is equivalent to $h_j \varphi_i - \varphi_i h_j=-  A_{j i}\varphi_i$. 

The equations \eqref{eq:h-h} - \eqref{eq:e-f} can be re-written as follows.  
\begin{lemma}\label{lem:U-dU}
The following relations hold in $\SD_\hbar(A; \Xi)$.
\beq
&&\wt{h}_i h_j - h_j \wt{h}_i = 0, \\ 
&&\wt{h}_i e_j - e_j \wt{h}_i = A_{i j} e_j, \quad
 h_i f_j  - f_j h_i =  - A_{i j}f_j, \\ 
&& e_i f_j - \omega(\xi_j, \xi_i)   f_j e_i =\delta_{i j} \frac{k_i -\wt{\nu_i}}{q_i - q_i^{-1}}.\label{eq:FE-comm}
\eeq
\end{lemma}

It is clear from the definitions of $\U^+_\hbar(A;\Xi)$ and $\U^-_\hbar(A;\Xi)$, 
Lemmas \ref{lem:dU-dU}, \ref{lem:co-dU},  and \ref{lem:U-dU}
that the elements $h_i-\wt{h_i}$, for all $i\in I$, are central and primitive. 
Thus they generate a Hopf ideal in $\SD_\hbar(A; \Xi)$, 
which will be denoted by $\langle h_i-\wt{h_i}\mid i\in I\rangle$. Let 
\[
\mathcal{p}: \SD_\hbar(A; \Xi)\lra \frac{\SD_\hbar(A; \Xi)}{\langle h_i-\wt{h_i}\mid i\in I\rangle}
\]
be the canonical surjection. 

We have the following result. 
\begin{theorem}\label{thm:QD}
There is an isomorphism   
\[
\iota: \U_\hbar(A; \Xi)\stackrel{\simeq}{\lra} \mathcal{p}(\SD_\hbar(A; \Xi))
\]
of topological quasi-triangular Hopf $(\Gamma, \omega)$-algebras. 
Consequently the universal $R$-matrix of $\U_\hbar(A; \Xi)$ is equal to $\iota^{-1}\ot \iota^{-1}(R)$, 
where $R=\mathcal{p}\ot \mathcal{p}(\SR)$. 
\end{theorem}
\begin{proof}
Recall that $\U_\hbar(A; \Xi)$ is defined by Definition \ref{def:special}, which is essentially the specialisation of $\U_{q;\Xi}(A)$ from $\C(q)$ to $\C((\hbar))$ via the ring homomorphism $q\mapsto \exp(\hbar)$. [However, $\U^+_\hbar(A;\Xi)=\wt\U^+_\hbar(A;\Xi)/\CJ^+$ and 
$\U^-_\hbar(A;\Xi)=\wt\U^+_\hbar(A;\Xi)'/\CJ^-$, which should not be confused with the quantum Borel subalgebras of $\U_\hbar(A; \Xi)$.]
Surjectivity of the map $\iota$ follows from Lemmas \ref{lem:Serre-U} and \ref{lem:Serre-dU}.
To prove injectivity, it suffices to verify that $\U^+_\hbar(A;\Xi)$ and and 
$\U^-_\hbar(A;\Xi)$ have precisely the defining presentations prescribed by 
the colour quantum Serre relations beside those inherited from $\wt\U^+_\hbar(A;\Xi)$ and $\wt\U^+_\hbar(A;\Xi)'$. More explicitly, it suffices to show that
\begin{enumerate}[label=(\roman*)]
\item 
the quantum Serre relations among the elements 
$e_i+\CJ^+$ are the complete set of defining relations for the subalgebra 
$\U^{++}_{q; \Xi}(A)$ of $\U^+_\hbar(A;\Xi)$, and
\item 
the quantum Serre relations among the elements 
$f_i+\CJ^-$ are the complete set of defining relations for  the subalgebra of 
$\U^-_\hbar(A;\Xi)$ generated by all 
$f_i+\CJ^-$. 
\end{enumerate}

Recall that our Cartan matrix $A$ is either of finite type or affine type, thus 
we can apply Lemma \ref{lem:Nichols} to conclude that (i) is indeed true. 
The analogous fact of Lemma \ref{lem:Nichols} for the dual Hopf algebra $\U^+_\hbar(A;\Xi)$
implies (ii). This establishes the  isomorphism  
$\U_\hbar(A; \Xi)\simeq \mathcal{p}(\SD_\hbar(A; \Xi))$.

The definition of quasi-triangularity involves only structure maps, thus
it is preserved by homomorphisms of Hopf colour algebras (in the algebraic and topological settings). Since $\mathcal{p}$ is a Hopf homomorphism,  the universal $R$-matrix of 
of $\mathcal{p}(\SD_\hbar(A; \Xi))$ is given by $R=\mathcal{p}\ot \mathcal{p}(\SR)$. 
Now $\iota$ is a Hopf isomorphism, and the second statement of the theorem follows. 
\end{proof}

\begin{remark}\label{rmk:AAB-2}
It will be interesting to find a clean and precise relationship between 
the Hopf $(\Gamma, \omega)$-algebras $\SD_\hbar(A; \Xi)$ and
Hopf algebras $D(\mathscr{E})$ over $\C(q)$
in \cite[\S3.2]{AAB}, which are quantum doubles of bosonised Nichols algebras.  
This was already highly nontrivial in the case of quantum supergroups  \cite{XZ, XZ25, Z92}. 
\end{remark}

\section{Conclusion}
In this paper we have developed a systematic framework for the construction of quantum groups associated with Lie and affine Lie colour algebras satisfying the Cartan--Weyl paradigm. 

We have identified important families for which a Cartan--Weyl description exists, and have exhibited the obstruction to such a description in the orthogonal and orthosymplectic cases in terms of the $2$-torsion in the grading group and the dimensions of the corresponding homogeneous components. 
For the resulting finite and untwisted affine families, the Kac--Moody construction \cite{Kv, Mr} 
extends naturally to the colour setting, with the algebras determined by reduced Cartan matrices together with grading data. 

This allows us to construct the quantised universal enveloping colour algebras $\U_{q,\Xi}(A)$ explicitly in terms of generators and relations, which recover the ordinary Drinfeld--Jimbo quantum groups when the grading is trivial. Moreover, the construction via Hopf pairings and quantum doubles endows these algebras with a quasi-triangular Hopf $(\Gamma,\omega)$-structure, with the colour quantum Serre relations arising canonically from the radicals of the relevant Hopf pairings.

These results open several directions for mathematical and physical applications. 
The quasi-triangular structure and the resulting universal $R$-matrices provide 
a natural starting point for the study of colour analogues of integrable models 
in statistical mechanics and for constructions in low-dimensional topology and knot theory. 
The preservation of the $\Gamma$-grading and the commutative factor $\omega$ 
also suggests applications to quantum systems in which the grading encodes generalised statistics, including parastatistical structures. 
At roots of unity, the braiding is naturally intertwined with the grading and commutative factor, 
and this appears to provide wide scope for new generalised statistics. 
At the same time, the present framework is naturally related to noncommutative geometry. 
The colour quantum groups may be regarded as quantum deformations of 
the noncommutative para-geometries \cite{Z24} 
and related graded geometries associated with Lie colour algebras.

A particularly interesting connection is with the theory of Nichols algebras. 
The positive colour quantum Borel algebra contains a braided part which, 
after bosonisation, is identified with a Nichols algebra associated with 
the diagonal braiding matrix
$
Q=(q_{ij})=\bigl(\omega(\xi_i,\xi_j)q^{A_{ij}}\bigr),
$
and the colour quantum Serre relations give the defining relations in the finite and affine cases considered here. 
This places the present construction within the broader theory of braided and 
pointed Hopf algebras and Weyl groupoids. 
Nevertheless, the colour formulation retains information that is not visible after bosonisation; 
the grading and the commutative factor remain intrinsic to the Hopf colour algebra. 
A precise comparison between the colour quantum groups, their bosonisations, 
and Nichols algebra constructions therefore is an interesting problem to study.

There are also fundamental structural questions beyond the scope of the Cartan--Weyl paradigm. 
In particular, most orthogonal and orthosymplectic Lie colour algebras do not admit the required Cartan--Weyl description, and hence their structure and graded representation theory cannot be approached directly by the methods developed here. 
It remains an important open problem to develop an alternative framework for these algebras and, in particular, to determine whether explicit quantum deformations like the ones obtained here can be constructed without relying on a root-space decomposition. 

More broadly, understanding the representation theory, invariant theory, 
and geometric structures of colour quantum groups, 
as well as their roles in knot theory, integrable systems and 
generalised statistics, should provide fertile ground for 
further developing the theory of this paper.

The results of this paper thus establish a bridge between Lie colour algebras, Cartan--Weyl and Kac--Moody theory, quantum groups, braided and Nichols algebra methods, and mathematical physics. An important feature of this bridge is that the colour data are retained rather than removed by cocycle twisting or bosonisation. This makes the framework suitable not only as a generalisation of the Drinfeld--Jimbo construction, but also as a setting in which genuinely colour-dependent mathematical and physical phenomena can be investigated.

\appendix
\section{Quasi-triangular Hopf $(\Gamma, \omega)$-algebras}\label{sect:double}
The theory of quasi-triangular Hopf algebras \cite{D2, D3} provides 
a theoretical framework for Baxter's $R$ matrices \cite{D1}. 
It was generalised to Hopf superalgebras in \cite{GZB93}.
Quantum groups \cite{D1, D2, J} and quantum supergroups 
\cite{BGZ, Y91, Y94, ZGB, Z98} are
the most important examples of quasi-triangular Hopf (super)algebras.

The theory also generalises to Hopf $(\Gamma, \omega)$-algebras naturally. 
Here we collect the basics of this graded theory
for an arbitrary grading group $\Gamma$ equipped with a commutative factor $\omega$, 
giving proofs of vital results including the quantum double construction. 
We refer to \cite{AAB} for related results on the quantum double 
from the point of view of Nichols algebras 
for bosonised Hopf algebras. 

We initially develop the framework in the algebraic setting for clarity, 
before extending the results to the topological setting.

Fix a commutative ring $\K$ with identity. Let $\Gamma$ be an additive abelian group with a commutative factor $\omega$.  
We do not impose any condition on $\omega$ in this section.

%
%
%

\medskip
\noindent{\bf Simplification of notation}. We will write the $\Gamma$-degree  of a homogeneous element $v$ in a vector space as $d v$ in this section, as the original notation $d(v)$ is too cumbersome in long handed computations. 

\subsection{Quasi-triangular Hopf $(\Gamma, \omega)$-algebras}\label{sect:Hopf-dual}
We present the definition of
quasi-triangular Hopf $(\Gamma, \omega)$-algebras in this section. 

\subsubsection{Hopf $(\Gamma, \omega)$-algebras}

Let $H$ be an associative $(\Gamma, \omega)$-algebra over $\K$, 
with multiplication $\mu: H\ot H\lra H$ and unit map $u: \K\lra H$. 
Let $1_H$ be the unit element of $H$, then $u: a\mapsto a 1_H$.
The $(\Gamma, \omega)$-algebra $H$ is  a 
$(\Gamma, \omega)$-bi-algebra if there exist $(\Gamma, \omega)$-algebra 
homomorphisms $\varepsilon: H\lra\C$ and $\Delta: H\lra H\ot H$ such that  
\beq
&&(\Delta\ot\id)\Delta= (\id\ot \Delta)\Delta: H\lra H\ot H\ot H, \label{eq:co-asso}\\
&&(\varepsilon\ot\id)\Delta =\id= (\id\ot \varepsilon)\Delta: H\lra H,  \label{eq:counit-Del}
\eeq
where the second condition involves the canonical identifications  
$\C\ot H\simeq H\simeq H\ot \C$.
Call $\varepsilon$ the co-unit and $\Delta$ the co-multiplication.  
Equation \eqref{eq:co-asso} is the co-associativity of the co-multiplication.  

We will adopt Sweedler's notation to write 
$\Delta(x)=\sum_{(x)}x_{(1)}\ot x_{(2)}$, 
$(\Delta\ot\id)\Delta(x)= \sum_{(x)}x_{(1)}\ot x_{(2)}\ot x_{(3)}$,  
and etc. for any $x\in H$. 
We denote by $\Delta'$ the opposite co-multiplication. For any $x\in H$, 
\[
\Delta'(x)=\sum_{(x)}\omega(d x_{(1)}, d x_{(2)}) x_{(2)}\ot x_{(1)}.
\] 
The following formula will be useful. 
\beq\label{eq:D2-H}
\baln
&(\id \ot \Delta)\Delta(y z) =(\Delta\ot\id)\Delta(y z)\\
&=\sum \omega(d y_{(2)}+d y_{(3)}, d z_{(1)}) \omega(d y_{(3)}, d z_{(2)}) \\
& \times  y_{(1)} z_{(1)} \ot y_{(2)} z_{(2)}\ot y_{(3)} z_{(3)},
\ealn
\eeq
where we note that $\omega(d y_{(2)}+d y_{(3)}, d z_{(1)})=\omega(d y -d y_{(1)}, d z_{(1)})$, since 
\(d(x)\) is additive on homogeneous elements (i.e., \(d(ab) = d(a) + d(b)\)).

A Hopf $(\Gamma, \omega)$-algebra \cite[\S2.4]{Z25} is a $(\Gamma, \omega)$-bi-algebra  
$(H, \mu, u, \Delta, \varepsilon)$ with a 
$(\Gamma, \omega)$-algebra anti-homomorphism 
$
S: H\lra H, 
$
called the antipode, such that 
\beq\label{eq:Del-S}
\mu(S\ot\Delta) = \mu(\Delta\ot S)= \varepsilon.
\eeq
The fact that $S$ is an $(\Gamma, \omega)$-algebra anti-homomorphism 
means that 
$S(x y)=\omega(d x, d y)S(y) S(x)$ for $x, y\in H$. 
Equation \eqref{eq:Del-S} can be expressed in Sweedler's notation as 
$\sum_{(x)}S(x_{(1)}) x_{(2)} = \sum_{(x)}x_{(1)}S(x_{(2)})= \varepsilon(x)$ for all $x\in A$.

If the antipode is bijective, 
$(H, \mu, u, \Delta', \epsilon, S^{-1})$ is a Hopf algebra, 
the opposite Hopf algebra of $(H, \mu, u, \Delta, \epsilon, S)$. 

\begin{remark}
For all naturally appearing Hopf (super)algebras, the antipodes are bijective. 
Here we also assume bijectivity of the antipodes of Hopf $(\Gamma, \omega)$-algebras. 
\end{remark}

A $\Gamma$-graded module for a Hopf $(\Gamma, \omega)$-algebra $H$ is a $\Gamma$-graded module for $H$ as an associative $(\Gamma, \omega)$-algebra. 
If $V$ and $W$ are $\Gamma$-graded $H$-modules,   
$V\ot W$ is a $\Gamma$-graded $H\ot H$-module with the action 
$(H\ot H)  \ot (V\ot W)\lra V\ot W$ defined by 
\[
(x\ot y)\cdot(v\ot w)= \omega(d y, d v) x\cdot v\ot y\cdot w, \quad x, y\in H, v\in V, w\in W.
\]
As the co-multiplication $\Delta: H\lra H\ot H$ is a $(\Gamma, \omega)$-algebra homomorphism, 
it induces an $H$-action $H  \ot (V\ot W)\lra V\ot W$ on $V\ot W$, which is given by 
\beq
x\ot(v\ot w)\mapsto \Delta(x)\cdot(v\ot w). \label{eq:act-tensor}
\eeq  

\subsubsection{Quasi-triangular Hopf $(\Gamma, \omega)$-algebras}

It is straightforward to devise a notion of quasi-triangular Hopf $(\Gamma, \omega)$-algebras
by generalising from the $\Z_2$-grading \cite{GZB93} to arbitrary grading. 
\begin{definition}
A Hopf $(\Gamma, \omega)$-algebra 
$(H, \mu, u, \Delta, \epsilon, S)$ with bijective antipode 
is called quasi-triangular if there exists 
an element $R$ of $H\ot H$ (or some completion of the tensor product in the setting of topological Hopf algebras), 
which is homogeneous of degree $0$ and is invertible, 
such that   
\beq
&R\Delta(x) = \Delta'(x) R, \label{eq:YB-1}\\
&(\Delta\ot \id)R = R_{13} R_{23}, \quad (\id\ot \Delta)R= R_{13} R_{12}, \label{eq:YB-2}\\
&(S\ot \id)R= (\id \ot S^{-1})R= R^{-1}, \label{eq:S-R}\\
&(\epsilon\ot \id)R= (\id \ot \epsilon)R= 1\ot 1, \label{eq:counit-R}
\eeq
where $R_{12}=R\ot 1$, $R_{23}=1\ot R$ and 
$R_{13}= (\id\ot \tau)R_{12}$.  
The element $R$ is called the universal $R$-matrix. 
\end{definition}

The following important theorem is a consequence of \eqref{eq:YB-1} and \eqref{eq:YB-2}. 
\begin{theorem}
The universal $R$-matrix satisfies the Yang-Baxter equation.
\beq
R_{12} R_{13} R_{23} = R_{23} R_{13} R_{12}. 
\eeq
\end{theorem}

\begin{remark}
The original aim of the theory of quantum groups and supergroups was to solve the Yang-Baxter equation \cite{BGZ, D1, J}.
This equation plays a fundamental role in the theory of Yang-Baxter type soluble models in statistical mechanics. 
\end{remark}

\subsection{Finite duals of Hopf $(\Gamma, \omega)$-algebras} 
We adapt results on finite duals of  Hopf algebras \cite{A} to the $\Gamma$-graded setting.
This section also contains a more detailed treatment of 
related material in \cite[\S2.4]{Z25},

\subsubsection{Finite duals of Hopf $(\Gamma, \omega)$-algebras} 
Let $(H, \mu, u, \Delta, \varepsilon, S)$ be  a Hopf $(\Gamma, \omega)$-algebra, 
whose opposite Hopf algebra is  $(H, \mu, u, \Delta', \epsilon, S^{-1})$.

Denote by $H^*=\Hom_\K(H, \K)$ and $(H\ot H)^*=\Hom_\K(H\ot H, \K)$ the $\Gamma$-graded dual $\K$-modules of $H$ and $H\ot H$ respectively.  Let 
\beq 
&\Delta^*: H^* \ot H^*\lra H^*, \quad \mu^*:  H^*\lra (H\ot H)^*,  \label{eq:maps*}\\
& u^*: H^*\lra \K, \quad \epsilon^*: \K \lra H^*, 
\quad  S^*: H^*\lra H^*
\eeq
be the adjoint maps of $\Delta$, $M$, $u$, $\epsilon$  and $S$ respectively, defined, for any $f, g\in H^*$, by
\beq
&&\langle\Delta^*(f\ot g), x \rangle  = \langle f\ot g, \Delta(x) \rangle, \quad \forall x\in H,  \label{eq:Del*}\\
&&\langle \mu^*(f), x\ot y\rangle =\langle f, x y\rangle,  \quad \forall x, y \in H \\
&&u^*(f) = \langle   f, 1_H \rangle,  \\
&&\epsilon^*(a) = a\epsilon, \quad \forall a\in \K, \\ 
&&\langle S^*(f), x\rangle = \langle f, S(x)\rangle, \quad \forall x\in H. \label{eq:S*}
\eeq
The maps are all homogeneous of degree $0$.
We denote $\epsilon^*$ by $u_{H^*}$. 

Now $(H^*, \Delta^*,  u_{H^*})$ is a $(\Gamma, \omega)$-algebra 
with multiplication $\Delta^*$ and unit map $u_{H^*}$, 
where the associativity of $\Delta^*$ follows from the co-associativity of $\Delta$. 
Now $H^*\ot H^*$ is a subalgebra of $(H\ot H)^*$, but  in general
${\rm Im}(\mu^*)\not\subset H^*\ot H^*$, thus $H^*$ is not a Hopf algebra. 

A co-finite homogeneous ideal $\CJ$ of $H$ (as a $(\Gamma, \omega)$-algebra) is a homogeneous $2$-sided ideal such that $H/\CJ$ is a finitely generated $\K$-module. Let 
\[
H^0 = \{f\in H^*\mid \ker(f)  \text{ contains a co-finite homogeneous ideal}\},
\]
and call it the finite dual of $H$.  Some comments are in order.

(1). 
If for any non-vanishing $x\in H$, there exists $f\in H^0$ such that $\langle f, x\rangle\ne 0$, 
we say that $H^0$ is dense in $H^*$. It is clear that this occurs 
if and only if for any non-zero $x\in H$,
there is a homogeneous co-finite ideal which does not contain $x$. 
In this case, $H$ is called a proper Hopf $(\Gamma, \omega)$-algebra.

(2). 
We call an $H$-module $\K$-finite if it is finitely generated over $\K$.
Given a $\K$-finite $H$-module $V$ generated
by the generators $v_1, \dots, v_r$   over $\K$, 
there are elements $  \rho_{i j} \in H^*$,
such that $x\cdot v_i=\sum_j   \rho_{j i}(x) v_j$. 
Call $  \rho_{j i}$ the matrix coefficients of $V$ 
relative to this generating set.  Clearly $  \rho_{i j}\in H^0$.

(3). 
If $f\in H^0$, let $\CJ_f$ be a co-finite homogeneous ideal contained in 
$\ker(f)$. Then $H/\CJ_f$ is a $\K$-finite $H$-module 
with the action $y\cdot(x+\CJ_f)= xy + \CJ_f$ for all $x, y\in H$. 
If $H/\CJ_f=0$, we have $f=0$. If $H/\CJ_f\ne 0$, let $t_i\in H$, for $i=1, 2, \dots, r$, 
be elements such that their images $t_i+\CJ_f$ generate $H/\CJ_f$ over $\K$. 
Then $1+\CJ_f = \sum_{i} a_i (t_i + \CJ_f)$ for some $a_i\in\K$, 
and hence $x+\CJ_f =   \sum_{i, j} a_i \rho_{j i} (x) (t_j + \CJ_f)$ for all $x\in H$, 
where $\rho_{i j}$ are the matrix coefficients of $H/\CJ_f$ with respect to the generating set $\{ t_i +\CJ_f\}$. 
Then $f(x) = \sum_{i, j}  a_i f(t_j) \rho_{j i}(x)$. This shows that 
$
f = \sum_{i, j}  a_i f(t_j) \rho_{j i}.
$
Therefore $f\in H^0$ if and only if it is a $\K$-linear combination of matrix coefficients of some 
$\K$-finite $H$-modules.

(4). If $\rho^{(1)}_{i j}$ and $\rho^{(2)}_{rs}$ are the matrix coefficients of two $\K$-finite  $H$-modules $V_1$ and $V_2$ respectively,  $\sum_{i, j; r, s} \K\rho^{(1)}_{i j}\rho^{(2)}_{rs}$ is the $\K$-span of the matrix coefficients of the tensor product $H$-module $V_1\ot V_2$, which is again $\K$-finite. Hence $H^0$ is a subalgebra of $H^*$. 

(5). 
As $H/\CJ_f$ is finitely generated over $\K$, and so is also $(H/\CJ_f)^*$. 
There is a canonical embedding $\iota_f: (H/\CJ_f)^*\hookrightarrow H^0$, $\bar{g}\mapsto \iota(\bar{g})$,  
such that $\langle \iota_f(\bar{g}), x\rangle = \langle \bar{g}, x+\CJ_f\rangle$ for all $x\in H$. 
Now it is easy to see that  $\mu^*(f)\in \iota_f((H/\CJ_f)^*)\ot \iota_f((H/\CJ_f)^*)\subset  H^0\ot H^0$.

The following result generalises a basic fact of usual Hopf algebras \cite{A} to 
Hopf $(\Gamma, \omega)$-algebra. 
\begin{lemma}
Let $(H, \mu, u, \Delta, \epsilon, S)$ be a Hopf $(\Gamma, \omega)$-algebra. 
Then its finite dual $H^0$ is a Hopf $(\Gamma, \omega)$-algebra with 
\[
\baln
\text{multiplication} &\quad&&   \mu^0 = \Delta^*|_{H^0}: H^0\ot H^0\lra H^0, \\ 
\text{unit map} &&&   u^0: \K \lra H^0, \quad a\mapsto  u_{H^*}(a),  \\ 
\text{co-multiplication} &&&   \Delta^0 = \mu^*|_{H^0}: H^0\lra    H^0\ot H^0, \\ 
\text{co-unit} &&&  \epsilon^0: H^0\lra \K, \quad f\mapsto f(1), \\
\text{antipode} &&&  S^0=S^*|_{H^0}: H^0\lra H^0.
\ealn
\]
\end{lemma}
\begin{proof}
The proof can be extracted from the discussions preceding the statement of the lemma. 
It is a straightforward generalisation,  to the present setting, of the proof of the corresponding result 
for usual Hopf algebras \cite{A}. We omit details. 
\end{proof}

Assume that the antipode of the Hopf $(\Gamma, \omega)$-algebra
$(H, \mu, u, \Delta, \epsilon, S)$ is bijective.  Then $S^0$ is bijective. 
The opposite of the finite dual Hopf $(\Gamma, \omega)$-algebra $H^0$ of $H$ has the structure maps
\beq
&\text{multiplication }  \mu_0=\mu^0, \quad
\text{unit map  }    u_0=u^0,  \label{eq:H0-1}\\ 
&\text{co-multiplication }  \Delta_0 = \tau_{H^0, H^0}\circ \Delta^0, \quad
\text{co-unit }  \epsilon_0=\epsilon^0, \label{eq:H0-2}\\
&\text{antipode  }   S_0=(S^0)^{-1}.  \label{eq:H0-3}
\eeq
Note in particular that 
\[
\baln
\langle \Delta_0(f), x\ot y\rangle &= \langle \Delta^0(f),\tau_{H, H}(x\ot y)\rangle 
=\omega(d x, d y)\langle f, y x\rangle, \\
\langle S_0(f), x\rangle &=\langle f,  S^{-1}(x)\rangle, \quad f\in H^0, x, y\in H. 
\ealn
\]
We denote this Hopf $(\Gamma, \omega)$-algebra by
$(H^0, \mu_0, u_0, \Delta_0, \epsilon_0, S_0)$.

Hereafter we will always refer to $(H^0, \mu_0, u_0, \Delta_0, \epsilon_0, S_0)$ 
whenever $H^0$ is considered as a Hopf $(\Gamma, \omega)$-algebra. 
We write $\Delta_0(f)=\sum_{(f)} f_{(1)} \ot f_{(2)}$ for any $f\in H^0$.

\subsubsection{Technical results on dual Hopf $(\Gamma, \omega)$-algebras}\label{sect:technical}

The dual Hopf $(\Gamma, \omega)$-algebras $H$ and $H^0$ act on each other. We have the  
following {$H$ actions on $H^0$,} 
\beq\label{eq:H-on-H0}
\baln
&\text{left action \ }   H\ot H^0\lra H^0, \quad x\ot f\mapsto x\rhd f, \\ 
&\text{right action \ }   H^0\ot H\lra H^0, \quad f\ot x \mapsto f\lhd x,
\ealn
\eeq
which are respectively defined,  for any $x\in H$ and $f\in H^0$, by
\[
\baln
&x\rhd f=\sum_{(f)} \omega(d x, d f_{(1)}) \langle f_{(1)}, x\rangle f_{(2)}, \\ 
&f\lhd x=\sum_{(f)} f_{(1)} \langle f_{(2)}, x\rangle; 
\ealn
\]
and the {$H^0$ actions on $H$,} 
\beq\label{eq:H0-on-H}
\baln
&\text{left action \ }    H^0\ot H\lra H, \quad f\ot x \mapsto f\brhd x,  \\
&\text{right action \ }   H\ot H^0\lra H, \quad x\ot f\mapsto x\blhd f, 
\ealn
\eeq
which are respectively defined,  for any $x\in H$ and $f\in H^0$, by
\[
\baln
&f\brhd x=\sum_{(x)} \langle S_0(f), x_{(1)}\rangle x_{(2)}, \\
& x\blhd f=\sum_{(x)} \omega(d x_{(2)}, d f) x_{(1)}\langle  S_0(f),  x_{(2)}\rangle. 
\ealn
\]
It is easy to show that these indeed define left or right actions.

The following fact can also be easily proven. 
\begin{lemma}
\begin{enumerate}
\item For any $x\in H$ and $f\in H^0$, 
\beq
\Delta_0(x\rhd f) &=& \sum_{(f)} x\rhd f_{(1)}\ot  f_{(2)}, \\
 \Delta_0( f\lhd x) &=& \sum_{(f)}  f_{(1)}\ot  f_{(2)}\lhd x, \\
\Delta(f\brhd x) &=& \sum_{(f)} f\brhd x_{(1)}\ot  x_{(2)}, \\
 \Delta( x\blhd f) &=& \sum_{(f)} x_{(1)}\ot   x_{(2)}\blhd f.
\eeq

\item
For any $x, y\in H$ and $f, g\in H^0$,
\beq
\langle x\rhd f, y\rangle&=	& \omega(d x, d f+ d y) \langle f, y S^{-1}(x)\rangle \\
\langle f\lhd x, y\rangle&=& \langle f,  x y\rangle \\
\langle f, g\brhd y\rangle
&=& \omega(d f, d g)  \langle S_0(g) f, y\rangle\\						
\langle f, y\blhd g\rangle
&=&\omega(d y, d g) \langle f S_0(g), y\rangle.
\eeq
\end{enumerate}
\end{lemma} 
\begin{proof}
The proof of the lemma is entirely straightforward. For example, 
the first relation of part (2) can be shown by the following calculations.
\[
\baln
\langle x\rhd f, y\rangle&=\sum \omega(d x, d f_{(1)}) \langle S_0(f_{(1)}), x\rangle  
										\langle f_{(2)}, y\rangle\\
										&=\sum \omega(d x, d f) 
										\langle f{(1)}\ot f{(2)}, S^{-1}(x)\ot y\rangle \\
										&= \omega(d x, d f+ d y) 
										\langle f, y S^{-1}(x)\rangle.
\ealn
\]
The other relations can be similarly proven. 
\end{proof}

\begin{lemma}
The following relations hold. 
\beq
\quad& \rhd(\id_H\ot\mu_0) = \mu_0 (\rhd\ot \rhd)(\id_H\ot\tau_{H, H^0}\ot\id_{H^0})(\Delta\ot \id_{H^0}\ot\id_{H^0}),  \\
\quad&  \lhd(\mu_0\ot \id_H) = \mu_0 (\lhd\ot \lhd)(\id_{H^0}\ot\tau_{H^0, H}\ot\id_{H})(\id_{H^0}\ot\id_{H^0}\ot \Delta),  \\
\quad&  \brhd(\id_{H^0}\ot\mu) = \mu(\brhd\ot \brhd)(\id_{H^0}\ot\tau_{H^0, H}\ot\id_{H})(\Delta_0\ot \id_{H}\ot\id_{H}),  \\
\quad& \blhd(\mu\ot \id_{H^0}) = \mu(\blhd\ot \blhd)(\id_{H}\ot\tau_{H, H^0}\ot\id_{H^0})(\id_{H}\ot\id_{H}\ot \Delta_0).
\eeq 
That is,  $H^0$ (resp. $H$) inherits the structure of a  $\Gamma$-graded module algebra 
under the left or right action of $H$ (resp. $H^0$).  
\end{lemma}
\begin{proof}
It is easy to show that the lemma is equivalent to the following relations 
for all $x, y\in H$ and $f, g\in H^0$.
\beq
x\rhd (f g) = \sum_{(x)} \omega(d x_{(2)}, d f) (x_{(1)}\rhd f) (x_{(2)}\rhd y),\\
(f g)\lhd x = \sum_{(x)} \omega(d g, d x_{(1)}) (f\lhd x_{(1)}) (g\lhd x_{(2)}),  \\
f\brhd (x y) = \sum_{(f)} \omega(d f_{(2)}, d x) (f_{(1)}\brhd x)(f_{(2)}\brhd  y),\\
(x y)\blhd f = \sum_{(f)} \omega(d y, d f_{(1)}) (x\blhd f_{(1)}) (y\blhd f_{(2)}). \label{eq:xy-r-f}
\eeq
We consider, for example, the last relation. We have 
\[
\baln
(x y)\blhd f&=\sum_{(x), ( y)} \omega(d x_{(2)}+d y_{(2)}, d f) \omega(d x_{(2)}, d y_{(1)}) \\
&\times x_{(1)} y_{(1)}\langle  S_0(f),  x_{(2)}y_{(2)} \rangle.
\ealn
\]
We ca re-write the right hand side as 
\[
\baln
&\sum_{(x), (y)} \omega(d x_{(2)}+d y_{(2)}, d f) \omega(d x_{(2)}, d y_{(1)}) \\
&\times x_{(1)} y_{(1)}\langle  \Delta^0(S_0(f)),  x_{(2)}\ot y_{(2)} \rangle\\
&=\sum_{(x), (y)} \omega(d x_{(2)}+d y_{(2)}), d f) \omega(d x_{(2)}, d y_{(1)}) \\
&\times x_{(1)} y_{(1)}\langle  S_0(f_{(1)})\ot S_0(f_{(2)}),  x_{(2)}\ot y_{(2)} \rangle\\
&=\sum_{(x), (y), (f)} \omega(d x_{(2)}+d y_{(2)}, d f) \omega(d x_{(2)}, d y_{(1)})\\ 
&\times \omega(d f_{(2)}, d x_{(2)}) 
x_{(1)} y_{(1)}\langle  S_0(f_{(1)}),  x_{(2)}\rangle \langle  S_0(f_{(2)}),   y_{(2)} \rangle.
\ealn
\]
Note that $d f_{(1)}+d x_{(2)}=0$ and $d f_{(2)}+d y_{(2)}=0$ on the right hand side. 
This enables us to manipulate the product of commutative factors as follows. 
\[
\baln
&\omega(d x_{(2)}+d y_{(2)}, d f) \omega(d x_{(2)}, d y_{(1)}) 
\omega(d f_{(2)}, d x_{(2)})\\ 
&\leadsto\omega(d f, d f) \omega(d x_{(2)}, d y_{(1)}) \omega(d x_{(2)}, d y_{(2)}) \\ 
&\leadsto\omega(d f, d f) \omega(d x_{(2)}, d y) \\ 
&\leadsto\omega(d f, d f) \omega(d y, d f_{(1)}). 
\ealn
\]
Therefore, 
\[
\baln
(x y)\blhd f
&=\sum_{(x), (y), (f)} \omega(d f, d f) \omega(d y, d f_{(1)}) 
x_{(1)} y_{(1)}\langle  S_0(f_{(1)}),  x_{(2)}\rangle \langle  S_0(f_{(2)}),   y_{(2)} \rangle.
\ealn
\]

We also have 
\[
\baln
&\sum_{(f)} \omega(d y,  d f_{(1)}) (x \blhd f_{(1)}) (y_{(1)}\blhd f_{(2)})\\
&=\sum_{(x), (y), (f)} \omega(d y +d x_{(2)},  d f_{(1)}) \omega(d y_{(2)}, d f_{(2)}) \\
&\times x_{(1)} y_{(1)}\langle  S_0(f_{(1)}),  x_{(2)}\rangle \langle  S_0(f_{(2)}),   y_{(2)} \rangle.
\ealn
\]
Similarly manipulating the product of commutative factors, we obtain
\[
\baln
&\omega(d y+d x_{(2)},  d f_{(1)}) \omega(d y_{(2)}, d f_{(2)}) \\
&\leadsto   \omega(d y,  d f_{(1)}) \omega(d x_{(2)},  d f_{(1)})
 \omega(d y_{(2)}, d f_{(2)})  \\
& \leadsto   \omega(d y,  d f_{(1)}) \omega(d f_{(1)},  d f_{(1)})
 \omega(d f_{(2)}, d f_{(2)})  \\
&\leadsto \omega(d y, d f_{(1)}) \omega(d f, d f) \\
&= \omega(d f, d f)\omega(d y, d f_{(1)}) . 
\ealn
\]
Hence 
\[
\baln
&\sum_{(f)} \omega(d y,  d f_{(1)}) (x \blhd f_{(1)}) (y_{(1)}\blhd f_{(2)})\\
&=\sum_{(x), (y), (f)}  \omega(d f, d f)\omega(d y, d f_{(1)}) 
 x_{(1)} y_{(1)}\langle  S_0(f_{(1)}),  x_{(2)}\rangle \langle  S_0(f_{(2)}),   y_{(2)} \rangle.
\ealn
\]
This proves \eqref{eq:xy-r-f}. The other relations can be shown similarly.
\end{proof}

\begin{lemma}\label{lem:lr-act}
The following relations hold. 
\beq
&
\rhd(\brhd\ot \id_{H^0})(\tau_{H, H^0}\ot\id_{H^0})(\id_H\ot\Delta_0)=I_{H^0}(\epsilon\ot\id_{H^0}), \\
&
\rhd(\brhd\ot \id_{H^0})(\id_{H^0}\ot \tau_{H^0, H})(\Delta_0\ot \id_H)=I_{H^0}(\id_{H^0}\ot \epsilon), \\
&
\blhd(\id_H\ot\lhd)(\id_H\ot\tau_{H, H^0})(\Delta\ot \id_{H^0})= I_{H}(\id_H\ot \epsilon_0), \\
&
\blhd(\id_H\ot\lhd)(\tau_{H^0, H}\ot \id_H)(\id_{H^0}\ot \Delta)= I_{H}(\epsilon_0\ot \id_H), 
\eeq
where $I_V$, for $V= H$ or $H^0$,  denotes the canonical isomorphisms $\K\ot V \simeq V \simeq  V\ot \K$. 
\end{lemma}
\begin{proof}
To illustrate the proof, we consider the second relation. 
For any $f\ot x\in H^0\ot H$, we have 
\beq\label{eq:lr-act-1}
&&\rhd(\brhd\ot \id_{H^0})(\id_{H^0}\ot \tau_{H^0, H})(\Delta_0\ot \id_H)(f\ot x) \\
&&=\sum_{(f)} \omega(d f_{(2)}, d x)  (f_{(1)}\brhd x)\rhd f_{(2)}.\nonumber 
\eeq
Using the definitions of the maps $\rhd$ and $\brhd$ to the right hand side, we obtain
\[
\baln
&\sum_{(f), (x)} \omega(d f_{(2)}+ d f_{(3)} , d x)  \omega(d x_{(2)}, d f_{(2)})  
\langle S_0(f_{(1)}), x_{(1)} \rangle 
\langle f_{(2)},  x_{(2)}\rangle f_{(3)}\\
&= \sum_{(f), (x)} \omega(d f_{(2)}+ d f_{(3)} , d x)  \omega(d x, d f_{(2)})
  \langle S_0(f_{(1)}) \ot f_{(2)}, x_{(1)}\ot x_{(2)} \rangle  f_{(3)}\\
&=\sum_{(f), (x)} \omega(d f_{(3)}, d x)
 \langle S_0(f_{(1)}) \ot f_{(2)}, x_{(1)}\ot x_{(2)} \rangle  f_{(3)}\\
&=\sum_{(f)} \omega(d f_{(3)} , d x)
 \langle S_0(f_{(1)})  f_{(2)}, x \rangle  f_{(3)}\\
&=\sum_{(f)} \omega(d f_{(2)} , d x)
 \langle \epsilon_0(f_{(1)}), x \rangle  f_{(2)}\\
&=\omega(d f , d x) \sum_{(f)} 
 \langle \epsilon_0(f_{(1)}), x \rangle  f_{(2)} = \epsilon(x) f. 
\ealn
\]
This proves 
\[
\rhd(\brhd\ot \id_{H^0})(\id_{H^0}\ot \tau_{H^0, H})(\Delta_0\ot \id_H)(f\ot x) 
=\epsilon(x) f.
\]
The proofs for the other relations are similar
\end{proof}

Now observe that both
$H^0\ot H$ and $H\ot H^0$  are  Hopf $(\Gamma, \omega)$-algebras with respective co-multiplications 
\beq
\wh\Delta= (\id_{H^0}\ot \tau_{H^0, H}\ot \id_H)(\Delta_0\ot \Delta), \\
\wh\Delta'= (\id_H\ot \tau_{H, H^0}\ot \id_{H^0})(\Delta\ot \Delta_0). 
\eeq
[{\bf Warning}: $\wh\Delta'$ is not the opposite co-multiplication of $\wh\Delta$.] 
They are co-associative, i.e., 
$(\id_{H^0\ot H}\ot \wh\Delta)\wh\Delta = (\wh\Delta\ot \id_{H^0\ot H})\wh\Delta$ and 
$(\id_{H\ot H^0}\ot \wh\Delta')\wh\Delta$ $ = (\wh\Delta'\ot \id_{H\ot H^0})\wh\Delta'$.  
We write $\wh\Delta^{(k)}$ and ${\wh\Delta'}{^{(k)}}$ for the $k$-th iterated co-multiplications. In particular,  
$\wh\Delta^{(2)}=(\id_{H^0\ot H} \ot \wh\Delta)\wh\Delta$ and 
${\wh\Delta'}{^{(2)}}=(\id_{H\ot H^0} \ot \wh\Delta')\wh\Delta'$. Note that 
\beq
 \wh\Delta^{(2)}= (\id_{H^0}\ot\wh\Delta' \ot \id_{H})\wh\Delta, \label{eq:Dh-Dh-1}\\
 {\wh\Delta'}{^{(2)}}= (\id_{H}\ot\wh\Delta \ot \id_{H^0})\wh\Delta'. \label{eq:Dh-Dh-2}
\eeq

The following result is an easy consequence of Lemma \ref{lem:lr-act}.
\begin{corollary} \label{cor:invert-D} The following relations hold. 
\beq
(\rhd\ot \id_H)(\brhd\ot\id_{H^0}\ot \id_H)\wh\Delta=\id_{H^0\ot H},  \label{cor-1}\\
(\blhd\ot\id_{H^0})(\id_H\ot \lhd\ot\id_{H^0})\wh\Delta'= \id_{H\ot H^0}, \label{cor-2}
\eeq
\end{corollary}

We also have the following result.
\begin{lemma}\label{lem:lr-inv}
The following relations hold. 
\beq
&&(\lhd\ot\brhd)\wh\Delta= \id_{H^0\ot H}, \label{eq:lr-inv}\\
&&(\brhd\ot \id_{H^0\ot H} \ot\lhd)(\id_{H^0}\ot\wh\Delta'\ot \id_H) 
=  \wh\Delta'(\brhd\ot \lhd). \label{eq:lr-D}
\eeq
\end{lemma}
\begin{proof}
For any $f\ot x\in H^0\ot H$, 
\[
\baln
(\lhd\ot\brhd)\wh\Delta(f\ot x)
&=\sum_{(f), (x)} \omega(d f_{(2)}, d x_{(1)}) f_{(1)}\lhd x_{(1)}\ot f_{(2)}\brhd x_{(2)}\\
&= \sum_{(f), (x)} \omega(d f_{(3)}, d x_{(1)})
		 f_{(1)}\langle f_{(2)},  x_{(1)}\rangle \ot \langle S_0(f_{(3)}),  x_{(2)}\rangle x_{(3)}\\
&= \sum_{(f), (x)}  f_{(1)}\langle f_{(2)}\ot  S_0(f_{(3)}),  x_{(1)}\ot  x_{(2)}\rangle x_{(3)}\\
&= \sum_{(f), (x)}  f_{(1)}\langle f_{(2)}  S_0(f_{(3)}),  x_{(1)}\rangle x_{(2)}
=f\ot x, 
\ealn
\]
proving the first formula. 

The second formula follows from an easy inspection. 
\end{proof}

Define $\K$-linear maps $
\psi: H^0\ot H\lra H\ot H^0$ and $\eta: H\ot H^0\lra H^0\ot H$
by the following compositions 
\beq
&\psi: H^0\ot H \stackrel{\wh\Delta}\longrightarrow
H^0\ot H\ot H^0\ot H \stackrel{\brhd\ot \lhd}\longrightarrow  H\ot H^0, \label{eq:tmult-1}\\
&\eta: H\ot H^0\stackrel{\wh\Delta'}\longrightarrow
H\ot H^0\ot H\ot H^0\stackrel{\rhd\ot \blhd}\longrightarrow
H^0\ot H. \label{eq:tmult-2}
\eeq
It is easy to see that for any $x\in H$ and $f\in H^0$, 
\[
\baln
\psi(f\ot x)&= \sum_{(f), (x)} \omega(d f_{(2)}, d x_{(1)}) f_{(1)}\brhd x_{(1)}\ot f_{(2)}\lhd x_{(2)}\\
&=\sum_{(f), (x)} \omega(d f-d f_{(1)}, d x-d x_{(3)}) \\
&\times \langle S_0(f_{(1)}), x_{(1)}\rangle  \langle f_{(3)}, x_{(3)}\rangle  x_{(2)}\ot f_{(2)}, \\
\eta(x\ot f) &= \sum_{(f), (x)} \omega(d x_{(2)}, d f_{(1)}) x_{(1)}\rhd  f_{(1)}\ot  x_{(2)}\blhd f_{(2)}\\
&= \sum_{(f), (x)} \omega(d x-d x_{(1)}, d f-d f_{(3)}) \omega(d x_{(3)}, d f_{(3)})\\
& \times \langle f_{(1)}, x_{(1)}\rangle   
\langle S_0(f_{(3)}), x_{(3)}\rangle   f_{(2)}\ot x_{(2)}.
\ealn
\]
The following result will play a crucial role in the quantum double construction. 
\begin{lemma}\label{lem:psi-map}
\begin{enumerate}
\item The maps $\psi$ and $\eta$ are inverses of each other. 
\item The following relations hold.
\begin{enumerate}[label=\roman*)]
\item \label{D-psi}
$(\psi\ot \psi)\wh\Delta =\wh\Delta' \psi: H^0\ot H\lra  H\ot H^0\ot H\ot H^0$;
\item \label{S-psi}
$\psi(S_0\ot S) = (S\ot S_0) \tau_{H^0, H} \psi^{-1} \tau_{H^0, H}: H^0\ot H\lra H\ot H^0$.  
\end{enumerate}
\end{enumerate}
\end{lemma}
\begin{proof}
Part (1) follows from the following computation.
\[
\baln
\eta\psi &= (\rhd(\brhd\ot\id_{H^0})\ot \blhd(\id_H\ot\lhd)) 
			(\id_{H^0}\ot\wh\Delta'\ot\id_H)\wh\Delta\\
			&=(\rhd(\brhd\ot\id_{H^0})\ot \blhd(\id_H\ot\lhd)) (\wh\Delta\ot \id_{H^0\ot H})\wh\Delta
					\quad \text{(by \eqref{eq:Dh-Dh-1})}\\
			&=(\id_{H^0}\ot \blhd(\id_H\ot\lhd)) (\rhd(\brhd\ot\id_{H^0})\ot\id_{H}) 
			\wh\Delta\ot \id_{H^0\ot H})\wh\Delta\\
			&=(\id_{H^0}\ot \blhd(\id_H\ot\lhd))\wh\Delta 
				\qquad\text{(by Corollary \ref{cor:invert-D})}\\
			&=\id_{H^0\ot H}  \qquad\text{(by Corollary \ref{cor:invert-D})}.
\ealn
\]

To prove part {\em (2). i)}, we consider
\[
\baln
(\psi\ot\psi)\wh\Delta 
= ((\brhd\ot\lhd)\wh\Delta\ot (\brhd\ot\lhd)\wh\Delta) \wh\Delta
= (\brhd\ot\lhd\ot \brhd\ot\lhd)(\wh\Delta\ot\wh\Delta) \wh\Delta.
\ealn
\]
Note that $(\wh\Delta\ot\wh\Delta) \wh\Delta=(\id_{H^0\ot H}\ot\wh\Delta\ot \id_{H^0\ot H}) (\id_{H^0\ot H}\ot\wh\Delta) \wh\Delta$. By using \eqref{eq:Dh-Dh-1}, we obtain
$(\wh\Delta\ot\wh\Delta) \wh\Delta=(\id_{H^0\ot H}\ot\wh\Delta\ot \id_{H^0\ot H}) (\id_{H^0}\ot\wh\Delta'\ot \id_H) \wh\Delta$.
 Hence 
\[
\baln
(\psi\ot\psi)\wh\Delta 
&= (\brhd\ot(\lhd\ot \brhd)\wh\Delta\ot\lhd)(\id_{H^0}\ot\wh\Delta'\ot \id_H) \wh\Delta\\
&=(\brhd\ot \id_{H^0\ot H} \ot\lhd)(\id_{H^0}\ot\wh\Delta'\ot \id_H) \wh\Delta \quad \text{(by \eqref{eq:lr-inv})}\\
&= \wh\Delta'(\brhd\ot \lhd) \wh\Delta  \quad \text{(by \eqref{eq:lr-D})}\\
&= \wh\Delta'\psi.
\ealn
\]

Now we prove part {\em (2) ii)}. We have 
\[
\baln
\psi(S_0\ot S)&=(\brhd\ot\lhd)(\id_{H^0}\ot\tau_{H^0, H}\ot \id_H) (S_0\ot S_0\ot S\ot S)(\Delta_0'\ot \Delta')\\
&=(\brhd\ot\lhd)(S_0\ot S\ot S_0\ot S)(\id_{H^0}\ot\tau_{H^0, H}\ot \id_H) (\Delta_0'\ot \Delta').
\ealn
\]
Claim:
\[
\baln
&\brhd(S_0\ot S) = S\blhd\tau_{H^0, H}: H^0\ot H\lra H, \\
&\lhd(S_0\ot S) = S_0\rhd\tau_{H^0, H}: H^0\ot H\lra H^0, 
\ealn
\]
which will be shown later. Using these relations, we obtain 
\[
\baln
\psi(S_0\ot S)
&=(\brhd\ot\lhd)(S_0\ot S\ot S_0\ot S)(\id_{H^0}\ot\tau_{H^0, H}\ot \id_H) (\Delta_0'\ot \Delta')\\
&=(S\ot S_0)(\blhd\ot \rhd)(\tau_{H^0, H}\ot\tau_{H^0, H})(\id_{H^0}\ot\tau_{H^0, H}\ot \id_H) (\Delta_0'\ot \Delta').
\ealn
\]
Note that $(\tau_{H^0, H}\ot\tau_{H^0, H})(\id_{H^0}\ot\tau_{H^0, H}\ot \id_H) (\Delta_0'\ot \Delta') 
= \tau_{H^0\ot H, H^0\ot H}\wh\Delta'\tau_{H^0, H}$. Thus 
\[
\baln
\psi(S_0\ot S)
&=(\brhd\ot\lhd)(S_0\ot S\ot S_0\ot S)(\id_{H^0}\ot\tau_{H^0, H}\ot \id_H) (\Delta_0'\ot \Delta')\\
&=(S\ot S_0)(\blhd\ot \rhd)\tau_{H^0\ot H, H^0\ot H}\wh\Delta'\tau_{H^0, H}\\
&=(S\ot S_0)\tau_{H^0, H}(\rhd\ot \blhd)\wh\Delta'\tau_{H^0, H}\\
&=(S\ot S_0)\tau_{H^0, H}\eta\tau_{H^0, H}.
\ealn
\]
This proves part {\em (2). ii)} in view of part {\em (2). i)}, granting the claim above.

Now we prove the claim. For any $f\ot x\in H^0\ot H$, we have 
\[
\baln
\brhd(S_0\ot S)(f\ot x) &= \sum_{(f)} \omega(d x_{(1)}, d x_{(2)}) \langle S_0^2(f), S(x_{(2)})\rangle S(x_{(1)})\\
&= \sum_{(f)} \omega(d f, d x - d x_{(2)}) \langle S_0(f), x_{(2)}\rangle S(x_{(1)})\\
&= \omega(d f, d x)  S(x\blhd f)
=S(\blhd\tau_{H^0, H}(f\ot x)); 
\ealn
\]
\[
\baln
\lhd (S_0\ot S)(f\ot x) &=S_0(f)\lhd S(x)\\
&= \sum_{(f)} \omega(d f_{(1)}, d f_{(2)}) S_0(f_{(2)})\langle S_0(f_{(1)}),  S(x)\rangle \\
&=  \omega(d f, d x)\sum_{(f)} \omega(d x, d f_{(1)}) \langle f_{(1)},  x\rangle S_0(f_{(2)})\\
&= \omega(d f, d x) S_0(x\rhd f)
= S_0(\rhd\tau_{H^0, H}(f\ot x)).
\ealn
\] 
These relations imply the claim, completing the proof of the lemma. 
\end{proof}


\subsection{Quantum double construction for Hopf $(\Gamma, \omega)$-algebras}
\label{sect:double-gen}

Assume that $(H, \mu, u, \Delta, \epsilon, S)$ is a proper Hopf $(\Gamma, \omega)$-algebra 
with bijective antipode, and denote by $(H^0, \mu_0, u_0, \Delta_0, \epsilon_0, S_0)$
the opposite of the finite dual Hopf $(\Gamma, \omega)$-algebra. Recall that  
the structure maps are defined 
by equations \eqref{eq:H0-1},  \eqref{eq:H0-2} and \eqref{eq:H0-3}. 
Let 
\beq \label{eq:D}
D(H)=H\ot_\K H^0.
\eeq

We have the following results.

\begin{theorem} \label{thm:D-alg}
 The $\Gamma$-graded $\K$-module $D(H)$ can be endowed with an associative $(\Gamma, \omega)$-algebra
 structure with unit map $\ol{u}: \K\lra D(H)$, $a\mapsto a1_{D(H)} = a 1\ot \epsilon$,  and 
 multiplication $\ol{\mu}: D(H)\ot D(H)\lra D(H)$ 
 defined by the following composition
\beq\label{eq:mubar}
D(H)\ot D(H)\stackrel{\id\ot\psi\ot\id}\loongrightarrow H\ot H\ot H^0\ot H^0 
\stackrel{\mu\ot\mu_0}\loongrightarrow D(H), 
\eeq
where $\psi: H^0\ot H\lra H\ot H^0$ is defined by \eqref{eq:tmult-1}.  
\end{theorem}

\begin{theorem} \label{thm:D-Hopf}
 The associative $(\Gamma, \omega)$-algebra $(D(H), \ol{\mu}, \ol{u})$ can be endowed with a Hopf $(\Gamma, \omega)$-algebraic 
structure with 
\beq
\text{co-multiplication}& & \ol{\Delta}:  D(H) \stackrel{\wh\Delta'}\longrightarrow D(H)\ot D(H), \\
\text{co-unit} &  & \ol{\epsilon}:  D(H) \stackrel{\epsilon\ot\epsilon^0}\loongrightarrow \K\ot \K\stackrel{\simeq}\lra \K,\\
\text{antipode} &  & \ol{S}:   D(H) \stackrel{(S_0\ot S)\tau}\loongrightarrow H^0\ot H \stackrel{\psi} \lra D(H).
\eeq 
\end{theorem}

\medskip
We now turn to the proof of the theorems. 

The following facts will be used presently. For any $x, y\in H$ and $f, g\in H^0$, 
\beq\label{eq:mult-d}
\phantom{XXX}
\ol\mu(x\ot f\ot y\ot g)
=\sum_{(f), (y)} \omega(d f_{(2)}, d y_{(1)}) x  (f_{(1)} \brhd y_{(1)}) \ot  (f_{(2)} \lhd y_{(2)}) g, 
\eeq
and in particular, 
\beq
&\ol{\mu}((x\ot \epsilon)\ot (y\ot g)) = x y\ot g,  \label{eq:xyg}\\
&\ol{\mu}((x\ot f)\ot (1\ot g)) = x\ot fg, \label{eq:xfg} \\
&\ol{\mu}((1\ot f)\ot (y\ot \epsilon)) = \psi(f\ot y).
\eeq

\begin{proof}[Proof of Theorem \ref{thm:D-alg}]
The main task is to prove associativity of the multiplication $\ol\mu$.  
We will prove this by direct calculations, 
which are very lengthy but  
involve little more than careful bookkeeping of the commutative factors.  
For peace of mind, we present the details. 

Consider $\ol\mu(\ol\mu(x\ot f\ot y\ot g)\ot z\ot h)$ for any $x, y, z\in H$ and $f, g, h\in H^0$. 
Using \eqref{eq:mult-d} and observing that $\omega(d f_{(2)}, d y_{(1)})= \omega(d f- d f_{(1)}, d y_{(1)})$ in the equation, we can express 
$\ol\mu(\ol\mu(x\ot f\ot y\ot g)\ot z\ot h) $ as 
\[
\baln
(\mu\ot\mu_0)\left(\sum_{(f), (y)} \omega(d f-d f_{(1)}, d y_{(1)}) x  (f_{(1)} \brhd y_{(1)}) \ot  \psi((f_{(2)} \lhd y_{(2)}) g  \ot z) \ot h\right).
\ealn
\]
Now 
\[
\baln
\psi((f_{(2)} \lhd y_{(2)}) g  \ot z) &=\sum \omega(d f_{(3)} +d y_{(2)}, d g_{(1)})\\
&\times  \omega(d f_{(3)} +d y_{(2)}+ d g_{(2)}, d z_{(1)})\\
&\times (f_{(2)} g_{(1)})\brhd z_{(1)}\ot ((f_{(3)}  \lhd y_{(2)}) g_{(2)})\lhd  z_{(2)}\\
&=\sum \omega(d f_{(3)} +d y_{(2)}, d g_{(1)}) \\
&\times  \omega(d f_{(3)} +d y_{(2)}+ d g_{(2)}, d z_{(1)})
  \omega(d g_{(2)}, d z_{(2)})\\
&\times (f_{(2)} g_{(1)})\brhd z_{(1)}\ot (f_{(3)}  \lhd (y_{(2)} z_{(2)})) (g_{(2)}\lhd  z_{(3)})\\
&=\sum \omega(d f_{(3)} +d y_{(2)}, d g_{(1)}+d z_{(1)}) 
 \omega(d g_{(2)}, d z -d z_{(1)})\\
&\times (f_{(2)} g_{(1)})\brhd z_{(1)}\ot (f_{(3)}  \lhd (y_{(2)} z_{(2)})) (g_{(2)}\lhd  z_{(3)}).
 \ealn
\]
Hence 
\beq\label{eq:mu-mu-id}
\baln
&\ol\mu(\ol\mu(x\ot f\ot y\ot g)\ot z\ot h)\\
&= \sum \omega(d f-d f_{(1)}, d y_{(1)})   \omega(d g_{(2)}, d z-d z_{(1)}) \\
&\times \omega(d f_{(3)} +d y_{(2)}, d g_{(1)}+d z_{(1)}) \\
&\times x  (f_{(1)} \brhd y_{(1)}) ((f_{(2)} g_{(1)})\brhd z_{(1)}) \\
&\ot (f_{(3)}  \lhd (y_{(2)} z_{(2)})) (g_{(2)}\lhd  z_{(3)}) h.
\ealn
\eeq

Now consider $\ol\mu(x\ot f\ot \ol\mu(y\ot g\ot z\ot h))$.  We have 
\[
\baln
\ol\mu(y\ot g\ot z\ot h)&=\sum \omega(d g_{(2)},  d z_{(1)})  \\
&\times y  (g_{(1)} \brhd  z_{(1)}) \ot  (g_{(2)}\lhd  z_{(2)}) h\\
&=\sum \omega(d g - d g_{(1)}, d z - d  z_{(2)})  \\
&\times y  (g_{(1)} \brhd  z_{(1)}) \ot  (g_{(2)}\lhd  z_{(2)}) h, 
\ealn
\]
and hence 
\[
\baln
&\ol\mu(x\ot f\ot \ol\mu(y\ot g\ot z\ot h)) \\
&= \sum \omega(d g_{(2)}, d z - d z_{(2)})  \\
&\times (\mu\ot \mu_0)(x\ot \psi(f\ot  y(g_{(1)}\brhd z_{(1)})) \ot (g_{(2)}\lhd z_{(2)})h)
\ealn
\]
Using 
\[
\baln
\psi(f\ot  y(g_{(1)}\brhd z_{(1)}))&=\sum \omega(d y_{(2)}, d g_{(1)}+d z_{(1)}) \\
&\times \omega(d f_{(2)}, d y_{(1)}+d g_{(1)}+d z_{(1)}) \\
&\times f_{(1)} \brhd (y_{(1)} (g_{(1)}\brhd z_{(1)})) \ot  f_{(2)}\lhd   (y_{(2)}  z_{(2)}) \\
&= \sum \omega(d y_{(2)}, d g_{(1)}+d z_{(1)}) \\
&\times \omega(d f_{(3)}, d y_{(1)}+d g_{(1)}+d z_{(1)}) \omega(d f_{(2)}, y_{(1)}) \\
&\times (f_{(1)} \brhd y_{(1)}) ((f_{(2)}g_{(1)}\brhd z_{(1)}) \ot  f_{(3)}\lhd   (y_{(2)}  z_{(2)}), 
\ealn
\]
we obtain 
\beq \label{eq:mu-id-mu}
\baln
&\ol\mu(x\ot f\ot \ol\mu(y\ot g\ot z\ot h)) \\
&= \sum \omega(d g_{(2)}, d z - d z_{(3)})  \omega(d y_{(2)}, d g_{(1)}+d z_{(1)}) \\
&\times \omega(d f_{(3)}, d y_{(1)}+d g_{(1)}+d z_{(1)}) \omega(d f_{(2)}, d y_{(1)}) \\
&\times x (f_{(1)} \brhd y_{(1)}) ((f_{(2)}g_{(1)}\brhd z_{(1)}) \ot  (f_{(3)}\lhd   (y_{(2)}  z_{(2)}) ) (g_{(2)}\lhd z_{(3)})h.
\ealn
\eeq
The product of commutative factors in the above equation can be re-written as 
\[
\baln
&\omega(d g_{(2)}, d z - d z_{(3)})  \omega(d y_{(2)}, d g_{(1)}+d z_{(1)}) \\
&\times \omega(d f_{(3)}, d y_{(1)}+d g_{(1)}+d z_{(1)}) \omega(d f_{(2)}, d y_{(1)}) \\
&= \omega(d g_{(2)}, d z - d z_{(3)})  \omega(d y_{(2)}, d g_{(1)}+d z_{(1)}) \\
&\times 
\omega(d f_{(3)}, d g_{(1)}+d z_{(1)}) \omega(d f_{(3)}, d y_{(1)}) \omega(d f_{(2)}, d y_{(1)}) \\
&=  \omega(d f - d f_{(1)}, d y_{(1)})  \omega(d g_{(2)}, d z - d z_{(3)}) \\
&\times \omega(d f_{(3)}+d y_{(2)}, d g_{(1)}+d z_{(1)}).
\ealn
\]
Therefore, the right hand sides of equations \eqref{eq:mu-mu-id} and \eqref{eq:mu-id-mu} agree, proving the associativity of $\ol\mu$. 

It follows from \eqref{eq:mult-d} that $1\ot\epsilon$ is the multiplicative identity, 
completing the proof of the theorem. 
\end{proof}

The proof of Theorem \ref{thm:D-Hopf} makes essential use of Lemma \ref{lem:psi-map}.  
\begin{proof}[Proof of Theorem \ref{thm:D-Hopf}]
Let us first prove the bi-algebra structure for $D(H)$.

Note that for any $x\in H$ and $f\in H^0$,  we have
\[
\ol{\Delta}(x\ot f)= \sum_{(x), (f)} \omega(d x_{(2)}, d f_{(1)}) x_{(1)}\ot f_{(1)}\ot x_{(2)}\ot f_{(2)}.
\]
In particular, 
\beq
\ol{\Delta}(x\ot \epsilon)= \sum_{(x)} x_{(1)}\ot \epsilon\ot x_{(2)}\ot \epsilon, \label{eq:D-x}\\
\ol{\Delta}(1\ot f)= \sum_{(f)} 1\ot f_{(1)}\ot 1\ot f_{(2)}.\label{eq:D-f}
\eeq
Also,  the following relations immediately follow from the definition of $\ol{S}$. 
\beq
\ol{S}(x\ot \epsilon)= S(x)\ot \epsilon, \quad \ol{S}(1\ot f) = 1\ot S_0(f). \label{eq:S-x-f}
\eeq

We now show that $\ol{\Delta}$ is an algebra homomorphism.
In the calculation below, we show omit the notation for the multiplication $\ol\mu$
whenever that will not cause confusion. 
For example, we will simply write 
$\ol{\mu}((x\ot f)\ot(y\ot g))$ as $(x\ot f)(y\ot g)$ for any $x, y\in H$ and $f, g\in H^0$. 
We can show that 
\[
\baln
\ol{\Delta}((x\ot \epsilon)(y\ot g))&=\ol{\Delta}(x\ot \epsilon)\ol{\Delta}(y\ot g),\\
\ol{\Delta}((x\ot f)(1\ot g))&=\ol{\Delta}(x\ot f)\ol{\Delta}(1\ot g). 
\ealn
\]
Thus $\ol{\Delta}((x\ot f)(y\ot g))=\ol{\Delta}(x\ot \epsilon)\ol{\Delta}((1\ot f) (y\ot \epsilon))\ol{\Delta}(1\ot g)$. If we can show that 
\beq\label{eq:Delta-prod}
\ol{\Delta}((1\ot f) (y\ot \epsilon))=\ol{\Delta}(1\ot f) \ol{\Delta}(y\ot \epsilon), 
\eeq
then $\ol{\Delta}((x\ot f)(y\ot g))= \ol{\Delta}(x\ot f) \ol{\Delta}(y\ot g)$, 
and thus $\ol{\Delta}$  is an algebra homomorphism. 
Let us prove \eqref{eq:Delta-prod}. We have 
\[
\baln
\ol{\Delta}(((1\ot f) (y\ot \epsilon))&= \wh\Delta'\psi(f\ot y)
= (\psi\ot \psi)\wh{\Delta}(f\ot y)  \quad \text{(by Lemma \ref{lem:psi-map}.(2).i))}\\
&= \sum \omega(d f_{(2)}, d y_{(1)})\psi(f_{(1)} \ot y_{(1)}) \ot  \psi(f_{(2)} \ot y_{(2)}).
\ealn
\]
The right hand side can be re-written as 
\[
\baln
&\sum \omega(d(f_{(2)}), d(y_{(1)})) (1\ot f_{(1)})(y_{(1)}\ot \epsilon) \ot  (1\ot f_{(2)})(y_{(2)}\ot \epsilon)\\
&= \sum (1\ot f_{(1)}   \ot  1\ot f_{(2)})  (y_{(1)}\ot \epsilon \ot  y_{(2)}\ot \epsilon)\\
&=\ol{\Delta}(1\ot f) \ol{\Delta}(y\ot \epsilon), 
\ealn
\]
proving equation \eqref{eq:Delta-prod}. 

It is clear now that the map $\ol{\epsilon}: D(H)\lra \K$ is an algebra homomorphism. It is also evident that 
$\ol{\mu}(\ol{\epsilon}\ot\id)\ol{\Delta}= \ol{\mu}(\id\ot\ol{\epsilon})\ol{\Delta}= \id$. 
This proves the bi-algebra structure of $D(H)$. 

We now consider the map $\ol{S}$. Let $x,  y\in H$ and $f, g\in H^0$. We have 
\beq
&&\ol{S}((x\ot \epsilon)(y\ot g))= \omega(d x, d y +d g) \ol{S}(y\ot g) \ol{S}(x\ot \epsilon), \label{eq:Sxyg} \\
&&\ol{S}((x\ot f)(1\ot g))=  \omega(d x+d f, d g) \ol{S}(1\ot g) \ol{S}(x\ot f), \label{eq:Sxfg}\\ 
&&\ol{S}((x\ot \epsilon)(1\ot f))= \ol{S}(x\ot f)   =\omega(d x, d f) \ol{S}(1\ot f))   \ol{S}(x\ot \epsilon), \label{eq:Sxf}\\
&&\ol{S}((1\ot f) (x\ot \epsilon)) = \omega(d f, d x)  \ol{S}(x\ot \epsilon) \ol{S}(1\ot f), \label{eq:Sfx}
\eeq
where the first two equations easily follow from the definition of $\ol{S}$, 
and \eqref{eq:Sxf} is a special case of \eqref{eq:Sxfg}. The relation \eqref{eq:Sfx} can be shown 
by using Lemma  \ref{lem:psi-map}.(2).ii) as follows.
\[
\baln
\ol{S}((1\ot f)(x\ot \epsilon))
&=  \ol{S}\psi(f\ot x) = \mu (S_0 \ot S)\tau \psi(f\ot x)\\
&=  (S\ot S_0) \tau \psi^{-1} \psi(f\ot x) \qquad \text{(by Lemma \ref{lem:psi-map}.(2).ii))}\\
&=\omega(d f, d x)(S(x)\ot S_0(f)) \\
&= \omega(d f, d x) \ol{S}(x\ot \epsilon) \ol{S}(1\ot f).
\ealn
\]

Now we show that $\ol{S}$ is an algebra anti-homomorphism. We have 
\[
\baln
\ol{S}((x\ot f)(y\ot g))
&=  \omega(d x, d f+d y +d g) \omega(d f+d y, d g) \\
&\times \ol{S}(1\ot g)  \ol{S}((1\ot f)(y\ot \epsilon))  \ol{S}(x\ot \epsilon)  
\quad \text{(by \eqref{eq:Sxyg}, \eqref{eq:Sxfg})}.\\
\ealn
\]
The right hand side can be re-written as 
\[
\baln
&\omega(d x, d f+d y +d g) \omega(d f+d y, d g) \omega(d f, d y)\\
&\times \ol{S}(1\ot g)  \ol{S}(y\ot \epsilon)  \ol{S}(1\ot f)   \ol{S}(x\ot \epsilon)  
\quad \text{(by \eqref{eq:Sfx})}\\
&=  \omega(d x+d f, d y +d g)  \ol{S}(y\ot g)    \ol{S}(x\ot f) 
\quad \text{(by \eqref{eq:Sxyg}, \eqref{eq:Sxfg})}, 
\ealn
\]
that is, 
$
\ol{S}((x\ot f)(y\ot g)) = \omega(d x+d f, d y +d g) \ol{S}(y\ot g)    \ol{S}(x\ot f). 
$

Finally we verify that $\ol{S}$ satisfies the required relations with $\ol{\Delta}$ and $\ol{\epsilon}$.  
We have 
\[
\baln
\ol{\mu}(\ol{S}\ot\id_{D(H)})\ol{\Delta}(x\ot f)&=\sum_{(x), (f)} \omega(d x_{(2)},  d f_{(1)})
\ol{S}(x_{(1)}\ot f_{(1)})(x_{(2)}\ot f_{(2)}). 
\ealn
\]
We can re-write the right hand side as follows
\[
\baln
&\sum_{(x), (f)} \omega(d x,  d f_{(1)})
 (1\ot S_0(f_{(1)}))(S(x_{(1)})\ot \epsilon) (x_{(2)}\ot f_{(2)})   \quad \text{(by \eqref{eq:Sxf})}\\
&=\sum_{(f)} \omega(d x,  d f_{(1)})
  (1\ot S_0(f_{(1)}))(\sum_{(x)} S(x_{(1)})x_{(2)}\ot f_{(2)}) \quad \text{(by \eqref{eq:xyg})}\\
&=\epsilon(x)\sum_{(f)} (1\ot S_0(f_{(1)})) (1\ot f_{(2)}) \quad \text{(by property of $S$)}\\
&=\epsilon(x)\sum_{(f)} 1\ot S_0(f_{(1)}) f_{(2)}  \quad \text{(by \eqref{eq:xfg})}\\
&=\epsilon(x) f(1) 1_{D(H)}  \quad \text{(by property of $S_0$)}.
\ealn
\]
where we recall that $1_{D(H)}=1\ot \epsilon$ is the identity of $D(H)$.  Thus 
\beq
\ol{\mu}(\ol{S}\ot\id_{D(H)})\ol{\Delta}(x\ot f) = \ol{\epsilon}(x\ot f) 1_{D(H)},  
\eeq
and we can similarly show that 
\beq
\ol{\mu}(\id_{D(H)}\ot \ol{S})\ol{\Delta}(x\ot f) = \ol{\epsilon}(x\ot f) 1_{D(H)}. 
\eeq

This completes the proof of Theorem \ref{thm:D-Hopf}.
\end{proof}

\begin{definition}\label{def:D(H)}
Call the Hopf $(\Gamma, \omega)$-algebra  $(D(H), \ol{\mu}, \ol{u}, \ol{\Delta}, \ol{\epsilon}, \ol{S})$ the 
quantum double of $(H, \mu, u, \Delta, \epsilon, S)$.  
\end{definition}

By inspecting the definitions of the structure maps of the quantum double, particularly equations \eqref{eq:xyg}, \eqref{eq:xfg}, \eqref{eq:D-x}, \eqref{eq:D-f} and \eqref{eq:S-x-f},   
we can see that $D(H)$ contains $H$ and $H^0$ as Hopf $(\Gamma, \omega)$-sub-algebras. 
More explicitly, 
\begin{lemma}\label{lem:subalgs}
There are injective Hopf $(\Gamma, \omega)$-algebra homomorphisms 
\[
\baln
&\iota: (H, \mu, u, \Delta, \epsilon, S)\lra (D(H), \ol{\mu}, \ol{u}, \ol{\Delta}, \ol{\epsilon}, \ol{S}),\quad x\mapsto x\ot\epsilon, \\
&\iota_0: (H^0, \mu_0, u_0, \Delta_0, \epsilon_0, S_0)\lra (D(H), \ol{\mu}, \ol{u}, \ol{\Delta}, \ol{\epsilon}, \ol{S}), \quad f\mapsto 1\ot f.
\ealn
\]
\end{lemma}

\subsubsection{Quasi-triangular structure of the quantum double}
Let  $(H, \mu, u, \Delta, \epsilon, S)$ be a proper Hopf $(\Gamma, \omega)$-algebra over $\K$, 
with $H$ being a free $\K$-module. 
Let $B=\{x_i\mid i\in I\}$ (for an index set $I$) be a homogeneous basis for $H$, and let $\ol{B}=\{f_i\mid i\in I\}$ 
be a homogeneous basis for $H^0$ such that $\langle f_i, x_j\rangle =\delta_{i j}$ for all $i, j\in I$. Then
for any $y\in H$  and $f\in H^0$, 
\beq\label{eq:ident}
y= \sum_{i\in I} \langle f_i, y\rangle x_i, \quad f= \sum_{i\in I} \langle f, x_i\rangle f_i. 
\eeq

\begin{theorem}
Let  $(H, \mu, u, \Delta, \epsilon, S)$ be a proper Hopf $(\Gamma, \omega)$-algebra over $\K$ 
with $H$ being a free $\K$-module, 
and retain  notation above. 
Assume that there is some appropriate completion  $D(H)\hat\ot D(H)$ of $D(H)\ot D(H)$ which contains 
the element  
\beq
\ol{R}=\sum_{i\in I} \iota(x_i)\ot \iota_0(f_i).
\eeq
Then the quantum double $(D(H), \ol{\mu}, \ol{u}, \ol{\Delta}, \ol{\epsilon}, \ol{S})$ is a 
quasi-triangular Hopf $(\Gamma, \omega)$-algebra with $\ol{R}$ as the universal $R$-matrix. 
\end{theorem}
\begin{proof}
Let us make some preparations for proving the theorem. 

We use  \eqref{eq:ident} to express $\iota_0(f)\iota(y) $, $\ol\Delta(\iota_0(f))$ and $\ol\Delta(\iota(f))$ 
in such a way that they can be readily used in the proof of the theorem. We first re-write 
\[
\baln
\psi(f\ot y)
&=\sum \omega(d f-d f_{(1)}, d y-d y_{(3)}) \\
&\times  \langle S_0(f_{(1)}), y_{(1)}\rangle  \langle f_{(3)}, y_{(3)}\rangle  y_{(2)}\ot f_{(2)} 
\ealn
\]
in two ways as follows
\[
\baln
\psi(f\ot y)
&=\sum \omega(d f-d f_{(1)}, d y-d y_{(3)}) 	\\
&\times  \langle S_0(f_{(1)}), y_{(1)}\rangle  \langle f_{(2)}, x_\ell \rangle 
\langle f_{(3)}, y_{(3)}\rangle  y_{(2)}\ot f_\ell, \\
\psi(f\ot y)
&=\sum \omega(d f-d f_{(1)}, d y-d y_{(3)})  \\
&\times  \langle S_0(f_{(1)}), y_{(1)}\rangle  \langle f_\ell, y_{(2)}\rangle \langle f_{(3)}, y_{(3)}
\rangle  x_\ell\ot f_{(2)}. 
\ealn
\]
By using properties of $\Delta$ and $\Delta_0$, we obtain 
\[
\baln
\psi(f\ot y)
&= \sum \omega(d f+d y_{(1)}, d y-d y_{(3)} - d y_{(1)})  
\omega(d y_{(3)},  d x_\ell) \\
&\times  \langle f_{(1)} \ot f_{(2)}\ot f_{(3)}, S^{-1}(y_{(1)}) \ot x_\ell\ot  y_{(3)}\rangle  
y_{(2)}\ot f_\ell\\
&= \sum \omega(d y-d y_{(3)}, d y_{(3)}+d x_\ell)\\
&\times  \langle f, y_{(3)} x_\ell S^{-1}(y_{(1)}) \rangle  y_{(2)}\ot f_\ell \quad \text{(by property of $\Delta_0$)}, 
\ealn
\]
\[
\baln
\psi(f\ot y)
&=\sum \omega(d f-d f_{(1)}, d y+d f_{(3)}) \\
&\times \omega(d f_{(3)}, d f_{(1)} +d f_\ell )  \omega(d f_\ell, d f_{(1)})\\
&\times  \langle S_0(f_{(1)})\ot f_\ell\ot f_{(3)}, y_{(1)}\ot y_{(2)}\ot y_{(3)}\rangle  x_\ell\ot f_{(2)}\\
&=\sum 
 \omega( d f_{(1)} +d f_\ell, d f_{(2)})  \omega(d f_\ell, d f_{(1)})\\
&\times  \langle S_0(f_{(1)}) f_\ell f_{(3)}, y\rangle  x_\ell\ot f_{(2)} \quad \text{(by property of $\Delta$)}.
\ealn
\]
These  immediately lead to the following two relations

\beq
\iota_0(f)\iota(y) 
&=& \sum \omega(d y-d y_{(3)},  d y_{(3)}+d x_\ell)  \label{eq:fy-1}\\
&&\times\langle f, y_{(3)} x_\ell S^{-1}(y_{(1)}) \rangle  \iota(y_{(2)}) \iota_0(f_\ell), \nonumber \\
\iota_0(f)\iota(y) 
&=&\sum \omega( d f_{(1)}, d f_{(2)})   \omega(d f_\ell, d f_{(1)}+d f_{(2)})   \label{eq:fy-2}\\
&&\times  \langle S_0(f_{(1)}) f_\ell f_{(3)}, y\rangle  \iota(x_\ell) \iota_0( f_{(2)}). \nonumber
\eeq

By using \eqref{eq:ident}, we obtain  
\[
\baln
{\Delta}(y)
&= \sum_{r, s\in I}\omega(d f_s, d f_r)
\langle f_r \ot f_s, \Delta(y)\rangle   x_r\ot x_s\\
&= \sum_{r, s\in I}\omega(d f_s, d f_r)
\langle f_r  f_s, y\rangle   x_r\ot x_s, 
\ealn
\]
and similar calculations also show that for any $g\in H^0$, 
\[
\baln
\Delta_0(g)
&= \sum_{r, s\in I}\langle g, x_s x_r\rangle   f_r\ot f_s.
\ealn
\]
These immediately lead to the following relations
\beq
&&\ol\Delta(\iota(y))
= \sum_{r, s\in I}\omega(d f_s, d f_r)
\langle f_r  f_s, y\rangle   \iota(x_r)\ot \iota(x_s), \quad y\in H, \label{eq:Dy} \\
&&\ol\Delta_0(\iota_0(g))
= \sum_{r, s\in I}\langle g, x_s x_r\rangle   \iota_0(f_r)\ot \iota_0(f_s), \quad \forall g\in H^0, \label{eq:D0h}
\eeq

Also we write $E=\sum_i x_i \ot  f_i$. Then 
\[
\ol{R} = (\iota\ot\iota_0)(E).
\]

Now we prove the theorem. 

Let us first show that $\ol{R}$  satisfies equation \eqref{eq:YB-1}.  
For any $y\in H$ and  $f\in H^0$,  
\[
\baln
\ol{R}\,  \ol\Delta(\iota(y)) &= \sum \omega(d f_i, d y_{(1)}) \iota(x_i) \iota(y_{(1)}) \ot \iota_0(f_i) \iota(y_{(2)}), \\
\ol{\Delta}'(\iota_0(f)) \ol{R} &=\sum
\omega(d f_{(1)}, d f_{(2)}+ d x_i)
\iota_0(f_{(2)})   \iota(x_i)\ot  \iota_0(f_{(1)})  \iota_0(f_i). 
\ealn
\]
Using \eqref{eq:fy-1} in the first equation, we obtain 
\[
\baln
\ol{R}\,  \ol\Delta(\iota(y)) 
 &= \sum \omega(d f_i, d y_{(1)}) 
\omega(d y_{(2)}+d y_{(3)}, d y_{(4)}+d x_\ell)\\
&\times \langle f_i, y_{(4)} x_\ell S^{-1}(y_{(2)}) \rangle  
 \iota(x_i) \iota(y_{(1)}) \ot  \iota(y_{(3)}) \iota_0(f_\ell)\\
 &= \sum \omega(d y_{(1)}+ d y_{(2)}+d y_{(3)}, d y_{(4)}+d x_\ell)\\
&\times \iota( y_{(4)} x_\ell S^{-1}(S(y_{(1)})y_{(2)}) \ot  \iota(y_{(3)}) \iota_0(f_\ell)\\
 &= \sum \omega(d y_{(1)}, d y_{(2)}+d x_\ell)\\
&\times \iota(y_{(2)} x_\ell)  \ot  \iota(y_{(1)}) \iota_0(f_\ell)  \qquad \text{(by property of $S$)}\\
&=\ol\Delta'(\iota(y)) \ol{R}. 
\ealn
\]
Using \eqref{eq:fy-2} in the second equation, we obtain 
\[
\baln
\ol{\Delta}'(\iota_0(f)) \ol{R} 
%
%
&=\sum\omega(d f_{(1)}, d f-d f_{(1)} + d x_i)\\ 
&\times \omega( d f_{(2)}, d f_{(3)})   \omega(d f_\ell, d f_{(2)}+d f_{(3)}) \\
&\times  \langle S_0(f_{(2)}) f_\ell f_{(4)}, x_i\rangle  \iota(x_\ell) \iota_0( f_{(3)})
\ot  \iota_0(f_{(1)})  \iota_0(f_i) \\
%
%
&=\sum \omega(d f_{(1)}+ d f_{(2)}, d f_{(3)})   \omega(d f_\ell, d f_{(1)}+d f_{(2)}+d f_{(3)}) \\
&\times  \iota(x_\ell) \iota_0( f_{(3)})
\ot  \iota_0(f_{(1)}S_0(f_{(2)}) f_\ell f_{(4)}) \\
&=\sum \omega(d f_\ell, d f_{(1)}) \\
&\times  \iota(x_\ell) \iota_0( f_{(1)})
\ot   \iota_0 (f_\ell) \iota_0( f_{(2)}) \qquad \text{(by property of $S_0$)}\\
&= \ol{R}\,  \ol\Delta(\iota_0(f)).
\ealn
\]
This proves that  $\ol{R}$ satisfies equation \eqref{eq:YB-1}.

Now we consider 
$(\ol{\Delta}\ot\id_{D(H)})\ol{R}=\sum_{i\in I} \ol{\Delta}(\iota(x_i))\ot \iota_0(f_i).$
Using \eqref{eq:Dy}, we obtain
\[
\baln
(\ol{\Delta}\ot\id_{D(H)})\ol{R}&= \sum_{r, s\in I}\omega(d f_s, d f_r)
 \iota(x_r)\ot \iota(x_s) \ot  \sum_{i\in I} \langle f_r  f_s, x_i\rangle  \iota_0(f_i)\\
&= \sum_{r, s\in I}\omega(d f_s, d f_r)
 \iota(x_r)\ot \iota(x_s) \ot  \iota_0(f_r)  \iota_0(f_s)\\
&=\sum_{r, s\in I}
 (\iota(x_r)\ot 1\ot \iota_0(f_r))  (1\ot \iota(x_s) \ot  \iota_0(f_s))\\
&=\ol{R}_{1 3} \ol{R}_{ 2 3}. 
\ealn
\]
By using \eqref{eq:D0h}, we can similarly show that 
\[
\baln
(\id_{D(H)}\ot\ol{\Delta})\ol{R}
&=\sum_{r, s\in I} \sum_{i\in I}\langle f_i, x_s x_r\rangle\iota(x_i) \ot    \iota_0(f_r)\ot \iota_0(f_s)\\
&=\sum_{r, s\in I} \iota(x_s) \iota(x_r) \ot    \iota_0(f_r)\ot \iota_0(f_s)\\
&= \ol{R}_{1 3} \ol{R}_{ 1 2}. 
\ealn
\]
This shows that $\ol{R}$ satisfies equation \eqref{eq:YB-2}.

To prove that $\ol{R}$ satisfies \eqref{eq:S-R}, we note that 
\[
(S\ot\id_{H^0})E= (\id_H\ot S_0^{-1})E, 
\]
as can be verified by the following calculations.
\[
\baln
\sum_i S(x_i)\ot f_i 
&= \sum_{i, j} \langle f_j,  S(x_i)\rangle  x_j \ot f_i 
= \sum_{i, j} \langle S(f_j),  x_i\rangle  x_j \ot f_i \\
&=\sum_i x_i\ot S(f_i) =   \sum_i x_i\ot S_0^{-1}(f_i). 
\ealn
\]
It immediately implies that
\[
(\ol{S}\ot \id_{D(H)})\ol{R} = ( \id_{D(H)}\ot \ol{S}^{-1})\ol{R}.
\]

Consider $
\ol{R} (\id_{D(H)}\ot \ol{S}^{-1})\ol{R}= (\iota\ot\iota_0)(E(\id_H\ot S_0^{-1})E).
$
We have 
\[
\baln
E(\id_H\ot S_0^{-1})E
&=\sum_{i, j, \ell} \omega(d f_i, d x_j)  x_i x_j \ot  f_i S_0^{-1}(f_j)\\
&=\sum_{i, j, \ell} \omega(d f_i, d x_j)  \langle  f_i S_0^{-1}(f_j), x_\ell  \rangle x_i x_j  \ot  f_\ell   \\
&=\sum_{i, j, \ell} \omega(d f_i, d x_j)  \langle  f_i\ot  S_0^{-1}(f_j), \Delta(x_\ell)  \rangle x_i x_j  \ot  f_\ell   \\
&=\sum_{i, j, \ell} \omega(d f_i, d x_j)  \langle  f_i\ot  f_j, (\id\ot S)\Delta(x_\ell)  \rangle x_i x_j  \ot  f_\ell   \\
&=\sum_\ell \mu(\id\ot S)\Delta(x_\ell)  \ot  f_\ell   
=\sum_\ell \epsilon(x_\ell)  \ot  f_\ell. 
\ealn
\]
Note that
\beq\label{eq:counit-E}
\sum_\ell \epsilon(x_\ell)  \ot  f_\ell =  1  \ot  \epsilon=  \sum_\ell x_\ell  \ot  \epsilon_0(f_\ell).
\eeq
Thus it follows the first part of the above relation that
$
E(\id_H\ot S_0^{-1})E= 1  \ot  \epsilon. 
$
Hence 
\[
\ol{R} (\id_{D(H)}\ot \ol{S}^{-1})\ol{R}=1_{D(H)}\ot 1_{D(H)}, 
\]
proving that $\ol{R}$ satisfies \eqref{eq:S-R}.

Finally it follows from \eqref{eq:counit-E} that  $\ol{R}$ satisfies equation \eqref{eq:counit-R}, 
proving the theorem.
\end{proof}

\subsection{Topological Hopf $(\Gamma, \omega)$-algebras}\label{sect:topo}

There are various definitions of topological Hopf algebras depending on the contexts. 
We consider a special kind of topological Hopf $(\Gamma, \omega)$-algebras, 
which are over the formal power series ring.
We also extend the theory to the formal Laurent series ring, 
where the quantum double construction for quantum groups works more naturally.
  
The material of this subsection is included for the convenience of the reader.  We do not claim any originality.  An extensive treatment of topological Hopf algebras over the formal power series ring 
can be found in \cite[\S XVI.1-4]{K}.

\subsubsection{The formal power series ring}
The formal power series ring in $\hbar$ is 
\beq\label{eq:Chbar}
\C[[\hbar]]=\Big\{\sum_{i=0}^\infty c_i \hbar^i\, \Big| \,  c_i \in \C\Big\}.
\eeq
It is endowed with the $\hbar$-adic topology, 
which is defined by a neighbourhood base  
$\{a+ {\mathfrak m}^k \mid k=0, 1, \dots\}$ of each element $a\in \C[[\hbar]]$, 
where ${\mathfrak m}=\hbar \C[[\hbar]]$ (the maximal ideal), 
and hence ${\mathfrak m}^k= \hbar^k \C[[\hbar]]$.  
As translation by $a$ is a homeomorphism, the topology is in fact defined 
by the neighbourhood base $\{ {\mathfrak m}^k \mid k=0, 1, \dots\}$ of $0$.

The $\hbar$-adic topology is also the metric topology of a non-Archimedean norm defined as follows. 
Define the valuation 
\beq\label{eq:valu}
v_\hbar: \C[[\hbar]]\lra \Z_+\cup\{+\infty\}
\eeq
such that $v_\hbar(0)=+\infty$, and for any $a =\sum_{i=0}^\infty c_i \hbar^i\in \C[[\hbar]]$ with $c_i\in \C$,  
\[
v_\hbar(a)=min\{i\mid c_i\ne 0\}. 
\]
Fix a positive number $\rho>1$.  The non-Archimedean norm is defined by 
\beq\label{eq:norm}
\|a\|= \rho^{-v_\hbar(a)}, \quad a\in\C[[\hbar]].
\eeq
Then $\C[[\hbar]]$ is complete, e.g., by considering Cauchy sequences. 

Topological $\C[[\hbar]]$-modules are equipped with the $\hbar$-adic topology defined by the neighbourhood base $\{\hbar^k V\mid k=0, 1, \dots\}$  at $0$.
Observe that every $\C[[\hbar]]$-module homomorphism $f: V\lra W$ is continuous
with respect to the $\hbar$-adic topology, since $f(\hbar^n V)=\hbar^n f(V)\subset \hbar^n W$.  
In particular $V^*=\Hom_{\C[[\hbar]]}(V, \C[[\hbar]])$ 
is the continuous dual module of $V$, i.e., the 
set of $\C[[\hbar]]$-linear maps $V\lra \C[[\hbar]]$. 

The completion of a topological $\C[[\hbar]]$-module $V$ is the inverse limit 
\[
\wh{V}:=\varprojlim_n V/\hbar^n V, 
\]
which is separated (i.e.,  Hausdorff) under its natural $\hbar$-adic topology.  

The topological tensor product $V\hat\ot_{\C[[\hbar]]} W$ 
of any two complete separated  topological $\C[[\hbar]]$-modules $V$ and $W$ 
is the $\hbar$-adic completion of the algebraic tensor product $V\ot_{\C[[\hbar]]}W$.  
The $\hbar$-adic topology on the algebraic tensor product is defined by 
the fundamental system of neighbourhoods of \(0\) given by 
$\hbar^n(V\ot_{\C[[\hbar]]}W)$ for $n\ge 0$, 
and the topological tensor product is the inverse limit 
\[
V\widehat{\otimes }_{\mathbb{C}[[\hbar]]}W:=\varprojlim\limits_n\frac{V\otimes _{\mathbb{C}[[\hbar]]}W}{\hbar^n(V\otimes _{\mathbb{C}[[\hbar]]}W)}.
\]

Since $\C[[\hbar]]$-module homomorphisms are automatically continuous in the $\hbar$-adic topology, it is straightforward to extend definitions of various algebras to this topological setting. 

\begin{definition} 
A topological associative $(\Gamma, \omega)$-algebra over $\C[[\hbar]]$ is 
a  $\Gamma$-graded topological $\C[[\hbar]]$-module $H$ equipped with an associative multiplication 
$\mu: H\hat\ot H$ $\lra$ $H$ and a unit map $u: \C[[\hbar]]\lra H$, 
which are homogeneous $\C[[\hbar]]$-module homomorphisms (thus continuous) of degree $0$. 
\end{definition} 
\begin{definition} 
A topological $(\Gamma, \omega)$-algebra $(H, \mu, u)$ over $\C[[\hbar]]$ is a topological Hopf $(\Gamma, \omega)$-algebra if there are $\C[[\hbar]]$-algebra homomorphisms, 
$\Delta: H\lra  H\hat\ot H$ (the co-multiplication), 
$\varepsilon: H\lra\C[[\hbar]]$ (co-unit), and an
anti-homomorphism 
$S: H\lra H$ (antipode),
which are homogeneous  of degree $0$, and satisfy the conditions \eqref{eq:co-asso}, \eqref{eq:counit-Del} and \eqref{eq:Del-S}.
\end{definition}

Consider the dual $\C[[\hbar]]$-module $H^*$ of $H$ and the
adjoint maps of $\Delta$, $\mu$, $u$, $\varepsilon$  and $S$
respectively defined by analogues of \eqref{eq:Del*}--\eqref{eq:S*}, 
where $H^*\hat\ot H^*$ and $(H\hat\ot H)^*$ are used
for the definitions of $\Delta^*$ and $\mu^*$  in \eqref{eq:maps*}.  
Inspired by \cite[\S4.1.D]{CP}, we take 
\[
H^0=\{ f\in H^*\mid \mu^*(f)\in H^*\hat\ot H^*\}.
\] 
This is a topological Hopf $(\Gamma, \omega)$-algebra. 
Given any sequence $(g_0, g_1, g_2, \dots)$ of elements in $H^0$, let $g=\sum_i \hbar^i g_i$. Then
$\mu^*(g) = \sum_i \hbar^i \mu^*(g_i)\in H^*\hat\ot H^*$, and hence $H^0$
is complete with respect to the $\hbar$-adic topology. We call $H$ proper if there is a positive integer $D$ such 
that for any $x\in \hbar^k H$ with non-zero image in $H/  \hbar^{k+1}H$, we have $H^0(x)/\hbar^{k+D}\C[[\hbar]]\ne 0$, 
where $H^0(x)=\{f(x)\mid f\in H^0\}$. 

We also use a special class of $\C[[\hbar]]$-modules, topologically free modules
(see  \cite[\S XVI. 2]{K} for a detailed treatment).
A $\C[[\hbar]]$-module $V[[\hbar]]$ is topologically free if it is isomorphic to 
\[
V[[\hbar]]=\Big\{\sum_{i=0}^\infty \hbar^i v_i \mid v_i\in V\Big\} \cong\prod_{n=0}^\infty V, 
\]
for some $\C$-vector space $V$.
Any topologically free $\C[[\hbar]]$-module $M$ is isomorphic to 
$M_0[[\hbar]]$ with $M_0=\frac{M}{\hbar M}$, and its dual module $M^*$ is isomorphic to $M_0^*[[\hbar]]$. 
If $M$ and $N$ are topologically free $\C[[\hbar]]$-modules,  $M\hat\ot N\simeq (M_0\ot N_0)[[\hbar]]$.

\subsubsection{Topological Hopf $(\Gamma, \omega)$-algebras over the formal Laurent series ring}\label{sect:lattice} {\ }

\medskip
\noindent
{\em The formal Laurent series ring}.
The formal Laurent series ring 
\beq\label{eq:Laurent}
\K=\C((\hbar)) =\left\{\hbar^{-k} \sum_{i=0}^\infty c_i \hbar^i \mid c_i \in \C, k\in\Z_+\right\} 
\eeq
is the fraction field of $R=\C[[\hbar]]$. The valuation extends to
\[
v_\hbar\left(\sum_{i=k}^\infty c_i\hbar^i\right)
=
\min\{i\mid c_i\neq0\},
\]
and defines the non-Archimedean norm
\[
\|a\|=\rho^{-v_\hbar(a)}.
\]
Thus $\K$ is itself a complete valued field. Its valuation topology
restricts to the $\hbar$-adic topology on $R$, and $R$ is an
open and closed subring of $\K$.

There is, however, an important distinction between the topology on the
field $\K$ and the $\hbar$-adic filtration of a $\K$-vector space. Since
$\hbar$ is invertible in $\K$, for every $\K$-vector space $M$ one has
$
\hbar^n M=M.
$
Consequently, the filtration by the subspaces $\hbar^n M$ is constant and
does not define a useful topology on $M$.

We therefore do not define the topology by this filtration. 
Instead, we retain the $\hbar$-adically complete
$R$-module as a distinguished lattice.

\medskip
\noindent
{\em Lattice equipped vector spaces}. 

\begin{definition}\label{def:latt-vect}
Let $V$ be a $\K$-vector space and let $V_R\subset V$ be an $R$-submodule
such that
$
V\cong  V_R\otimes_R \K.
$
We call $V_R$ an $R$-lattice in $V$, and say that $V$ is a latticed equipped $\K$-vector space.

A $\K$-linear map
$
f:V\longrightarrow W
$
between lattice equipped $\K$-vector spaces
is \emph{lattice continuous} if, for some integer $N\geq 0$,
\[
\hbar^Nf(V_R)\subseteq W_R.
\]
\end{definition}

If $V$ is a finite dimensional lattice equipped $\K$-vector space, 
every $\K$-linear map from $V$ to another
lattice equipped space is lattice continuous: after choosing an
$R$-basis of $V_R$, finitely many powers of $\hbar$ clear the denominators
of the images of the basis vectors.

The lattice continuous dual is
\beq\label{eq:V-vee}
V^\vee
:=
\Hom_R(V_R,R)\ot_R\K.
\eeq
Equivalently, it consists of the $\K$-linear functionals
$f:V\to \K$ for which
$
f(V_R)\subseteq\hbar^{-N}R
$
for some $N\ge 0$.

Let $(V,V_R)$ and $(W,W_R)$ be lattice equipped $\K$-vector spaces, with
$V_R$ an $W_R$ complete and separated. Define their completed tensor product
by
\beq\label{eq:ot_K}
V\wh\ot_\K W
:=
(V_R\wh\ot_R W_R)\ot_R\K.
\eeq
Thus completion is performed before extending scalars.

The lattice equipment is part of the structure, so these constructions
depend on the chosen $R$-lattices. This is natural for quantum
groups, where the $R$-model is part of the data.

\medskip
\noindent
{\em The lattice  topology}. 

\begin{definition}[Lattice Topology]\label{def:latt}
For a lattice equipped $\K$-vector space $(V, V_R)$,  
the \emph{lattice topology} on $V$ is the unique translation-invariant topology,  
for which the fundamental neighbourhoods of $0$ are 
$
 \hbar^n V_R, $ for $n \ge 0.
$
\end{definition}

The lattice continuity condition is equivalent to continuity for the lattice topologies. 
Indeed,
$
\hbar^Nf(V_R)\subseteq W_R
\text{ implies }
f(\hbar^{N+k}V_R)\subseteq\hbar^kW_R
$
for every $k\ge 0$, and the converse follows by taking $k=0$.

\begin{remark}
The lattice topology is understood relative to the chosen
lattice; it is not the $\hbar$-adic topology of the $\K$-module itself.
\end{remark}

\medskip
\noindent
{\em Topological Hopf $(\Gamma, \omega)$-algebras over $\K$}. 
A topological Hopf $(\Gamma,\omega)$-algebra over $\K$ is a lattice equipped
$\Gamma$-graded $\K$-vector space
\[
(H,H_R),\qquad H=H_R\ot_R\K,
\]
where $H_R$ is a topological Hopf $(\Gamma,\omega)$-algebra over $R$, and
the Hopf structure maps on $H$ are obtained by scalar extension.
The completed tensor product in the co-product is understood via
\eqref{eq:ot_K}. All continuity statements are relative to the
distinguished lattice $H_R$.

\section*{List of notations}
We list frequently used notations in each section and give brief explanations. 

\renewcommand{\descriptionlabel}[1]{\hspace{\labelsep}\normalfont #1}
\begin{description}[
  leftmargin=.5cm,
  labelwidth=3.5cm,
  labelsep=0.3cm
]

\item[Section  2] {\ }

\begin{description}[
  leftmargin=1.5cm,
  labelwidth=3.5cm,
  labelsep=0.3cm
]
    \item[$\Gamma$] additive abelian group
    \item[$\omega$] commutative factor
    \item[$\Gamma^\pm$] $\{\alpha\in \Gamma\mid \omega(\alpha, \alpha)=\pm 1\}$
    \item[$\Gamma_{\Z_2}$] $2$-torsion subgroup
    \item[$\Gamma_R(V)$]  support of $\Gamma$-graded vector space $V$
    \item[$\Gamma_{R, \Z_2}(V)$]  $\Gamma_R(V)\cap\Gamma_{\Z_2}$
   \item[$\Gamma_R^\pm(V)$]  $\Gamma_R(V)\cap\Gamma^\pm$
   \item[$V_\pm$]  subspaces with graded support $\Gamma_R^\pm(V)$

   \item[$\Vect$]  category of $\Gamma$-graded vector spaces
   	\item[$d(v)$]	degree of $v$		 			 
   \item[$\tau_{V, W}$]   $\omega$-permutation, \eqref{eq:def-tau}
   \item[$\digamma_\xi$] degree shifting functor 
   \end{description}

\bigskip
 \item[Section  3]{\ }
 
\begin{description}[
  leftmargin=1.5cm,
  labelwidth=3.5cm,
  labelsep=0.3cm
] 
   \item[$\kappa$] bilinear form, non-degenerate 
   \item[$\tr_{(\Gamma, \omega)}$] $\omega$-trace
   \item[$\gl(V), \ \fsl(V)$] general and special linear Lie colour superalgebras
   \item[$\osp(V; \kappa)$] orthosymplectic Lie colour superalgebra
   \item[$\fg$] Lie $(\Gamma, \omega)$-algebra
   \item[$\CL(\fg)$] loop Lie $(\Gamma, \omega)$-algebra
   \item[$\wh\fg$] central extension of $\CL(\fg)$
     
\end{description}

\bigskip

\item[Section  4]{ \ }

\begin{description}[
  leftmargin=1.5cm,
  labelwidth=3.5cm,
  labelsep=0.3cm
]

    \item[$A$] reduced Cartan matrix, finite or affine type 
    \item[$\Phi$; $\Pi$] the set of roots and the set of simple roots
       
    \item[$\Xi$] sequence in $\Gamma$ 
     \item[$\U_{q, \Xi}(A)$] colour quantum group over $\C[q, q^{-1}]$
     \item[$\Delta, \varepsilon, S$] co-multiplication, co-unit, antipode
     \item[$\U^\pm_{q, \Xi}(A)$] quantum Borel subalgebras, \eqref{eq:subgroups}
      \item[$\U^0, \U^+_+, \U^-_-$] subalgebras of $\U^\pm_{q, \Xi}(A)$,  \eqref{eq:subgroups}
    \item[$\wt\U_{q, \Xi}(A)$] auxiliary colour quantum group over $\C[q, q^{-1}]$
     \item[$\wt\U^\pm_{q, \Xi}(A)$] auxiliary quantum Borel subalgebras, \eqref{eq:aux-subgp}
     \item[$\wt\U^0, \wt\U^+_+, \wt\U^-_-$] 	subalgebras of $\wt\U^\pm_{q, \Xi}(A)$, 
     																	\eqref{eq:aux-subgp}
    \item[$S_{i j}^\pm$] colour quantum Serre elements, \eqref{eq:def-S+}, \eqref{eq:def-S-}
    \item[$\CJ$] homogeneous Hopf ideal, Definition \ref{def:Serre-ideal}
    \item[$\CJ^+_+, \CJ^-_-$]  Definition \ref{def:Serre-ideal}
    \item[$\hat{E}_{; i}^{R, \pm}$, $\hat{E}_{; i}^{L, \pm}$] equation \eqref{eq:deriv-i}

\end{description}

\bigskip
\item[Section  5] {\ }

 \begin{description}[
  leftmargin=1.5cm,
  labelwidth=3.5cm,
  labelsep=0.3cm
]
    \item[\text{$\C[[\hbar]]$}]  power series ring with $\hbar$-adic topology
    \item[\text{$\C((\hbar))$}]  Laurent series ring with $\hbar$-adic topology
     \item[\text{lattice topology}] Definition \ref{def:latt}
    \item[$\U_\hbar(A; \Xi)_R$] topological colour quantum group over $\C[[\hbar]]$
    \item[$\U_\hbar(A; \Xi)$]    lattice topological colour quantum group over $\C((\hbar))$
    \item[$\wt\U^+_\hbar(A;\Xi)_R$] auxiliary quantised Borel subalgebra over $\C[[\hbar]]$ 
    \item[$\wt\U^+_\hbar(A;\Xi)$]  auxiliary quantised Borel subalgebra over $\C((\hbar))$
    \item[$\wt\U^0, \wt\U_+$]  subalgebras of $\wt\U^+_\hbar(A;\Xi)$
	\item[$(\wt\U_+)_\mu$] weight $\mu$ subspace
    \item[$\wt\U^+_\hbar(A;\Xi)^\vee$] lattice continuous dual of $\wt\U^+_\hbar(A;\Xi)$,  
    														a $(\Gamma, \omega)$-algebra
     \item[$\wt\U^+_\hbar(A;\Xi)'$]  $(\Gamma, \omega)$-subalgebra 
     													of $\wt\U^+_\hbar(A;\Xi)^\vee$, a Hopf $(\Gamma, \omega)$-algebra
     \item[$\langle \  , \ \rangle$]		Hopf pairing between $\wt\U^+_\hbar(A;\Xi)$ and 
     								$\wt\U^+_\hbar(A;\Xi)'$
    \item[$\CJ^\pm$] 	right and left radicals of Hopf pairing 
     \item[$S^+_{i j}\in  \wt\U_+$] colour quantum Serre elements
     
     \item[$\U^+_\hbar(A;\Xi)'$] $\wt\U^+_\hbar(A;\Xi)'/\CJ^-$
     \item[$\U^-_\hbar(A;\Xi)$]  $\U^-_\hbar(A;\Xi)=\U^+_\hbar(A;\Xi)'$
     \item[$\SD_\hbar(A; \Xi)$] quantum double of $\wt\U^+_\hbar(A;\Xi)$
     \item[$\SR$] universal $R$-matrix of $\SD_\hbar(A; \Xi)$
\end{description}

\bigskip

\item[Appendix A]{ \ }

\begin{description}[
  leftmargin=1.5cm,
  labelwidth=3.5cm,
  labelsep=0.3cm
]

	\item[$H^0$]  finite dual Hopf $(\Gamma, \omega)$-algebra of $H$
	\item[$\rhd$, \  $\lhd$]  left,  right $H$-actions on $H^0$, \eqref{eq:H-on-H0}
	\item[$\brhd$, \ $\blhd$]  left,   right $H^0$-actions on $H$, \eqref{eq:H0-on-H}
	\item[$\psi$]                twist isomorphism $H^0\ot H \lra H\ot H^0$, \eqref{eq:tmult-1}
    \item[\text{lattice equipped}] Definition \ref{def:latt-vect}
     \item[\text{lattice continuous}] Definition \ref{def:latt-vect}
    \item[\text{lattice topology}] Definition \ref{def:latt}
    \item[\text{$\wh\ot_{\C[[\hbar]]}$}]   	completed tensor product over $\C[[\hbar]]$
    														 	in $\hbar$-adic topology
    \item[\text{$\wh\ot_{\C((\hbar))}$}] 	completed tensor product over $\C((\hbar))$
    															in lattice topology
\end{description}
\end{description}

\bigskip

\noindent{\bf Acknowledgements}. 
We thank Naruhiko Aizawa, Zhanna Kuznetsova, Francesco Toppan  and Joris Van der Jeugt for sharing 
their insights on Lie colour algebras.
The main results of this work were reported at 
the 2025 Autumn Conference on Lie Theory,  Nov. 14–16, 2025, 
East China Normal University, Shanghai. We thank the organisers for the invitation 
and the audience for comments.


\bigskip
\noindent{\bf Declaration}.
The authors declare that they have no known competing financial interests or personal relationships 
that could have appeared to influence the work reported in this paper.



\end{document}